\documentclass[10pt]{amsart}

\usepackage{hyperref}
\hypersetup{nesting=true,debug=true,naturalnames=true}
\usepackage{graphicx,amssymb,upref}
\usepackage[usenames,dvipsnames]{pstricks}
\usepackage{epsfig}
\usepackage{color}
\usepackage[latin1]{inputenc}
\usepackage[T1]{fontenc}
\usepackage{amsmath, amsthm}
\usepackage[english]{babel}
\usepackage{amssymb}
\usepackage{latexsym}
\usepackage{graphicx,amssymb,upref}
\usepackage[all]{xy}
\hypersetup{nesting=true,debug=true,naturalnames=true}
\usepackage{tikz}
\usetikzlibrary{matrix}

\date{\today}

\newtheorem{teo}{Theorem}[section]
\newtheorem{prop}[teo]{Proposition}
\newtheorem{cor}[teo]{Corollary}

\newtheorem{defin}[teo]{Definition}
\newtheorem{rem}[teo]{Remark}
\newtheorem{ej}[teo]{Example}

\newtheorem{apart}[teo]{ }

\newcommand{\ot}{\otimes}
\newcommand{\co}{\circ}

\title[Cohomological Obstructions and Weak Crossed Products over  Weak Hopf Algebras]
{Cohomological Obstructions and weak crossed products over  Weak Hopf Algebras}

\author[Ram\'on Gonz\'alez Rodr\'{\i}guez]{Ram\'on Gonz\'alez Rodr\'{\i}guez}
\author[Ana Bel\'en Rodr\'{\i}guez Raposo]{Ana Bel\'en Rodr\'{\i}guez Raposo}

\address[Ram\'on Gonz\'{a}lez Rodr\'{\i}guez]{Universidade de Vigo. Departamento de Matem\'{a}tica Aplicada II,
 E.E. Telecomunicación, Campus Universitario Lagoas-Marcosende, E-36280
Vigo, Spain.} \email[R. Gonz\'{a}lez]{rgon@dma.uvigo.es}

\address[Ana Bel\'en Rodr\'{\i}guez Raposo]{Universidade de
Santiago de Compostela. Departamento de Did\'acticas Aplicadas,   Facultade de Ciencias da Educación, Campus Norte, E-15771 Santiago de Compostela, Spain.}
\email[M.B. Rodr\'{\i}guez]{anabelen.rodriguez.raposo@usc.es}

\begin{document}

\begin{abstract} Let $H$ be a cocommutative weak Hopf algebra and let $(B, \varphi_{B})$ a weak left $H$-module algebra. In this paper, for a twisted convolution invertible morphism $\sigma:H\ot H\rightarrow B$ we  define  its  obstruction $\theta_{\sigma}$ as a degree three Sweedler 3-cocycle with values in the center of $B$.  We obtain that the class of this obstruction vanish in third Sweedler cohomology group $\mathcal{H}^3_{\varphi_{Z(B)}}(H, Z(B))$ if, and only if, there exists a twisted convolution invertible  2-cocycle $\alpha:H\ot H\to B$ such that $H\ot B$ can be endowed with a weak crossed product structure with $\alpha$ keeping a cohomological-like relation with $\sigma$. Then, as a consequence,  the class of the obstruction of $\sigma$ vanish if, and only if, there exists a cleft extension of $B$ by $H$.
\end{abstract}

\maketitle



\section*{Introduction}

Crossed products of a Hopf algebra by an algebra have been widely studied 
in relation to extensions of algebras, generalizing classical results of semidirect products and extensions of groups, together with a generalization of group cohomology to the Hopf algebra setting. In \cite{Moss} Sweedler defines the so-called Sweedler's cohomology for a cocommutative Hopf algebra $H$ and a commutative $H$-module algebra $B$. In this paper he also shows that any cleft $H$-extension of algebras $B\hookrightarrow A$ (that is, roughly speaking, a split extension) can be realized as a crossed 
product given in terms of the action of $H$ on $B$ and a 2-cocycle $\sigma:H\ot H\to B$. Moreover, cleft extensions of $B$ are classified by the second cohomology group $\mathcal{H}^2(H,B)$. Several generalizations of these results were carried out by Doi and Takeuchi \cite{doi3}, Blattner, Cohen and Montgomery \cite{bcm} and Blattner and Montgomery \cite{blat-susan} by dropping out the conditions of cocommutativity and commutativity, 
and the associativity of the action $\varphi_B:H\ot B\to B$ and thus, the 
use of Sweedler's cohomology. However some of its formalism is preserved: 
for an arbitraty Hopf algebra $H$ and an arbitrary algebra $B$, a crossed 
product is given in terms of a measuring $\varphi_B:H\ot B\to B$ and a formal 2-cocycle $\sigma:H\ot H\to B$ that must also satisfy the twisted condition needed to substitute the associativity of $\varphi_B$. Moreover, two such crossed products are equivalent if the cocycles satisfy a cohomological-like equivalence. This last result was interpreted in an actual cohomological setting by Doi in \cite{doi}, where he shows that cleft extensions of an algebra $B$ by a cocommutative Hopf algebra $H$ with the same action are classified by $\mathcal{H}^2(H,Z(B))$, where $Z(B)$ denotes the center of $B$. All these results can be interpreted in a symmetric monoidal category with base object $K$ (see for example \cite{nm1} and \cite{me1} for cleft extensions in a monoidal setting). 

The next objective became to decide when an algebra $B$ admits a cleft extension by $H$. Following some classical results of obstructions to extensions of groups (see, for example, \cite{MacLane-hom}), Schauenburg finds 
in \cite{SCH} a relation between the third Sweedler's cohomology group $\mathcal{H}^3(H, Z(B))$ and cleft extensions, where $Z(B)$ denotes the center of $B$. For a measuring $\varphi_B$ and a twisted morphism $\sigma$, he generalizes the notion of obstruction as Sweedler  three cocycle $\theta_{\sigma}$ on $H$ with values on the center of $B$ and shows that the class $[\theta_{\sigma}]\in \mathcal{H}^3(H, Z(B))$ vanish if, and only if, $\varphi_B$ and $\sigma$ give rise to a crossed product on $H\ot B$ and, at last, to a cleft extension. 

With the apparition of weak Hopf algebras as generalizations of groupoid algebras  (see \cite{GK}) all the theory of cleft extensions, Sweedler's cohomology  and crossed products needed a deep review. Recall that the main point of a  algebra-coalgebra $H$ to be weak is that its unit does not need to be coassociative, nor its counit associative. These apparently innocent generalizations conceptually imply the existence of two objects, different from the base object $K$ in the ground monoidal category when $H$ is actually weak, that somehow will play the role of $K$. From a practical 
point of view, this lack of (co)associativity of the (co)unit forces to a 
change in the definition of regular morphisms, and thus to a change in the tackling of cleft extensions, cohomological interpretations of crossed products and a rethinking of cohomology and crossed products themselves. For the cocommutative case these problems were successfully solved in \cite{nmra5} and  \cite{nmra6}, where the authors explore the meaning of cleft extension and weak crossed product for a cocommutative weak Hopf algebra $H$ weakly acting on an algebra $B$, and define Sweedler's cohomology in weak contexts. In order to achieve these objectives, they consider the unit in $Reg(H,B)$ as $\varphi_B\circ (H\ot\eta_B)$ (and thus, regular morphisms depend on $\varphi_B$ and we denote the set by $Reg_{\varphi_B}(H,B)$), where $\varphi_B$ is the weak action of $H$ on $B$, and $\eta_B$ is the unit of $B$. Moreover, to study weak crossed products they consider a 
preunit instead of a unit, so they obtain an algebra as a subobject of $H\ot B$, whose product is given in terms of $\varphi_B$ and a twisted formal 2-cocycle $\sigma:H\ot H\to B$. In such terms, they are able to define 
a cohomology theory for a cocommutative weak Hopf algebra $H$ and a commutative $H$-module algebra $B$. Moreover, they identify cleft $H$-extensions of a weak $H$-module algebra $B$ with products with convolution invertible 2-cocycle and classify them by $\mathcal{H}^2_{\varphi_Z(B)}(H, Z(B))$, this is, the second cohomology group. The relation of weak crossed products and cleft extensions for the non-cocommutative case was also studied in \cite{Gu2-Val} by  Guccione, Guccione and  Valqui. 

So once we have the proper concepts of cleft extensions, weak crossed products and Sweedler's cohomology for the weak setting, we just need to find out the role of obstructions in relation to cleft  extensions and their 
cohomological meaning, and these are the main objectives of the present paper. In order to attain such objectives, we first make a wide review of weak crossed products, and we find that we just need a measuring $\varphi_B:H\ot B\to B$ together with a twisted morphism $\sigma:H\ot H\to B$ that does not need to be convolution invertible but a formal 2-cocycle to define a weak crossed product on $H\ot B$. Moreover we obtain necessary and 
sufficient conditions for weak crossed products to be equivalent that, in 
particular, are given in terms of morphisms in $Reg_{\varphi_B}(H, B)$. We finally use these results in the particular case of a cocommutative weak Hopf algebra $H$ and a weak $H$-module algebra $(B, \varphi_B)$. We consider a twisted convolution invertible morphism $\sigma:H\ot H\to B$ and define its cohomological obstruction $\theta_{\sigma}$ through the center 
of $B$. We obtain that this obstruction vanish in $\mathcal{H}^3_{\varphi_{Z(B)}}(H, Z(B))$ if, and only if, there exists a twisted convolution invertible 2-cocycle $\alpha:H\ot H\to B$ such that $H\ot B$ can be endowed 
with a weak crossed product structure with $\alpha$ keeping a cohomological-like relation with $\sigma$. This result means, in terms of cleft extensions, that if $(B, \varphi_B)$ is a weak $H$-module algebra with $\sigma:H\ot H\to B$ twisted and convolution invertible then its obstruction vanish if, and only if, there exists a cleft extension of $B$ by $H$.

Throughout this paper $\sf C$ denotes a strict symmetric monoidal category with tensor product $\ot$, unit object $K$ and symmetry isomorphism $c$. There is no loss of generality in assuming that $\sf C$ is strict because every monoidal category is monoidally equivalent to a strict one. Then, we may work as if the constrains were all identities. We also assume that in $\sf C$ every idempotent morphism splits, i.e., for any morphism $q:M\rightarrow M$ such that $q\co q=q$ there exists an object $N$, called the image of $q$, and morphisms $i:N\rightarrow M$, $p:M\rightarrow N$ such that $q=i\co p$ and $p\co i=id_N$ ($id_{N}$ denotes the identity morphism for $N$). The morphisms $p$ and $i$ will be called a factorization of $q$. Note that $N$, $p$ and $i$ are unique up to isomorphism. The categories satisfying this property constitute a broad class that includes, among others, the categories with epi-monic decomposition for morphisms and categories with (co)equalizers.  Finally, given
objects $M$, $N$, $P$ and a morphism $f:N\rightarrow P$, we write
$M\ot f$ for $id_{M}\ot f$ and $f\ot M$ for $f\ot id_{M}$.

An algebra in $\sf C$ is a triple $A=(A, \eta_{A},
\mu_{A})$, where $A$ is an object in $\sf C$ and
 $\eta_{A}:K\rightarrow A$ (unit), $\mu_{A}:A\otimes A
\rightarrow A$ (product) are morphisms in $\sf C$ such that
$\mu_{A}\circ (A\otimes \eta_{A})=id_{A}=\mu_{A}\circ
(\eta_{A}\otimes A)$, $\mu_{A}\circ (A\otimes \mu_{A})=\mu_{A}\circ
(\mu_{A}\otimes A)$. We say that an algebra $A$ is commutative if $\mu_{A}=\mu_{A}\co c_{A,A}$.

Given two algebras $A= (A, \eta_{A}, \mu_{A})$
and $B=(B, \eta_{B}, \mu_{B})$, a morphism $f:A\rightarrow B$ in $\sf C$  is an algebra
morphism if $\mu_{B}\circ (f\otimes f)=f\circ \mu_{A}$ and  $ f\circ
\eta_{A}= \eta_{B}$. 

Also, if  $A$, $B$ are algebras in $\sf C$, the object
$A\otimes B$ is  an algebra in $\sf C$, where
$\eta_{A\otimes B}=\eta_{A}\otimes \eta_{B}$ and $\mu_{A\otimes
B}=(\mu_{A}\otimes \mu_{B})\circ (A\otimes c_{B,A}\otimes B).$ Note that, if $A$ and $B$ are commutative algebras, so is $A\ot B$.

A coalgebra  in $\sf C$ is a triple ${D} = (D,
\varepsilon_{D}, \delta_{D})$, where $D$ is an object in $\sf C$ and $\varepsilon_{D}: D\rightarrow K$ (counit),
$\delta_{D}:D\rightarrow D\otimes D$ (coproduct) are morphisms in
$\sf C$ such that $(\varepsilon_{D}\otimes D)\circ
\delta_{D}= id_{D}=(D\otimes \varepsilon_{D})\circ \delta_{D}$,
$(\delta_{D}\otimes D)\circ \delta_{D}=
 (D\otimes \delta_{D})\circ \delta_{D}.$ We say that a coalgebra $D$ is cocommutative if $\delta_{D}=c_{D,D}\co \delta_{D}$.

 If ${D} = (D, \varepsilon_{D},
 \delta_{D})$ and
$E = (E, \varepsilon_{E}, \delta_{E})$ are coalgebras,
a morphism  $f:D\rightarrow E$ in $\sf C$  is a coalgebra morphism if $(f\otimes f)\circ
\delta_{D} =\delta_{E}\circ f$ and $\varepsilon_{E}\circ f
=\varepsilon_{D}.$ 

If  $D$, $E$ are coalgebras in $\sf C$, the tensor product $D\otimes E$ is a
coalgebra in $\sf C$, where $\varepsilon_{D\otimes
E}=\varepsilon_{D}\otimes \varepsilon_{E}$ and $\delta_{D\otimes
E}=(D\otimes c_{D,E}\otimes E)\circ( \delta_{D}\otimes \delta_{E}).$ Note that, if $D$ and $E$ are cocommutative coalgebras, so is $D\ot E$.

Finally, if $A$ is an algebra, $C$ is a coalgebra and $f: C\rightarrow
A$, $g:C\rightarrow A$ are morphisms in $\sf C$, we define the
convolution product by $f\ast g=\mu_{A}\circ
(f\otimes g)\circ \delta_{C}$.

\section{Generalities on measurings and crossed products in a weak setting}

In  this section we resume some basic facts
about the general theory of weak crossed products in  ${\sf C}$,
introduced in \cite{mra-preunit}, particularized for measurings over a weak Hopf
algebra $H$. Firstly, we recall the notion of weak Hopf algebra, introduced in \cite{GK}, and sumarize some basic properties of these algebraic objects in a  monoidal setting. 

\begin{defin}
\label{wha} {\rm A weak bialgebra $H$  is an object in {\sf C} with an algebra structure $(H, \eta_{H},\mu_{H})$ and a coalgebra structure $(H, \varepsilon_{H},\delta_{H})$ such that the following axioms hold:
\begin{itemize}
\item[(a1)] $\delta_{H}\circ \mu_{H}=(\mu_{H}\otimes \mu_{H})\circ
\delta_{H\otimes H},$
\item[(a2)]$\varepsilon_{H}\circ \mu_{H}\circ
(\mu_{H}\otimes H)=(\varepsilon_{H}\otimes \varepsilon_{H})\circ
(\mu_{H}\otimes \mu_{H})\circ (H\otimes \delta_{H}\otimes H)$
\item[ ]$=(\varepsilon_{H}\otimes \varepsilon_{H})\circ
(\mu_{H}\otimes \mu_{H})\circ (H\otimes
(c_{H,H}\circ\delta_{H})\otimes H),$
\item[(a3)]$(\delta_{H}\otimes H)\circ \delta_{H}\circ
\eta_{H}=(H\otimes \mu_{H}\otimes H)\circ (\delta_{H}\otimes
\delta_{H})\circ (\eta_{H}\otimes \eta_{H})$ 
\item[ ]$=(H\otimes(\mu_{H}\circ c_{H,H})\otimes H)\circ (\delta_{H}\otimes
\delta_{H})\circ (\eta_{H}\otimes \eta_{H}).$
\end{itemize}

Moreover, if there exists a morphism $\lambda_{H}:H\rightarrow H$
in {\sf C} (called the antipode of $H$) satisfying
\begin{itemize}
\item[(a4)] $id_{H}\ast \lambda_{H}=((\varepsilon_{H}\circ
\mu_{H})\otimes H)\circ (H\otimes c_{H,H})\circ ((\delta_{H}\circ
\eta_{H})\otimes H),$
\item[(a5)] $\lambda_{H}\ast
id_{H}=(H\otimes(\varepsilon_{H}\circ \mu_{H}))\circ (c_{H,H}\otimes
H)\circ (H\otimes (\delta_{H}\circ \eta_{H})),$
\item[(a6)]$\lambda_{H}\ast id_{H}\ast
\lambda_{H}=\lambda_{H},$
\end{itemize}
we will say that the weak bialgebra is a weak Hopf algebra.

We say that $H$ is commutative, if it is commutative as algebra and we say that $H$ is cocommuative if it is cocommutative as  coalgebra.
}
\end{defin}

\begin{apart}
{\rm 
Let $H$ be a weak biagebra. For $n\geq 1$, we denote by $H^{\ot n}$ the $n$-fold tensor power $H\ot \cdots \ot H$. By $H^{\ot 0}$ we denote the unit object of ${\sf C}$, i.e., $H^{\ot 0}=K$. If $n\geq 2$,  $m_{H}^{\ot n}$ denotes the morphism $m_{H}^{\ot n}:H^{\ot n}\rightarrow H$ defined  by $m_{H}^{\ot 2}=\mu_{H}$ and by $m_{H}^{\ot 3}=m_{H}^{\ot 2}\co (H\ot \mu_{H}),\cdots, m_{H}^{\ot k}=m_{H}^{\ot (k-1)}\co (H^{k-2}\ot \mu_{H})$  for $k> 2$.  On the other hand,  with $\delta_{H^{\ot n}}$ we denote the coproduct defined in
the coalgebra $H^{\ot n}$. Then by the coassociativity of $\delta_{H}$ and the naturality of $c$, for $k=1,\cdots, n-1,$ 
$$ \delta_{H^{\ot n}}=\delta_{H^{\ot (n-k)}\ot H^{\ot k}}
$$
holds. By \cite[Proposition 2.10]{nmra5} we have that 
$$
\delta_{H}\co m_{H}^{\ot n}=(m_{H}^{\ot n}\ot m_{H}^{\ot n})\co \delta_{H^{\ot n}}.
$$

Finally, note that, if $H$ is cocommutative, then so is $H^{\ot n}$.
}
\end{apart}

\begin{apart}
{\rm 
For any weak bialgebra, if we define the morphisms $\Pi_{H}^{L}$ (target), $\Pi_{H}^{R}$
(source), $\overline{\Pi}_{H}^{L}$ and $\overline{\Pi}_{H}^{R}$ by
$$\Pi_{H}^{L}=((\varepsilon_{H}\circ \mu_{H})\otimes
H)\circ (H\otimes c_{H,H})\circ ((\delta_{H}\circ \eta_{H})\otimes
H),$$
$$\Pi_{H}^{R}=(H\otimes(\varepsilon_{H}\circ
\mu_{H}))\circ (c_{H,H}\otimes H)\circ (H\otimes (\delta_{H}\circ
\eta_{H})),$$
$$\overline{\Pi}_{H}^{L}=(H\otimes
(\varepsilon_{H}\circ \mu_{H}))\circ ((\delta_{H}\circ
\eta_{H})\otimes H),$$
$$\overline{\Pi}_{H}^{R}=((\varepsilon_{H}\circ \mu_{H})\otimes
H)\circ(H\otimes (\delta_{H}\circ \eta_{H})),$$
it is straightforward to show  that they are
idempotent and   the equalities
\begin{equation}
\label{ts1} \Pi_{H}^{L}\circ
\overline{\Pi}_{H}^{L}=\Pi_{H}^{L},\;\;\;\; \Pi_{H}^{L}\circ
\overline{\Pi}_{H}^{R}=\overline{\Pi}_{H}^{R},\;\;\;\;
\Pi_{H}^{R}\circ
\overline{\Pi}_{H}^{L}=\overline{\Pi}_{H}^{L},\;\;\;\;
\Pi_{H}^{R}\circ \overline{\Pi}_{H}^{R}=\Pi_{H}^{R},
\end{equation}
\begin{equation}
\label{ts2}
\overline{\Pi}_{H}^{L}\circ
\Pi_{H}^{L}=\overline{\Pi}_{H}^{L},\;\;\;\;
\overline{\Pi}_{H}^{L}\circ \Pi_{H}^{R}=\Pi_{H}^{R},\;\;\;\;
\overline{\Pi}_{H}^{R}\circ \Pi_{H}^{L}=\Pi_{H}^{L},\;\;\;\;
\overline{\Pi}_{H}^{R}\circ \Pi_{H}^{R}=\overline{\Pi}_{H}^{R}, 
\end{equation}
hold.

On the other hand, denote by $H_L$ the image of the target morphism $\Pi_{H}^{L}$ and let $p_{H}^{L}:H\rightarrow H_L$, $i_{H}^{L}:H_L\rightarrow H$ be the morphisms such that $i_{H}^{L}\co p_{H}^{L}=\Pi_{H}^{L}$ and $p_{H}^{L}\co i_{H}^{L}=id_{H_L}$. Then $$(H_L, \eta_{H_L}=p_{H}^{L}\co \eta_H, \mu_{H_L}=p_{H}^{L}\co \mu_H\co (i_{H}^{L}\ot i_{H}^{L}))$$ is an algebra and $$(H_L, \varepsilon_{H_L}=\varepsilon_{H}\co i_{H}^{L}, \delta_{H_L}=(p_{H}^{L}\ot p_{H}^{L})\co \delta_H\co i_{H}^{L})$$ is a coalgebra. The morphisms $\eta_{H_L}$, $\mu_{H_L}$, $\varepsilon_{H_L}$ and $\delta_{H_L}$ are the unique morphisms satisfying
\begin{equation}
\label{res1}
i_{H}^{L}\co \eta_{H_L}=\eta_{H}, \;\;\; i_{H}^{L}\co \mu_{H_L}=\mu_{H}\co (i_{H}^{L}\ot i_{H}^{L}),
\end{equation}
\begin{equation}
\label{res2}
\varepsilon_{H_L}\co p_{H}^{L}=\varepsilon_{H}, \;\;\; \delta_{H_L}\co p_{H}^{L}=(p_{H}^{L}\ot p_{H}^{L})\co \delta_{H}
\end{equation}
respectively.

For the morphisms target and source we have the following identities:
\begin{equation}
\label{mupi}\Pi^{L}_{H}\co \mu_{H}\co  (H\ot \Pi^{L}_{H})=\Pi^{L}_{H}\co \mu_{H},\;\;\;\Pi^{R}_{H}\co \mu_{H}\co ( \Pi^{R}_{H}\ot
H)=\Pi^{R}_{H}\co \mu_{H},
\end{equation}
\begin{equation}
\label{mupibar} \overline{\Pi}^{L}_{H}\co \mu_{H}\co (H\ot \overline{\Pi}^{L}_{H})=\overline{\Pi}^{L}_{H}\co \mu_{H},\;\;\;\overline{\Pi}^{R}_{H}\co \mu_{H}\co (\overline{\Pi}^{R}_{H}\ot H)=\overline{\Pi}^{R}_{H}\co \mu_{H},
\end{equation}
\begin{equation}
\label{d-pi} (H\ot \Pi^{L}_{H})\circ \delta_{H}\circ
\Pi^{L}_{H}=\delta_{H}\circ \Pi^{L}_{H},\;\;\;( \Pi^{R}_{H}\ot
H)\circ \delta_{H}\circ \Pi^{R}_{H}=\delta_{H}\circ \Pi^{R}_{H},
\end{equation}
\begin{equation}
\label{d-pibar} (H\ot \overline{\Pi}^{R}_{H})\circ \delta_{H}\circ
\overline{\Pi}^{R}_{H}=\delta_{H}\circ \overline{\Pi}^{R}_{H},\;\;\;(\overline{\Pi}^{L}_{H}\ot
H)\circ \delta_{H}\circ \overline{\Pi}^{L}_{H}=\delta_{H}\circ \overline{\Pi}^{L}_{H},
\end{equation}
\begin{equation}
\label{d-pi1} \mu_{H}\co (H\ot \Pi_{H}^{L})=((\varepsilon_{H}\co
\mu_{H})\ot H)\co (H\ot c_{H,H})\co (\delta_{H}\ot H),
\end{equation}
\begin{equation}
\label{d-pi2} (H\ot \Pi^{L}_{H})\co \delta_{H}= (\mu_{H}\ot
H)\co (H\ot c_{H,H})\co ((\delta_{H}\co \eta_{H})\ot H),
\end{equation}
\begin{equation}
\label{d-pi3} \mu_{H}\co (\Pi_{H}^{R} \ot H)=(H\ot
(\varepsilon_{H}\co \mu_{H}))\co (c_{H,H}\ot H)\co (H\ot \delta_{H})
\end{equation}
\begin{equation}
\label{d-pi4} (\Pi_{H}^{R} \ot H)\co \delta_{H}=(H\ot
\mu_{H})\co (c_{H,H}\ot H)\co (H\ot (\delta_{H}\co \eta_{H}))
\end{equation}
\begin{equation}
\label{d-pi11} \mu_{H}\co (\overline{\Pi}_{H}^{R}\ot
H)=((\varepsilon_{H}\co \mu_{H})\ot H)\co (H\ot \delta_{H}),
\end{equation}
\begin{equation}
\label{d-pi21} \mu_{H}\co (H\ot \overline{\Pi}_{H}^{L})=(H\ot
(\varepsilon_{H}\co \mu_{H}))\co (\delta_{H}\ot H),
\end{equation}
\begin{equation}
\label{d-pi31} (\overline{\Pi}_{H}^{L}\ot H)\co \delta_{H}=(H\ot
\mu_{H})\co ((\delta_{H}\co \eta_{H})\ot H),
\end{equation}
\begin{equation}
\label{d-pi41} (H\ot \overline{\Pi}_{H}^{R})\co
\delta_{H}=(\mu_{H}\ot H)\co (H\ot (\delta_{H}\co \eta_{H})),
\end{equation}
\begin{equation}
\label{edpi} 
\delta_H\co \eta_H=(\Pi_{H}^{R}\ot H)\co \delta_H\co \eta_H=(H\ot\Pi_{H}^{L})\co \delta_H\co \eta_H=
(H\ot \overline{\Pi}_{H}^{R})\co \delta_H\co \eta_H
\end{equation}
$$=(\overline{\Pi}_{H}^{L}\ot H)\co \delta_H\co \eta_H,$$
\begin{equation}
\label{empi} 
\varepsilon_H\co \mu_H=\varepsilon_H\co \mu_H\co(\Pi_{H}^{R}\ot H)=\varepsilon_H\co \mu_H\co (H\ot\Pi_{H}^{L})=
\varepsilon_H\co \mu_H\co (\overline{\Pi}_{H}^{R}\ot H)
\end{equation}
$$=\varepsilon_H\co \mu_H\co (H\ot \overline{\Pi}_{H}^{L}).$$

If $H$ is a weak Hopf algebra in {\sf C}, the antipode
$\lambda_{H}$ is unique, antimultiplicative, anticomultiplicative
and leaves the unit  and the counit invariant, i.e.:
\begin{equation}
\label{anti} \lambda_{H}\circ \mu_{H}=\mu_{H}\circ
(\lambda_{H}\otimes \lambda_{H})\circ
c_{H,H},\;\;\;\;\delta_{H}\circ \lambda_{H}=c_{H,H}\circ
(\lambda_{H}\otimes \lambda_{H})\circ \delta_{H},
\end{equation}
\begin{equation}
\label{iuc} \lambda_{H}\circ
\eta_{H}=\eta_{H},\;\;\;\;\varepsilon_{H}\circ
\lambda_{H}=\varepsilon_{H}.
\end{equation}

Also, it is straightforward to show that $\Pi_{H}^{L}$, $\Pi_{H}^{R}$ satisfy the equalities
\begin{equation}
\label{id} \Pi_{H}^{L}=id_{H}\ast
\lambda_{H},\;\;\Pi_{H}^{R}=\lambda_{H}\ast id_{H},\;\;\Pi_{H}^{L}\ast id_H=id_H=id_{H}\ast \Pi_{H}^{R},
\end{equation}
$$\Pi_{H}^{R}\ast \lambda_H=\lambda_H=\lambda_{H}\ast \Pi_{H}^{L}, $$
\begin{equation}
\label{id1} \Pi_{H}^{L}\ast \Pi_{H}^{L}=\Pi_{H}^{L}, \;\;\;\; \Pi_{H}^{R}\ast \Pi_{H}^{R}=\Pi_{H}^{R}
\end{equation}
and 
\begin{equation}
\label{pi-lamdda} 
\Pi_{H}^{L}=\lambda_H\co\overline{\Pi}_{H}^{L}=\overline{\Pi}_{H}^{R}\co \lambda_H,\;\;\;\; 
\Pi_{H}^{R}=\overline{\Pi}_{H}^{L}\co \lambda_H =\lambda_H\co\overline{\Pi}_{H}^{R}.
\end{equation}

Finally we also have
\begin{equation}\label{eq-id-bpi}
	\mu_H\circ c_{H,H}\circ (H\ot \bar{\Pi}_H^L)\circ \delta_H = id_H = \mu_H\circ c_{H,H}\circ ( \bar{\Pi}_H^R\ot H)\circ \delta_H.
\end{equation}

}

\end{apart}

Now we recall the notions of measuring, left weak $H$-module algebra,  and left $H$-module algebra.

\begin{defin}
\label{def}
{\rm Let $H$ be a weak Hopf algebra and let $B$ be an algebra. We say that the morphism $\varphi_{B}:H\ot B\rightarrow B$ is a measuring if 
\begin{itemize}
\item[(b1)] $\varphi_{B}\co (H\ot \mu_{B})=\mu_{B}\co (\varphi_{B}\ot \varphi_{B})\co (H\ot c_{H,B}\ot B)\co
(\delta_{H}\ot B\ot B).$
\end{itemize}
Set $u_1^{\varphi_B} = \varphi_B\circ (H\ot \eta_B).$ If $\varphi_{B}$ is a measuring satisfying 
\begin{itemize}
	\item[(b2)] $\varphi_{B}\co (\eta_{H}\ot B)=id_{B}$,
	\item[(b3)] $u_{1}^{\varphi_{B}}\co \mu_{H}=\varphi_{B}\co (H\ot
	u_{1}^{\varphi_{B}}),$
\end{itemize}
we will say that $(B,\varphi_{B})$ is a left weak $H$-module algebra.
If we replace (b3) by
$$\varphi_{B}\co (\mu_{H}\ot B)=\varphi_{B}\co (H\ot
	\varphi_{B})$$
we will say that $(B, \varphi_{B})$ is a left $H$-module algebra.
}
\end{defin}

If $(B,\varphi_{B})$ is a left weak  $H$-module algebra the following equivalent conditions are satisfied: 
\begin{equation}
	\label{B1}
	\varphi_{B}\co (\Pi^{L}_{H}\ot B)=\mu_{B}\co (u_{1}^{\varphi_{B}}\ot B),
\end{equation}
\begin{equation}
	\label{B2}
	\varphi_{B}\co (\overline{\Pi}^{L}_{H}\ot B)=\mu_{B}\co c_{B,B}\co
	(u_{1}^{\varphi_{B}}\ot B),
\end{equation}
\begin{equation}
	\label{B3}
	u_{1}^{\varphi_{B}}\co \Pi^{L}_{H}=u_{1}^{\varphi_{B}},
\end{equation}
\begin{equation}
	\label{B4}
	u_{1}^{\varphi_{B}}\co \overline{\Pi}^{L}_{H}=
	u_{1}^{\varphi_{B}},
\end{equation}
\begin{equation}
	\label{B5}
	u_{1}^{\varphi_{B}}\co \mu_{H}=u_{1}^{\varphi_{B}}\co \mu_{H}\co (H\ot \overline{\Pi}^{L}_{H}),
\end{equation}
\begin{equation}
	\label{B6}
	u_{1}^{\varphi_{B}}\co \mu_{H}=u_{1}^{\varphi_{B}}\co \mu_{H}\co (H\ot \Pi^{L}_{H}).
\end{equation}

Let $\varphi_{B}$ be a measuring and $n\geq
1$. With  $\varphi^{\ot n}_{B}$ we will denote the morphism
$$\varphi^{\ot n}_{B}:H^{\ot n}\ot B\rightarrow B$$
defined as $\varphi^{\ot 1}_{B}=\varphi_{B}$ and
$\varphi^{\ot n}_{B}=\varphi_{B}\co (H\ot \varphi^{\ot (n-1)}_{B})$.  

Note that, by (b1) of  Definition \ref{def}, proceeding by induction, it is easy to show that 
\begin{equation}
	\label{vp-n}
	\varphi^{\ot n}_{B}\co (H^{\ot n}\ot \mu_{B})=\mu_{B}\co
	(\varphi^{\ot n}_{B}\ot \varphi^{\ot n}_{B})\co (H^{\ot n}\ot c_{H^{\ot n},B}\ot B)\co
	(\delta_{H^{\ot n}}\ot B\ot B)
\end{equation}
holds.

For $n\geq 2$ we define the morphism 
\begin{equation}\label{def-un}
	u_{n}^{\varphi_{B}} = \varphi_{B}\co (m_{H}^{\ot n}\ot \eta_{B})=\varphi_{B}^{\ot (n-1)}\co (H^{\ot (n-1)}\ot u_{1}^{\varphi_{B}}).
\end{equation} 
Note that  
\begin{equation}
	\label{u}
	u_{n}^{\varphi_{B}}=u_{1}^{\varphi_{B}}\co m_{H}^{\ot n}
\end{equation}
and, by \cite[Proposition 2.11]{nmra5}, we have 
\begin{equation}
	\label{idem-n} 
	u_{n}^{\varphi_{B}}\ast u_{n}^{\varphi_{B}}=u_{n}^{\varphi_{B}}.
\end{equation}

\begin{defin}
{\rm  Let $H$ be a weak Hopf algebra and let $B$ be an algebra. For a measuring $\varphi_{B}$ and any morphism $\sigma:H\ot H\rightarrow B$, we define the morphisms
$$P_{\varphi_{B}}:H\ot B\rightarrow B\ot H, \;\;\;F_{\sigma}:H\ot H\rightarrow B\ot
H, \;\;\;G_{\sigma}:H\ot H\rightarrow H\ot
B, $$ by
\begin{equation}
\label{psiBH} P_{\varphi_{B}}=(\varphi_{B}\ot H)\co (H\ot c_{H,B})\co
(\delta_{H}\ot B),
\end{equation}
\begin{equation}
\label{sigmaBH} F_{\sigma}=(\sigma\ot \mu_{H})\co \delta_{H^{\ot 2}}
\end{equation}
and 
\begin{equation}
\label{sigmaHB} G_{\sigma}=(\mu_{H}\ot \sigma)\co \delta_{H^{\ot 2}}.
\end{equation}
}
\end{defin}

By \cite[Proposition 3.3]{nmra5} and some easy computations we have the following result.

\begin{prop}
\label{psi-prop} Let $H$ be a weak Hopf algebra and
let $\varphi_{B}:H\ot B\rightarrow B$ be a measuring. The morphism
$P_{\varphi_{B}}$ defined in (\ref{psiBH}) satisfies
\begin{equation}\label{wmeas-wcp}
(\mu_B\ot H)\co (B\ot P_{\varphi_{B}})\co (P_{\varphi_{B}}\ot B) =
P_{\varphi_{B}}\co (H\ot \mu_B).
\end{equation} 
The morphisms $\nabla_{B\ot H}^{\varphi_{B}}:B\ot H\rightarrow B\ot H$ and $\nabla_{H\ot B}^{\varphi_{B}}:H\ot B\rightarrow H\ot B$ defined by 
$$
\nabla_{B\ot H}^{\varphi_{B}} = (\mu_B\ot H)\co (B\ot P_{\varphi_{B}})\co (B\ot H\ot
\eta_B) \
$$
and 
$$
\nabla_{H\ot B}^{\varphi_{B}} = (H\ot \mu_B)\co (((H\ot \varphi_{A})\co (\delta_{H}\ot \eta_{A}))\ot A)
$$
 are idempotent. Also, we have the following identities:
\begin{equation}
\label{eta-psi} P_{\varphi_{B}}\co (H\ot \eta_{B})= (u_{1}^{\varphi_{B}}\ot H)\co \delta_{H}, \;\; 
\end{equation}
\begin{equation}
\label{e-psi} (B\ot \varepsilon_{H})\co P_{\varphi_{B}} = \varphi_{B},
\end{equation}
\begin{equation}
\label{nabla-psi} \nabla_{B\ot H}^{\varphi_{B}}\co P_{\varphi_{B}} = P_{\varphi_{B}},
\end{equation}
\begin{equation}
\label{nabla-nabla} \nabla_{B\ot H}^{\varphi_{B}}= (\mu_{B}\ot H)\co (B\ot ((u_{1}^{\varphi_{B}}\ot H)\co \delta_{H})), 
\end{equation}
\begin{equation}
\label{nabla-nabla2} \nabla_{H\ot B}^{\varphi_{B}} = (H\ot \mu_B)\co (((H\ot u_{1}^{\varphi_{B}})\co \delta_{H})\ot B), 
\end{equation}
\begin{equation}
\label{nabla-eta}  \nabla_{B\ot H}^{\varphi_{B}}\co (\eta_{B}\ot H)=(u_{1}^{\varphi_{B}}\ot H)\co \delta_{H},
\end{equation}
\begin{equation}
\label{nabla-eta1}  \nabla_{H\ot B}^{\varphi_{B}}\co (H\ot\eta_{B})=(H\ot u_{1}^{\varphi_{B}})\co \delta_{H},
\end{equation}
\begin{equation}
\label{left-module} (\mu_{B}\ot H)\co (B\ot \nabla_{B\ot H}^{\varphi_{B}})=(B\ot \nabla_{B\ot H}^{\varphi_{B}})\co (\mu_{B}\ot H),
\end{equation}
\begin{equation}
\label{right-module} (H\ot \mu_{B})\co (\nabla_{H\ot B}^{\varphi_{B}}\ot B)=(\nabla_{H\ot B}^{\varphi_{B}}\ot B)\co (H\ot \mu_{B}),
\end{equation}
\begin{equation}
\label{nabla-varep} (B\ot \varepsilon_{H})\co \nabla_{B\otimes
H}^{\varphi_{B}}=\mu_{B}\co (B\ot u_{1}^{\varphi_{B}}),
\end{equation}
\begin{equation}
\label{nabla-varep1} (\varepsilon_{H}\ot B)\co \nabla_{H\otimes
B}^{\varphi_{B}}=\mu_{B}\co (u_{1}^{\varphi_{B}}\ot B),
\end{equation}
\begin{equation}
\label{nabla-delta} (B\ot \delta_{H})\co \nabla_{B\ot
H}^{\varphi_{B}}=(\nabla_{B\ot H}^{\varphi_{B}}\ot H)\co (B\ot \delta_{H}),
\end{equation}
\begin{equation}
\label{nabla-delta1} (\delta_{H}\ot B)\co \nabla_{H\ot
B}^{\varphi_{B}}=(H\ot \nabla_{H\ot B}^{\varphi_{B}})\co (\delta_{H}\ot B),
\end{equation}
\begin{equation}
\label{nabla-fi} \mu_{B}\co (u_{1}^{\varphi_{B}}\ot \varphi_{B})\co (\delta_{H}\ot
B)=\varphi_{B},
\end{equation}
\begin{equation}
\label{nabla-fiAH} (\mu_{B}\ot H)\co (u_{1}^{\varphi_{B}}\ot P_{\varphi_{B}})\co
(\delta_{H}\ot B)=P_{\varphi_{B}},
\end{equation}
\begin{equation}
\label{eta-psi-varep} (B\ot \varepsilon_{H})\co P_{\varphi_{B}}\co
(H\ot \eta_{B})=u_{1}^{\varphi_{B}},
\end{equation}
\begin{equation}
\label{epc} (\mu_{B}\ot H)\co (u_{1}^{\varphi_{B}}\ot c_{H,B})\co (\delta_{H}\ot B)=(\mu_{B}\ot H)\co (H\ot c_{B,B})\co ((P_{\varphi_{B}}\co (H\ot \eta_{B}))\ot B).
\end{equation}
\end{prop}

On the other hand, by a similar proof to the one used in \cite[Proposition 3.4]{nmra5} we have the following proposition.
\begin{prop}
\label{sigma-prop1} Let $H$ be a  weak Hopf algebra, let 
$\varphi_{B}:H\ot B\rightarrow B$ be a measuring and let $\sigma:H\ot
H\rightarrow B$  be a morphism. The morphisms $F_{\sigma}$ and $G_{\sigma}$ 
defined in  (\ref{sigmaBH}) and  (\ref{sigmaHB})  satisfy the 
identities:
\begin{equation}
\label{delta-sigmaHB} (B\ot \delta_{H})\co
F_{\sigma}=(F_{\sigma}\ot \mu_{H})\co \delta_{H^{\ot 2}}, 
\end{equation}
\begin{equation}
\label{delta-sigmaHB1} (\delta_{H}\ot B)\co
G_{\sigma}=(\mu_{H}\ot G_{\sigma})\co \delta_{H^{\ot 2}}.
\end{equation}
\end{prop}

Moreover, we also have the proposition:

\begin{prop}
\label{sigmaprop} Let $H$ be a  weak Hopf algebra, let 
$\varphi_{B}:H\ot B\rightarrow B$ be a measuring. The  equality 
\begin{equation}
\label{psiHB-1} 
\mu_{B}\co (B\ot u_{1}^{\varphi_{B}})\co P_{\varphi_{B}}=\varphi_{B}, 
\end{equation} 
holds.

Let $\sigma:H\ot H\rightarrow B$  be a morphism  and $u_2^{\varphi_B}$ be the morphism defined in (\ref{def-un}).  If $\sigma\ast u_{2}^{\varphi_{B}}=\sigma$, the equality 
\begin{equation}
\label{sigmaHB-1} 
\mu_{B}\co (B\ot u_{1}^{\varphi_{B}})\co F_{\sigma}=\sigma, 
\end{equation}
holds and, as a consequence, we have the following identities:
\begin{equation}
\label{sigmaHB-2} 
\nabla_{B\ot H}^{\varphi_{B}}\co F_{\sigma}=F_{\sigma}, 
\end{equation}
\begin{equation}
\label{sigmaHB-3} 
(B\ot \varepsilon_{H})\co F_{\sigma}=\sigma. 
\end{equation}

Moreover, if $\sigma:H\ot H\rightarrow B$  is a morphism satisfying (\ref{sigmaHB-2}) and (\ref{sigmaHB-3}), we have that 
$\sigma\ast u_{2}^{\varphi_{B}}=\sigma$.

\end{prop}

\begin{proof} 
Note that, (\ref{psiHB-1}) holds because 
$$\mu_{B}\co (B\ot u_{1}^{\varphi_{B}})\co P_{\varphi_{B}}\stackrel{{\scriptsize \blue  (\ref{nabla-varep})}}{=}(B\ot \varepsilon_{H})\co \nabla_{B\otimes
H}^{\varphi_{B}}\co P_{\varphi_{B}}\stackrel{{\scriptsize\blue  (\ref{nabla-psi})}}{=}(B\ot \varepsilon_{H})\co P_{\varphi_{B}}\stackrel{{\scriptsize\blue  (\ref{e-psi})}}{=}\varphi_{B}.$$

Trivially, 
\begin{equation}
\label{eqnew}
\mu_{B}\co (B\ot u_{1}^{\varphi_{B}})\co F_{\sigma}=\sigma\ast u_{2}^{\varphi_{B}}
\end{equation}
 and then we obtain (\ref{sigmaHB-1}). On the other hand, 
$$\nabla_{B\ot H}^{\varphi_{B}}\co F_{\sigma}\stackrel{{\scriptsize \blue  (\ref{delta-sigmaHB})}}{=}((\mu_{B}\co (B\ot u_{1}^{\varphi_{B}})\co F_{\sigma})\ot \mu_{H})\co \delta_{H^{\ot 2}}\stackrel{{\scriptsize\blue  (\ref{sigmaHB-1})}}{=}F_{\sigma}.$$

Therefore, 
$$(B\ot \varepsilon_{H})\co F_{\sigma}\stackrel{{\scriptsize \blue  (\ref{sigmaHB-2})}}{=}(B\ot \varepsilon_{H})\co \nabla_{B\ot H}^{\varphi_{B}}\co F_{\sigma}\stackrel{{\scriptsize \blue  (\ref{nabla-varep})}}{=}\mu_{B}\co (B\ot u_{1}^{\varphi_{B}})\co F_{\sigma}\stackrel{{\scriptsize\blue  (\ref{sigmaHB-1})}}{=}\sigma.$$

Finally, let $\sigma:H\ot H\rightarrow B$ be a morphism satisfying (\ref{sigmaHB-2}) and (\ref{sigmaHB-3}). Then, 
$$\sigma\stackrel{{\scriptsize \blue  (\ref{sigmaHB-3})}}{=}(B\ot \varepsilon_{H})\co F_{\sigma} \stackrel{{\scriptsize \blue  (\ref{sigmaHB-2})}}{=}(B\ot \varepsilon_{H})\co \nabla_{B\ot H}^{\varphi_{B}}\co F_{\sigma} \stackrel{{\scriptsize \blue  (\ref{nabla-varep})}}{=}\mu_{B}\co (B\ot u_{1}^{\varphi_{B}})\co F_{\sigma}\stackrel{{\scriptsize \blue  (\ref{eqnew})}}{=}\sigma\ast u_{2}^{\varphi_{B}}.$$
\end{proof}

\begin{apart}
{\rm Let $H$ be a  weak Hopf algebra, let  $\varphi_{B}:H\ot B\rightarrow B$ be a measuring and let $\sigma:H\ot H\rightarrow B$  be a morphism. In a similar way to what was proven in the previous proposition, we can ensure that the equality
\begin{equation}
\label{eqnew1}
\mu_{B}\co (u_{1}^{\varphi_{B}}\ot B)\co G_{\sigma}=u_{2}^{\varphi_{B}}\ast \sigma
\end{equation}
holds. As a consequence we can obtain the following result:
}
\end{apart}

\begin{prop}
\label{sigmaprop1} Let $H$ be a  weak Hopf algebra, let  $\varphi_{B}:H\ot B\rightarrow B$ be a measuring and let $\sigma:H\ot H\rightarrow B$  be a morphism. Let $u_{2}^{\varphi_{B}}$ be the morphism defined in (\ref{def-un}).  If $u_{2}^{\varphi_{B}}\ast \sigma=\sigma$, the equality 
\begin{equation}
\label{sigmaHB-11} 
\mu_{B}\co (u_{1}^{\varphi_{B}}\ot B)\co G_{\sigma}=\sigma, 
\end{equation}
holds and, as a consequence, we have the following identities:
\begin{equation}
\label{sigmaHB-21} 
\nabla_{H\ot B}^{\varphi_{B}}\co G_{\sigma}=G_{\sigma}, 
\end{equation}
\begin{equation}
\label{sigmaHB-31} 
(\varepsilon_{H}\ot B)\co G_{\sigma}=\sigma. 
\end{equation}

Moreover, if $\sigma:H\ot H\rightarrow B$  is a morphism satisfying (\ref{sigmaHB-21}) and (\ref{sigmaHB-31}), we have that 
$u_{2}^{\varphi_{B}}\ast \sigma=\sigma$.
\end{prop}

\begin{rem}
{\rm 
By the previous propositions, \cite[Propositions 2.7 and 2.8]{Gu2-Val} and (\ref{ts1}), (\ref{ts2}), if $\sigma:H\ot H\rightarrow B$ satisfies that  $\sigma\ast u_{2}^{\varphi_{B}}=\sigma$, we have 
\begin{equation}
\label{sigpir}
\sigma\co (\mu_{H}\ot H)\co (H\ot \Pi_{H}^{R}\ot H)=\sigma\co (H\ot \mu_{H})\co (H\ot \Pi_{H}^{R}\ot H),
\end{equation}
\begin{equation}
\label{sigbpil}
\sigma\co (\mu_{H}\ot H)\co (H\ot \bar{\Pi}_{H}^{L}\ot H)=\sigma\co (H\ot \mu_{H})\co (H\ot \bar{\Pi}_{H}^{L}\ot H)
\end{equation}
and, if $u_{2}^{\varphi_{B}}\ast \sigma=\sigma$,  the equalities
\begin{equation}
\label{sigpil}
\sigma\co (\mu_{H}\ot H)\co (H\ot \Pi_{H}^{L}\ot H)=\sigma\co (H\ot \mu_{H})\co (H\ot \Pi_{H}^{L}\ot H),
\end{equation}
\begin{equation}
\label{sigbpir}
\sigma\co (\mu_{H}\ot H)\co (H\ot \bar{\Pi}_{H}^{R}\ot H)=\sigma\co (H\ot \mu_{H})\co (H\ot \bar{\Pi}_{H}^{R}\ot H)
\end{equation}
hold.
}
\end{rem}

\begin{apart}
{\rm 
Let $H$ be a weak Hopf algebra and let $B$ be an algebra. By the previous results, if $\varphi_{B}:H\ot B\rightarrow B$ is a measuring, and  $\sigma:H\ot H\rightarrow B$  is a morphism such that $\sigma\ast u_{2}^{\varphi_{B}}=\sigma$, we have a quadruple $${\Bbb B}_{H}=(B,H,\psi_{H}^{B}=P_{\varphi_{B}}, \sigma_{H}^{B}=F_{\sigma})$$ as the ones introduced in \cite{mra-preunit} to define the notion of weak crossed product. For the quadruple ${\mathbb B}_{H}$ there exists a product in $B\ot H$ defined by 
$$
\mu_{B\ot_{\varphi_{B}}^{\sigma}H}=(\mu_B\ot H)\co (\mu_B\ot F_{\sigma})\co (B\ot
P_{\varphi_{B}}\ot H)
$$
and let $\mu_{A\times_{\varphi_{B}}^{\sigma} H}$ be the product
$$ \mu_{B\times_{\varphi_{B}}^{\sigma} H} = p_{B\ot H}^{\varphi_{B}}\co\mu_{B\ot H}\co
(i_{B\ot H}^{\varphi_{B}}\ot i_{B\ot H}^{\varphi_{B}}),
$$
where $B\times_{\varphi_{B}}^{\sigma}  H$,
$i_{B\ot H}^{\varphi_{B}}:B\times_{\varphi_{B}}^{\sigma}  H\rightarrow B\ot H$ and $p_{B\ot H}^{\varphi_{B}}:B\ot H\rightarrow B\times_{\varphi_{B}}^{\sigma}  H$ denote the image, the injection, and the
projection associated to the factorization of $\nabla_{B\ot H}^{\varphi_{B}}$.

Following \cite{mra-preunit} we say that ${\Bbb B}_{H}$
satisfies the twisted condition if
\begin{equation}\label{twcp}
(\mu_B\ot H)\co (B\ot P_{\varphi_{B}})\co (F_{\sigma}\ot B) =
(\mu_B\ot H)\co (B\ot F_{\sigma})\co (P_{\varphi_{B}}\ot H)\co
(H\ot P_{\varphi_{B}})
\end{equation}
and   the  cocycle condition holds if
\begin{equation}\label{cwcp}
(\mu_B\ot H)\co (B\ot F_{\sigma}) \co (F_{\sigma}\ot H) =
(\mu_B\ot H)\co (B\ot F_{\sigma})\co (P_{\varphi_{B}}\ot H)\co
(H\ot F_{\sigma}).
\end{equation}

Note that, if ${\Bbb B}_{H}$
satisfies the twisted condition, by \cite[Proposition 3.4]{mra-preunit}, 
 and (\ref{sigmaHB-2}) we obtain
\begin{equation}\label{c11}
(\mu_B\otimes H)\circ (B\otimes F_{\sigma})\circ
(P_{\varphi_{B}}\otimes H)\circ (H\otimes \nabla_{B\otimes H}^{\varphi_{B}}) =
 (\mu_B\otimes H)\circ (B\otimes
F_{\sigma})\circ (P_{\varphi_{B}}\otimes H),
\end{equation}
\begin{equation}\label{aw1}
 (\mu_B\otimes H)\circ
(A\otimes F_{\sigma})\circ (\nabla_{B\otimes H}^{\varphi_{B}}\otimes H) =
(\mu_B\otimes H)\circ (B\otimes F_{\sigma}).
\end{equation}

Then, composing with $B\ot \varepsilon_{H}$ in (\ref{aw1}), and applying (\ref{sigmaHB-3}) we have 
\begin{equation}
\label{sigmaHB-4} 
\mu_{B}\co (B\ot \sigma)\co ( \nabla_{B\ot H}^{\varphi_{B}}\ot H)=\mu_{B}\co (B\ot \sigma). 
\end{equation}

Therefore, if ${\Bbb B}_{H}$
satisfies the twisted condition, the following equality
\begin{equation}
\label{sigmaHB-5} 
\mu_{B}\co (B\ot \sigma)\co ( (\nabla_{B\ot H}^{\varphi_{B}}\co (\eta_{B}\ot H))\ot H)= \sigma. 
\end{equation}
holds.
}
\end{apart}

\begin{teo}
\label{teo-twis-cocy}
Let $H$ be a  weak Hopf algebra, let 
$\varphi_{B}:H\ot B\rightarrow B$ be a measuring, and let $\sigma:H\ot
H\rightarrow B$  be a morphism such that $\sigma\ast u_{2}^{\varphi_{B}}=\sigma$. Let ${\Bbb B}_{H}$ be the associated quadruple. Then, the following assertions hold. 

\begin{itemize}
\item[(i)] The quadruple ${\Bbb
B}_{H}$ satisfies the twisted condition (\ref{twcp}) iff  $\sigma$ satisfies the twisted condition 
\begin{equation}
\label{t-sigma} \mu_{B}\co (B\ot \sigma)\co (P_{\varphi_{B}}\ot H)\co (H\ot P_{\varphi_{B}})=\mu_{B}\co (B\ot \varphi_{B})\co
(F_{\sigma}\ot B).
\end{equation}
\item[(ii)] The quadruple ${\Bbb
B}_{H}$ satisfies the cocycle condition (\ref{cwcp}) iff  $\sigma$ satisfies the cocycle condition 
\begin{equation}
\label{c-sigma} \mu_{B}\co (B\ot \sigma)\co (P_{\varphi_{B}}\ot H)\co (H\ot F_{\sigma})=\mu_{B}\co (B\ot \sigma)\co (F_{\sigma}\ot B).
\end{equation} 
\end{itemize}
\end{teo}

\begin{proof} The proof follows as in \cite[Theorems 3.12 and 3.13]{nmra5} because the cocommutativity condition for $H$ is not necessary and  (\ref{sigmaHB-3}) holds.

\end{proof}

If the twisted and the cocycle conditions hold, the product
$\mu_{B\ot_{\varphi_{B}}^{\sigma}H}$ is associative and normalized with respect to
$\nabla_{B\otimes H}^{\varphi_{B}}$, i.e., 
\begin{equation}
\label{norz}
\nabla_{B\otimes H}^{\varphi_{B}}\co \mu_{B\ot_{\varphi_{B}}^{\sigma}H}=\mu_{B\ot_{\varphi_{B}}^{\sigma}H}=\mu_{B\ot_{\varphi_{B}}^{\sigma}H}\co (\nabla_{B\otimes H}^{\varphi_{B}}\ot \nabla_{B\otimes H}^{\varphi_{B}}))
\end{equation}
and  we have
\begin{equation}
\label{otra-prop} \mu_{B\ot_{\varphi_{B}}^{\sigma}H}\co (\nabla_{B\otimes H}^{\varphi_{B}}\ot B\ot
H)=\mu_{B\ot_{\varphi_{B}}^{\sigma}H}= \mu_{B\ot_{\varphi_{B}}^{\sigma}H}\circ (B\otimes H\otimes
\nabla_{B\otimes H}^{\varphi_{B}}).
\end{equation}

Then, $\mu_{B\times_{\varphi_{B}}^{\sigma}H}$ is associative as
well (see \cite[Proposition 3.7]{mra-preunit}). Hence we define:

\begin{defin}\label{wcp-def}
{\rm If ${\Bbb B}_{H}$  satisfies
(\ref{twcp}) and (\ref{cwcp}) we say that $(B\ot H,
\mu_{B\ot_{\varphi_{B}}^{\sigma}H})$ is a weak crossed product.}
\end{defin}

The next natural question that arises is if it is possible to endow
$\mu_{B\times_{\varphi_{B}}^{\sigma} H}$ with a unit, and hence with an algebra structure. As we
recall in \cite{mra-preunit}, we need to use the notion
of preunit to obtain this unit. In our setting, if $\mu_{B\otimes_{\varphi_{B}}^{\sigma} H}$ is an associative product, by  \cite[Remark 2.4]{mra-preunit}, $\nu:K\rightarrow B\ot H$ is a preunit if 
\begin{equation}
\label{p-unit}
\mu_{B\ot_{\varphi_{B}}^{\sigma}H}\circ (B\otimes H\otimes \nu)=\mu_{B\ot_{\varphi_{B}}^{\sigma}H}\circ (\nu\otimes B\otimes H),\;\;\; \nu =
\mu_{B\ot_{\varphi_{B}}^{\sigma}H}\circ (\nu\otimes \nu).
\end{equation}

By \cite[Corollary 3.12]{mra-preunit}, we know that, if $\nu$ is a preunit for $(B\ot H,
\mu_{B\ot_{\varphi_{B}}^{\sigma}H})$,  the object $B\times_{\varphi_{B}}^{\sigma} H$ is an algebra with  product $\mu_{B\times_{\varphi_{B}}^{\sigma} H}$ and unit $\eta_{B\times_{\varphi_{B}}^{\sigma} H}=p_{B\otimes H}^{\varphi_{B}}\circ\nu$. 

The folllowing proposition is a tool to establish the conditions under which the morphism $\nu=\nabla_{B\otimes H}^{\varphi_{B}}\co (\eta_{B}\ot \eta_{H})$ is a preunit for $\mu_{B\ot_{\varphi_{B}}^{\sigma}H}$.

\begin{prop}
\label{tech-lem}
Let $H$ be a  weak Hopf algebra, let 
$\varphi_{B}:H\ot B\rightarrow B$ be a measuring, and let $\sigma:H\ot
H\rightarrow B$  be a morphism such that $\sigma\ast u_{2}^{\varphi_{B}}=\sigma$. Then, the following equalities hold. 
\begin{equation}
\label{sigma-eta-1} \sigma\co (\eta_{H}\ot H)=\sigma\co c_{H,H}\co(H\ot \overline{\Pi}_{H}^{L})\co \delta_{H}, 
\end{equation}
\begin{equation}
\label{sigma-eta-2} \sigma\co (H\ot \eta_{H})=\sigma\co (H\ot \Pi_{H}^{R})\co \delta_{H}.
\end{equation} 
\end{prop}

\begin{proof}  The first equality follows from equalities (\ref{sigbpil}) and (\ref{eq-id-bpi}), and the second one is consequence of (\ref{sigpir}) and (\ref{id}).
\end{proof}

\begin{defin}
{\rm  Let $H$ be a  weak Hopf algebra, let 
$\varphi_{B}:H\ot B\rightarrow B$ be a measuring, and let $\sigma:H\ot
H\rightarrow B$  be a morphism.  We say that $\sigma$
satisfies the normal condition if
\begin{equation}
\label{normal-sigma} \sigma\co (\eta_{H}\ot H)=\sigma\co (H\ot
\eta_{H})=u_{1}^{\varphi_{B}}.
\end{equation}

Therefore, if  $\sigma\ast u_{2}^{\varphi_{B}}=\sigma$, by Proposition \ref{tech-lem}, $\sigma$ is normal if and only if
$$\sigma\co c_{H,H}\co(H\ot \overline{\Pi}_{H}^{L})\co \delta_{H}=\sigma\co (H\ot \Pi_{H}^{R})\co \delta_{H}=u_{1}^{\varphi_{B}}.
$$
}

\end{defin}

\begin{teo}
\label{norma-sigma-prop3} Let $H$ be a  weak Hopf algebra, let 
$\varphi_{B}:H\ot B\rightarrow B$ be a measuring such that 
\begin{equation}
\label{nvar-eta}
\nabla_{B\ot H}^{\varphi_{B}}\co (B\ot \eta_{H})=P_{\varphi_{B}}\co (\eta_{H}\ot B),
\end{equation} 
and let $\sigma:H\ot
H\rightarrow B$  be a morphism such that $\sigma\ast u_{2}^{\varphi_{B}}=\sigma$. Let  ${\Bbb B}_{H}$ be the associated  quadruple and assume that ${\Bbb
B}_{H}$ satisfies the twisted and the cocycle conditions
(\ref{twcp}) and (\ref{cwcp}). Then, $\nu=\nabla_{B\otimes H}^{\varphi_{B}}\co (\eta_{B}\ot \eta_{H})$ is a preunit for the weak crossed
product associated to ${\Bbb B}_{H}$ if and only if
\begin{equation}
\label{sigma-preunit1} F_{\sigma}\co (\eta_{H}\ot
H)=F_{\sigma}\co (H\ot \eta_{H})=\nabla_{B\otimes H}^{\varphi_{B}}\co (\eta_{B}\ot
H).
\end{equation}

\end{teo}

\begin{proof}  By \cite[Theorem 3.11]{mra-preunit} a morphism $\upsilon$ is a preunit for the associated weak crossed if and only if 
\begin{equation}\label{pre1-wcp}
 (\mu_B\otimes H)\circ (B\otimes F_{\sigma})\circ (P_{\varphi_{B}}\otimes H)\circ (H\otimes \upsilon) =
 \nabla_{B\otimes H}^{\varphi_{B}}\circ (\eta_B\otimes H),
 \end{equation}
\begin{equation}\label{pre2-wcp}
 (\mu_B\otimes H)\circ (B\otimes F_{\sigma})\circ (\upsilon\otimes H) = \nabla_{B\otimes H}^{\varphi_{B}}\circ (\eta_B\otimes H)
\end{equation}
and 
\begin{equation}\label{pre3-wcp}
(\mu_B\otimes H)\circ (B\otimes P_{\varphi_{B}})\circ (\upsilon\otimes B) =
(\mu_{B}\ot H)\co (B\ot \upsilon) 
\end{equation}
hold.  Then the theorem follows because, on the one hand
\begin{itemize}
\item[ ]$\hspace{0.38cm} (\mu_B\otimes H)\circ (B\otimes P_{\varphi_{B}})\circ (\nu\otimes B) $
\item [ ]$=  (\mu_B\otimes H)\circ (B\otimes P_{\varphi_{B}})\circ ((P_{\varphi_{B}}\circ (\eta_{H}\ot \eta_{B}))\otimes B)  $ {\scriptsize ({\blue  by the unit properties})}
\item [ ]$=P_{\varphi_{B}}\circ (\eta_{H}\ot B)$ {\scriptsize ({\blue by (\ref{wmeas-wcp}) and the unit properties})}
\item [ ]$=\nabla_{B\ot H}^{\varphi_{B}}\co (B\ot \eta_{H})   $ {\scriptsize ({\blue by (\ref{nvar-eta})})}
\item [ ]$=(\mu_{B}\ot H)\co (B\ot \nu)$ {\scriptsize ({\blue  by the unit properties})}.
\end{itemize} 
and, on the other hand, by the unit properties and (\ref{twcp}), we have 
$$ (\mu_B\otimes H)\circ (B\otimes F_{\sigma})\circ (P_{\varphi_{B}}\otimes H)\circ (H\otimes \nu) =F_{\sigma}\co (H\ot \eta_{H}).$$

Finally, by (\ref{aw1}),  the equality 
$$(\mu_B\otimes H)\circ (B\otimes P_{\varphi_{B}})\circ (\nu\otimes B) =F_{\sigma}\co (\eta_{H}\ot H)$$
holds. 
\end{proof} 

\begin{rem}
{\rm Note that if $\upsilon$ is a preunit for the weak crossed
product associated to ${\Bbb B}_{H}$, by  \cite[Theorem 3.11]{mra-preunit}, the equality (\ref{pre3-wcp}) holds for $\upsilon$. Then, 
\begin{equation}\label{pre4-wcp}
\nabla_{B\otimes H}^{\varphi_{B}}\co \upsilon=\upsilon
\end{equation}
holds. Therefore the preunit of a weak crossed product, if it exists, is unique because if $(B\ot H,
\mu_{B\ot_{\varphi_{B}}^{\sigma}H})$ admits two preunits $\upsilon_{1}$, $\upsilon_{2}$, we have $\eta_{B\times_{\varphi_{B}}^{\sigma} H}=p_{B\otimes H}^{\varphi_{B}}\circ \upsilon_{1}=p_{B\otimes H}^{\varphi_{B}}\circ \upsilon_{2}$ and then
$$\upsilon_{1}=\nabla_{B\otimes H}^{\varphi_{B}}\co \upsilon_{1}=\nabla_{B\otimes H}^{\varphi_{B}}\co \upsilon_{2}=\upsilon_{2}.$$
}
\end{rem}

As a consequence of Theorem \ref{norma-sigma-prop3}, and by Proposition \ref{tech-lem} we have: 

\begin{cor}
\label{norma-sigma-prop4} Let $H$ be a  weak Hopf algebra, let 
$\varphi_{B}:H\ot B\rightarrow B$ be a measuring, let $\sigma:H\ot
H\rightarrow B$  be a morphism and let  ${\Bbb B}_{H}$ be the associated  quadruple such that  the assumptions of Theorem \ref{norma-sigma-prop3} hold. Then, $\nu=\nabla_{B\otimes H}^{\varphi_{B}}\co (\eta_{B}\ot \eta_{H})$ is a preunit for the weak crossed
product associated to ${\Bbb B}_{H}$ if and only if $\sigma$
satisfies the normal condition (\ref{normal-sigma}).
\end{cor}

\begin{proof}
Considering (\ref{nabla-eta}), the proof follows from  the equalities 
$$ F_{\sigma}\co (\eta_{H}\ot H)=((\sigma\co c_{H,H}\co(H\ot \overline{\Pi}_{H}^{L})\co \delta_{H})\ot H)\co \delta_{H},
$$
and
$$
F_{\sigma}\co (H\ot \eta_{H})=((\sigma\co (H\ot \Pi_{H}^{R})\co \delta_{H})\ot H)\co \delta_{H},
$$
 which hold by  (\ref{d-pi31}), (\ref{d-pi4}) and the naturality of $c$.
 \end{proof}

Therefore, as a consequence of the previous results,  we obtain the complete characterization of weak crossed products associated to a measuring. 

\begin{cor}
\label{crossed-product1}
Let $H$ be a  weak Hopf algebra, let 
$\varphi_{B}:H\ot B\rightarrow B$ be a measuring, let $\sigma:H\ot
H\rightarrow B$  be a morphism and let  ${\Bbb B}_{H}$ be the associated  quadruple such that  the assumptions of Theorem \ref{norma-sigma-prop3} hold. Then the
following statements are equivalent:
\begin{itemize}
\item[(i)] The product $\mu_{B\ot_{\varphi_{B}}^{\sigma}H}$ is associative with preunit
$\nu=\nabla_{B\otimes H}^{\varphi_{B}}\co (\eta_{B}\ot \eta_{H})$ and normalized
with respect to $\nabla_{B\otimes H}^{\varphi_{B}}.$
\item[(ii)] The morphism $\sigma$ satisfies the twisted condition (\ref{t-sigma}), the cocycle condition (\ref{c-sigma}) and  the normal condition (\ref{normal-sigma}).
\end{itemize}
\end{cor}

\begin{rem}
\label{miracle}
{\rm Let $H$ be a  weak Hopf algebra. If $(B,\varphi_{B})$ is a left weak $H$-module algebra the equality (\ref{nvar-eta}) holds because:
\begin{itemize}
\item[ ]$\hspace{0.38cm} \nabla_{B\ot H}^{\varphi_{B}}\co (B\ot \eta_{H})$
\item [ ]$=  (\mu_{B}\ot H)\co (B\ot ((u_{1}^{\varphi_{B}}\ot H)\co \delta_{H}\co \eta_{H})) $ {\scriptsize ({\blue  by (\ref{nabla-nabla})})}
\item [ ]$= ((\varphi_{B}\co c_{B,H})\ot H)\co (B\ot ((\overline{\Pi}_{H}^{L}\ot H)\co \delta_{H}\co \eta_{H}))  $  {\scriptsize ({\blue by (\ref{B2})})}
\item [ ]$= ((\varphi_{B}\co c_{B,H})\ot H)\co (B\ot (\delta_{H}\co \eta_{H}))$ {\scriptsize  ({\blue by (\ref{edpi})})}
\item [ ]$=P_{\varphi_{B}}\co (\eta_{H}\ot B)$ {\scriptsize  ({\blue by naturality of $c$}).}
\end{itemize} 

Then, if we work with a left weak $H$-module algebra $(B,\varphi_{B})$, Corollary \ref{crossed-product1} holds without assuming (\ref{nvar-eta}).
}
\end{rem}

\section{Equivalent weak crossed products}

The general theory of equivalent weak crossed products was presented in \cite{equivalent}. In this section we remember the criterion obtained in \cite{equivalent} that characterises the equivalence between two weak crossed products and we give the translation of this criterion to the particular setting of weak crossed products induced by measurings. 

We shall start by introducing the notion of equivalence of weak crossed products induced by measurings.

\begin{defin}
\label{def-equiv}
{\rm  Let $H$ be a  weak Hopf algebra, let 
$\varphi_{B}, \;\phi_{B}:H\ot B\rightarrow B$ be  measurings, and let $\sigma, \;\tau:H\ot
H\rightarrow B$  be  morphisms such that $\sigma\ast u_{2}^{\varphi_{B}}=\sigma$, $\tau\ast u_{2}^{\phi_{B}}=\tau$. Assume that $\sigma$, $\tau$ satisfy the twisted
condition (\ref{t-sigma}) and the 2-cocycle condition
(\ref{c-sigma}), and suppose that $\nu$ is a preunit for $\mu_{B\ot_{\varphi_{B}}^{\sigma}H}$, and $u$ is a preunit for $\mu_{B\ot_{\phi_{B}}^{\tau}H}$. We say that $(B\ot H, \mu_{B\ot_{\varphi_{B}}^{\sigma}H})$ and $(B\ot H, \mu_{B\ot_{\phi_{B}}^{\tau}H})$ are equivalent weak crossed products if there is an isomorphism 
$$\varUpsilon :B\times_{\varphi_{B}}^{\sigma}H\rightarrow B\times_{\phi_{B}}^{\tau}H$$ of  algebras, left $B$-modules and right $H$-comodules,  where  the left actions are defined by $\varphi_{B\times_{\varphi_{B}}^{\sigma}H}= p_{B\otimes H}^{\varphi_{B}}\co (\mu_{B}\ot H)\co (B\ot i_{B\otimes H}^{\varphi_{B}})$, $\varphi_{B\times_{\phi_{B}}^{\tau}H}= p_{B\otimes H}^{\phi_{B}}\co (\mu_{B}\ot H)\co (B\ot i_{B\otimes H}^{\phi_{B}})$, and the right coactions are $\rho_{B\times_{\varphi_{B}}^{\sigma}H}=(p_{B\otimes H}^{\varphi_{B}}\ot H)\co (B\ot \delta_{H})\co i_{B\otimes H}^{\varphi_{B}}$, $\rho_{B\times_{\phi_{B}}^{\tau}H}=(p_{B\otimes H}^{\phi_{B}}\ot H)\co (B\ot \delta_{H})\co i_{B\otimes H}^{\phi_{B}}$.}
\end{defin}

In  our setting the general criterion \cite[Theorem 1.7]{equivalent}  that characterizes equivalent weak crossed products  admits the following formulation:

\begin{teo}
\label{Teo22}
Let $H$ be a  weak Hopf algebra, let 
$\varphi_{B}, \;\phi_{B}:H\ot B\rightarrow B$ be  measurings and let $\sigma, \;\tau:H\ot
H\rightarrow B$  be  morphisms such that $\sigma\ast u_{2}^{\varphi_{B}}=\sigma$, $\tau\ast u_{2}^{\phi_{B}}=\tau$. Assume that $\sigma$, $\tau$ satisfy the twisted
condition (\ref{t-sigma}), the 2-cocycle condition
(\ref{c-sigma}) and suppose that $\nu$ is a preunit for $\mu_{B\ot_{\varphi_{B}}^{\sigma}H}$ and $u$ is a preunit for $\mu_{B\ot_{\phi_{B}}^{\tau}H}$. Let $(B, H, P_{\varphi_{B}}, F_{\sigma})$ be the quadruple associated to $\varphi_{B}$, and $\sigma$ and let $(B, H, P_{\phi_{B}}, F_{\tau})$ be the one associated to $\phi_{B}$ and $\tau$.The following assertions are equivalent:
\begin{itemize}
\item[(i)] The weak crossed products  $(B\ot H, \mu_{B\ot_{\varphi_{B}}^{\sigma}H})$ and $(B\ot H, \mu_{B\ot_{\phi_{B}}^{\tau}H})$ are equivalent.

\item[(ii)] There exist two morphisms $T, S:B\ot H\rightarrow B\ot H,$
of left $B$-modules for the trivial action $\varphi_{B\ot H}=\mu_{B}\ot H$, and right $H$-modules for the  trivial coaction $\rho_{B\ot H}=B\ot \delta_{H}$,  satisfying the conditions 
\begin{equation}
\label{preserv-preunit}
T\co \nu=u,
\end{equation}
\begin{equation}
\label{preserv-product}
T\co \mu_{B\ot_{\varphi_{B}}^{\sigma}H}=\mu_{B\ot_{\phi_{B}}^{\tau}H}\co (T\ot T),
\end{equation}
\begin{equation}
\label{preserv-idemp}
S\co T=\nabla_{B\otimes H}^{\varphi_{B}},\;\; T\co S=\nabla_{B\otimes H}^{\phi_{B}}.
\end{equation}
\item[(iii)] There exist two morphisms  $\theta,\gamma:H\rightarrow B\ot H$  of right $H$-modules for the  trivial coaction satisfying the conditions 
\begin{equation}
\label{gamma-theta-idemp}
\theta=\nabla_{B\otimes H}^{\varphi_{B}}\co \theta,
\end{equation}
\begin{equation}
\label{gamma-theta-special}
(\mu_{B}\ot H)\co (B\ot \theta)\co \gamma=\nabla_{B\otimes H}^{\varphi_{B}}\co (\eta_{B}\ot H),
\end{equation}
\begin{equation}
\label{gamma-theta-psi}
P_{\phi_{B}}=(\mu_{B}\ot H)\co (\mu_{B}\ot \gamma)\co (B\ot P_{\varphi_{B}})\co (\theta\ot B),
\end{equation}
\begin{equation}
\label{gamma-theta-sigma}
F_{\tau}=(\mu_{B}\ot H)\co (B\ot \gamma)\co \mu_{B\ot_{\varphi_{B}}^{\sigma}H}\co (\theta\ot \theta) ,
\end{equation}
\begin{equation}
\label{gamma-theta-preunit}
u=(\mu_{B}\ot H)\co (B\ot \gamma)\co \nu.
\end{equation}

\end{itemize}

\end{teo}

\begin{proof} The proof of this theorem is the one developed in \cite[Theorem 1.7]{equivalent}. It is not difficult to check the right $H$- comodule condition for the morphisms $S$, $T$, $\theta$ and $\gamma$.  We leave the details of the proof to the reader.
\end{proof}

\begin{prop}
\label{lemma1.16}
Let $H$ be a  weak Hopf algebra, let 
$\varphi_{B}, \;\phi_{B}:H\ot B\rightarrow B$ be  measurings and let $\sigma, \;\tau:H\ot
H\rightarrow B$  be  morphisms such that $\sigma\ast u_{2}^{\varphi_{B}}=\sigma$, $\tau\ast u_{2}^{\phi_{B}}=\tau$. Assume that $\sigma$, $\tau$ satisfy the twisted
condition (\ref{t-sigma}), the 2-cocycle condition
(\ref{c-sigma}) and suppose that $\nabla_{B\otimes H}^{\varphi_{B}}\co (\eta_{B}\ot \eta_{H})$ is a preunit for $\mu_{B\ot_{\varphi_{B}}^{\sigma}H}$ and $\nabla_{B\otimes H}^{\phi_{B}}\co (\eta_{B}\ot \eta_{H})$ is a preunit for $\mu_{B\ot_{\phi_{B}}^{\tau}H}$. If    $(B\ot H, \mu_{B\ot_{\varphi_{B}}^{\sigma}H})$ and  $(B\ot H, \mu_{B\ot_{\phi_{B}}^{\tau}H})$ are equivalent weak crossed products, there exists  morphisms $T, S:B\ot H\rightarrow B\ot H$ of left $B$-modules for the trivial action and right $H$-comodules for the trivial coaction such that  
\begin{equation}
\label{uT}
\nabla_{B\otimes H}^{\phi_{B}}\co (\eta_{B}\ot \eta_{H})=T\co (\eta_{B}\ot \eta_{H}), \;\;\; \nabla_{B\otimes H}^{\varphi_{B}}\co (\eta_{B}\ot \eta_{H})=S\co (\eta_{B}\ot \eta_{H}).
\end{equation}
\end{prop}
\begin{proof}  If $(B\ot H, \mu_{B\ot_{\varphi_{B}}^{\sigma}H})$ and $(B\ot H, \mu_{B\ot_{\phi_{B}}^{\tau}H})$ are equivalent weak crossed products, there exists an  isomorphism of  algebras, left $B$-modules and right $H$-comodules
$\varUpsilon:B\times_{\varphi_{B}}^{\sigma}H\rightarrow B\times_{\phi_{B}}^{\tau}H$.  By (i)$\Rightarrow$ (ii) of the previous theorem there exists two morphisms  of left $B$-modules and right $H$-comodules $T, S:B\ot H\rightarrow B\ot H,$   defined by 
$$
T=i_{B\ot H}^{\phi_{B}}\co \varUpsilon\co p_{B\ot H}^{\varphi_{B}},\;\; S=i_{B\ot H}^{\varphi_{B}}\co \varUpsilon^{-1}\co p_{B\ot H}^{\phi_{B}} 
$$
and satisfying the conditions  (\ref{preserv-idemp}),
\begin{equation}
\label{TST}
T\co S\co T=T
\end{equation}
and
\begin{equation}
\label{STS}
S\co T\co S=S. 
\end{equation}

Then, 
$$\nabla_{B\otimes H}^{\phi_{B}}\co (\eta_{B}\ot \eta_{H})\stackrel{{\scriptsize \blue  (\ref{preserv-preunit})}}{=}T\co \nabla_{B\otimes H}^{\varphi_{B}}\co (\eta_{B}\ot \eta_{H})\stackrel{{\scriptsize \blue  (\ref{preserv-idemp})}}{=}T\co S\co T\co (\eta_{B}\ot \eta_{H})\stackrel{{\scriptsize \blue  (\ref{TST})}}{=} T\co (\eta_{B}\ot \eta_{H}).$$

On the other hand, if $\nu$ and $u$ are the preunits of $(B\ot H, \mu_{B\ot_{\varphi_{B}}^{\sigma}H})$ and $(B\ot H, \mu_{B\ot_{\phi_{B}}^{\tau}H})$, by  (\ref{preserv-preunit}) we have that $S\circ T\circ \nu=S\circ u$. Then, by (\ref{preserv-idemp})  we have that $\nabla_{B\ot H}^{\varphi_{B}}\circ \nu=S\circ u$ and applying (\ref{pre4-wcp}) we obtain that 
\begin{equation}
\label{preserv-preunit1}
S\circ u= \nu
\end{equation}
holds.  Therefore, in our particular case, we have 
$$\nabla_{B\otimes H}^{\varphi_{B}}\co (\eta_{B}\ot \eta_{H})\stackrel{{\scriptsize \blue  (\ref{preserv-preunit1})}}{=}S\co \nabla_{B\otimes H}^{\phi_{B}}\co (\eta_{B}\ot \eta_{H})\stackrel{{\scriptsize \blue  (\ref{preserv-idemp})}}{=}S\co T\co S\co (\eta_{B}\ot \eta_{H})\stackrel{{\scriptsize \blue  (\ref{STS})}}{=} S\co (\eta_{B}\ot \eta_{H}).$$

\end{proof}

\begin{teo}
\label{equiv-hh}
Let $H$ be a  weak Hopf algebra, let 
$\varphi_{B}, \;\phi_{B}:H\ot B\rightarrow B$ be  measurings, and let $\sigma, \;\tau:H\ot
H\rightarrow B$  be  morphisms such that $\sigma\ast u_{2}^{\varphi_{B}}=\sigma$, $\tau\ast u_{2}^{\phi_{B}}=\tau$. Assume that $\sigma$, $\tau$ satisfy the twisted
condition (\ref{t-sigma}), the 2-cocycle condition
(\ref{c-sigma}) and suppose that $\nu$ is a preunit for $\mu_{B\ot_{\varphi_{B}}^{\sigma}H}$, and $u$ is a preunit for $\mu_{B\ot_{\phi_{B}}^{\tau}H}$. 
The following assertions are equivalent:
\begin{itemize}
\item[(i)] The weak crossed products  $(B\ot H, \mu_{B\ot_{\varphi_{B}}^{\sigma}H})$ and $(B\ot H, \mu_{B\ot_{\phi_{B}}^{\tau}H})$ are equivalent.
\item[(ii)] There exists two morphisms $h,h^{-1}:H\rightarrow B$ such that
\begin{equation}
\label{h1} h^{-1}\ast h=u_{1}^{\varphi_{B}},
\end{equation}
\begin{equation}
\label{h2}
h\ast h^{-1}\ast h=h, \;\;\;  h^{-1}\ast h\ast h^{-1}=h^{-1},
\end{equation}
\begin{equation}
\label{h4} \phi_{B}=\mu_{B}\co (\mu_{B}\ot h^{-1})\co (h\ot P_{\varphi_{B}})\co (\delta_{H}\ot B), 
\end{equation}
\begin{equation}
\label{h5}
\tau=\mu_{B}\co (B\ot h^{-1})\co \mu_{B\ot_{\varphi_{B}}^{\sigma}H}\co (((h\ot H)\co \delta_{H})\ot 
((h\ot H)\co \delta_{H})).
\end{equation}
\begin{equation}
\label{h3}
u=((\mu_{B}\co (B\ot h^{-1}))\ot H)\co (B\ot \delta_{H})\co \nu, 
\end{equation}
\end{itemize}
\end{teo}

\begin{proof} First we will prove that (i)$\Rightarrow$ (ii). By Theorem \ref{Teo22} there exists  two morphisms $T, S:B\ot H\rightarrow B\ot H$ of left $B$-modules for the trivial action and right $H$-comodules for trivial coaction defined as in the proof of the previous proposition and  satisfying the conditions (\ref{preserv-preunit}), (\ref{preserv-product}),  (\ref{preserv-idemp}), (\ref{TST}) and (\ref{STS}). Also,  $S$ preserves the preunit, i.e.,  (\ref{preserv-preunit1}) holds, and  $S$ is multiplicative, i.e., $S\co \mu_{B\ot_{\phi_{B}}^{\tau}H}=\mu_{B\ot_{\varphi_{B}}^{\sigma}H}\co (S\ot S)$ (\cite[(37)]{equivalent}). Also, by (ii)$\Rightarrow$(iii) of Theorem \ref{Teo22}, there exists two morphisms  $\theta, \gamma:H\rightarrow B\ot H$  of right $H$-comodules defined by 
\begin{equation}
\label{tg}
\theta=S\co (\eta_{B}\ot H), \;\;\; \gamma=T\co (\eta_{B}\ot H). 
\end{equation}

Then, 
\begin{equation}
\label{tg-1}
S=(\mu_{B}\ot H)\co (B\ot \theta), \;\;\; T= (\mu_{B}\ot H)\co (B\ot \gamma). 
\end{equation}

For $\theta$ and $\gamma$ the equalities (\ref{gamma-theta-idemp}),  (\ref{gamma-theta-special}),  (\ref{gamma-theta-psi}), (\ref{gamma-theta-sigma}), and  (\ref{gamma-theta-preunit}) hold. Moreover, 
\begin{equation}
\label{gamma-theta-idemp-1}
\gamma=\nabla_{B\otimes H}^{\phi_{B}}\co \gamma,
\end{equation}
\begin{equation}
\label{gamma-theta-special-1}
(\mu_{B}\ot H)\co (B\ot \gamma)\co \theta=\nabla_{B\otimes H}^{\phi_{B}}\co (\eta_{B}\ot H),
\end{equation}
\begin{equation}
\label{gamma-theta-psi-1}
P_{\varphi_{B}}=(\mu_{B}\ot H)\co (\mu_{B}\ot \theta)\co (B\ot P_{\phi_{B}})\co (\gamma\ot B),
\end{equation}
\begin{equation}
\label{gamma-theta-sigma-1}
F_{\sigma}=(\mu_{B}\ot H)\co (B\ot \theta)\co \mu_{B\ot_{\phi_{B}}^{\tau}H}\co (\gamma\ot \gamma) ,
\end{equation}
\begin{equation}
\label{gamma-theta-preunit-1}
\nu=(\mu_{B}\ot H)\co (B\ot \theta)\co u,
\end{equation}
also hold. Define 
\begin{equation}
\label{h-h-1}
h=(B\ot \varepsilon_{H})\co \theta,\;\; h^{-1}=(B\ot \varepsilon_{H})\co \gamma.
\end{equation}

Then, by the condition of right $H$-comodule morphism for $\theta$ and $\gamma$, we have 
\begin{equation}
\label{th-gh-1}
\theta=(h\ot H)\co \delta_{H},\;\;\;\gamma=(h^{-1}\ot H)\co \delta_{H}.
\end{equation}

The equality (\ref{h1}) holds because 
\begin{itemize}
\item[ ]$\hspace{0.38cm}h^{-1}\ast h$
\item [ ]$=(B\ot \varepsilon_{H})\co (\mu_{B}\ot H)\co (B\ot \theta)\co \gamma$ {\scriptsize ({\blue  by the comodule morphism condition for $\gamma$  and counit properties})}
\item [ ]$=(B\ot \varepsilon_{H})\co \nabla_{B\otimes H}^{\varphi_{B}}\co (\eta_{B}\ot H)$
 {\scriptsize ({\blue by (\ref{gamma-theta-special})})}
\item [ ]$=  u_{1}^{\varphi_{B}} $ {\scriptsize ({\blue by  counit properties}).}
\end{itemize} 

Also, 
\begin{itemize}
\item[ ]$\hspace{0.38cm}h\ast h^{-1}\ast h$
\item [ ]$=\mu_{B}\co (H\ot \varepsilon_{H}\ot \varphi_{B})\co (\theta\ot \eta_{B})$ {\scriptsize ({\blue by the comodule morphism condition for $\theta$ and (\ref{h1})})}
\item [ ]$=(B\ot \varepsilon_{H})\co \nabla_{B\otimes H}^{\varphi_{B}}\co \theta$
 {\scriptsize ({\blue by counit properties and naturality of $c$})}
\item [ ]$= (B\ot \varepsilon_{H})\co \theta$ {\scriptsize ({\blue by   (\ref{gamma-theta-idemp})})}
\item [ ]$=h $ {\scriptsize ({\blue by  counit properties})}
\end{itemize} 
and 
\begin{itemize}
\item[ ]$\hspace{0.38cm}h^{-1}\ast h\ast h^{-1}$
\item [ ]$=(\mu_{B}\ot \varepsilon_{H})\co (B\ot \gamma)\co P_{\varphi_{B}}\co (H\ot \eta_{B})$
{\scriptsize ({\blue by naturality of $c$})}
\item [ ]$=(\mu_{B}\ot \varepsilon_{H})\co (B\ot \gamma)\co \nabla_{B\otimes H}^{\varphi_{B}}\co (\eta_{B}\ot H)$ {\scriptsize ({\blue by  properties of $\eta_{B}$})}
\item [ ]$= (\mu_{B}\ot \varepsilon_{H})\co (B\ot \gamma)\co  (\mu_{B}\ot H)\co (B\ot \theta)\co \gamma$ {\scriptsize ({\blue by  (\ref{gamma-theta-special})})}
\item [ ]$=(\mu_{B}\ot \varepsilon_{H})\co (B\ot  (\nabla_{B\otimes H}^{\varphi_{B}}\co (\eta_{B}\ot H)))\co \gamma $ {\scriptsize ({\blue by  the associativity of $\mu_{B}$ and (\ref{gamma-theta-special-1})}).}
\item [ ]$=(B\ot \varepsilon_{H})\co \nabla_{B\otimes H}^{\varphi_{B}}\co \gamma$  {\scriptsize ({\blue by (\ref{left-module}) and  properties of $\eta_{B}$})}
\item [ ]$=h^{-1}$ {\scriptsize ({\blue by (\ref{gamma-theta-idemp-1})}).}
\end{itemize} 

The equality (\ref{h3}) follows directly from (\ref{gamma-theta-preunit}) because $\gamma$ is a morphism of right $H$-comodules.  Moreover, composing in (\ref{gamma-theta-psi}) with $B\ot  \varepsilon_{H}$, by (\ref{th-gh-1}) we prove (\ref{h4}). Finally,  (\ref{h5}) holds because 
\begin{itemize}
\item[ ]$\hspace{0.38cm}\tau $
\item [ ]$=  (B\ot \varepsilon_{H})\co F_{\tau}$ {\scriptsize ({\blue by  (\ref{sigmaHB-3}) for $\tau$})}
\item [ ]$=  (\mu_{B}\ot \varepsilon_{H})\co (B\ot \gamma)\co \mu_{B\ot_{\varphi_{B}}^{\sigma}H}\co (\theta\ot \theta)$ {\scriptsize ({\blue by  (\ref{gamma-theta-sigma})})}
\item [ ]$=  (\mu_{B}\ot h^{-1})\co \mu_{B\ot_{\varphi_{B}}^{\sigma}H}\co (((h\ot H)\co \delta_{H})\ot 
((h\ot H)\co \delta_{H}))  $
{\scriptsize ({\blue by (\ref{th-gh-1})}).}
\end{itemize} 

Conversely, to prove (ii)$\Rightarrow$(i), define 
$$\theta=(h\ot H)\co \delta_{H},\;\;\; \gamma=(h^{-1}\ot H)\co \delta_{H}.$$ 

Then $\theta$ and $\gamma$ are morphisms of right $H$-comodules, $h=(B\ot \varepsilon_{H})\co \theta$ and $h^{-1}=(B\ot \varepsilon_{H})\co \gamma.$ To prove the equivalence between $(B\ot H, \mu_{B\ot_{\varphi_{B}}^{\sigma}H})$ and $(B\ot H, \mu_{B\ot_{\phi_{B}}^{\tau}H})$, we must show that  (\ref{gamma-theta-idemp}),  (\ref{gamma-theta-special}), (\ref{gamma-theta-psi}), (\ref{gamma-theta-sigma}) and (\ref{gamma-theta-preunit}) hold. First note that,  (\ref{gamma-theta-preunit}) follows from  (\ref{h3}). Also, (\ref{gamma-theta-idemp}) holds because:
\begin{itemize}
\item[ ]$\hspace{0.38cm}\nabla_{B\otimes H}^{\varphi_{B}}\co \theta$
\item [ ]$=((h\ast u_{1}^{\varphi_{B}})\ot H)\co \delta_{H} $ {\scriptsize ({\blue by the coassociativity of $\delta_{H}$ and (\ref{nabla-nabla})})}
\item [ ]$=((h\ast h^{-1}\ast h)\ot H)\co \delta_{H}$ {\scriptsize ({\blue by (\ref{h1})})}
\item [ ]$=  \theta $ {\scriptsize ({\blue by  (\ref{h2})}).}
\end{itemize} 

On the other hand,   (\ref{gamma-theta-special})  follows by 
\begin{itemize}
\item[ ]$\hspace{0.38cm}(\mu_{B}\ot H)\co (B\ot \theta)\co \gamma $
\item [ ]$=((h^{-1}\ast h)\ot H)\co \delta_{H}$ {\scriptsize ({\blue by  coassociativity of  $\delta_{H}$})}
\item [ ]$=(u_{1}^{\varphi_{B}}\ot H)\co \delta_{H}$ {\scriptsize ({\blue by  (\ref{h1})})}
\item [ ]$=\nabla_{B\otimes H}^{\varphi_{B}}\co (\eta_{B}\ot H)$ {\scriptsize ({\blue by (\ref{nabla-eta})}).}
\end{itemize} 
and (\ref{gamma-theta-psi})  follows by 
\begin{itemize}
\item[ ]$\hspace{0.38cm}P_{\phi_{B}}$
\item [ ]$= ((\mu_{B}\circ (\mu_{B}\ot h^{-1}))\ot H)\co  (h\ot ((P_{\varphi_{B}}\ot H)\co (H\ot c_{H,B})\co (\delta_{H}\ot B)))\co  (\delta_{H}\ot B) $ {\scriptsize ({\blue by  (\ref{h4}) and }}
\item[ ]$\hspace{0.38cm}${\scriptsize {\blue coassociativity of $\delta_{H}$})}
\item [ ]$= (\mu_{B}\ot H)\co (\mu_{B}\ot \gamma)\co (B\ot P_{\varphi_{B}})\co (\theta\ot B)$ {\scriptsize ({\blue by  coassociativity of $\delta_{H}$ and the naturality of $c$}).}
\end{itemize} 

Finally, (\ref{gamma-theta-sigma}) holds because 
\begin{itemize}
\item[ ]$\hspace{0.38cm}F_{\tau}$
\item [ ]$=((\mu_{B}\co (B\ot h^{-1})\co \mu_{B\ot_{\varphi_{B}}^{\sigma}H}\co (((h\ot H)\co \delta_{H})\ot  ((h\ot H)\co \delta_{H})))\ot \mu_{H})\co \delta_{H^{\ot 2}} $ {\scriptsize ({\blue by (\ref{h5})})}

\item [ ]$= ((\mu_{B}\co (\mu_{B}\ot h^{-1}))\ot H)\co (\mu_{B}\ot ((F_{\sigma}\ot \mu_{H})\co \delta_{H^{\ot 2}}))\co (B\ot P_{\varphi_{B}}\ot H)\co (\theta\ot \theta)$ {\scriptsize ({\blue by  coassociativity}}
\item[ ]$\hspace{0.38cm}${\scriptsize {\blue of $\delta_{H}$ and the naturality of $c$})}
\item [ ]$= (\mu_{B}\ot H)\co (\mu_{B}\ot \gamma)\co (\mu_{B}\ot  F_{\sigma})\co (B\ot P_{\varphi_{B}}\ot H)\co (\theta\ot \theta) $ {\scriptsize ({\blue by (\ref{delta-sigmaHB})})}
\item [ ]$= (\mu_{B}\ot H)\co (B\ot \gamma)\co \mu_{B\ot_{\varphi_{B}}^{\sigma}H}\co (\theta\ot \theta)$ {\scriptsize ({\blue by the definition  of $\mu_{B\ot_{\varphi_{B}}^{\sigma}H}$}).}
\end{itemize} 

\end{proof}

\begin{rem}
\label{hh-1}
{\rm 
Note that, in the conditions of (ii) of Theorem \ref{equiv-hh}, composing with $H\ot \eta_{B}$ in 
(\ref{h4}), we obtain the identity
\begin{equation}
\label{h6}
h\ast h^{-1}=u_{1}^{\phi_{B}}.
\end{equation}
}
\end{rem}

\begin{defin}
{\rm 
\label{g-trans} Let $H$ be a  weak Hopf algebra and let $\varphi_{B}:H\ot B\rightarrow B$ be a measuring. We will say that the pair of morphisms $h,h^{-1}:H\rightarrow B$ is a gauge trasformation for $\varphi_{B}$ if they satisfy (\ref{h1}) and (\ref{h2}).

By the previous Theorem \ref{equiv-hh} we know that, under suitable conditions, equivalent weak crossed products are related by gauge transformations. After the next discussion, we should be able to secure that the converse is also true.
}
\end{defin}

\begin{apart}
\label{basic}
{\rm 
Let $H$ be a  weak Hopf algebra and let $\varphi_{B}:H\ot B\rightarrow B$ be a measuring. Let $(h,h^{-1})$ be a gauge transformation for $\varphi_{B}$ and let $\sigma:H\ot H\rightarrow B\ot H$ be a morphism satisfying the identity $\sigma\ast u_{2}^{\varphi_{B}}=\sigma$, the twisted condition (\ref{t-sigma}) and the 2-cocycle condition (\ref{c-sigma}). Suppose that $\nu$ is a preunit for the associated weak crosse product $\mu_{B\ot_{\varphi_{B}}^{\sigma}H}$.

Define $\theta$ and $\gamma$ as in (\ref{th-gh-1}), i.e., $\theta=(h\ot H)\co \delta_{H}$ and $\gamma=(h^{-1}\ot H)\co \delta_{H}$. Then $\theta$ and $\gamma$ are morphisms of right $H$-comodules. Also, by (\ref{nabla-nabla}), the coassociativity of $\delta_{H}$ and the condition of gauge transformation, we have that $\nabla_{B\ot H}^{\varphi_{B}}\co \theta=\theta$ and then  (\ref{gamma-theta-idemp}) holds. By similar arguments and the associativity of $\mu_{B}$ we obtain the equality 
\begin{equation}
\label{tn2}
(\mu_{B}\ot H)\co (B\ot \gamma)\co \nabla_{B\ot H}^{\varphi_{B}}=(\mu_{B}\ot H)\co (B\ot \gamma).
\end{equation}

Moreover, by the coassociativity of $\delta_{H}$ and the condition of gauge transformation we have 
$$
(\mu_{B}\ot H)\co (B\ot \theta)\co \gamma=(u^{\varphi_{B}}_1\ot H)\co \delta_{H}=\nabla_{B\ot H}\ot (\eta_{B}\ot H)
$$
and then (\ref{gamma-theta-special}) holds. As a consequence, we obtain
\begin{equation}
\label{tn4}
(\mu_{B}\ot H)\co (B\ot ((\mu_{B}\ot H)\co (B\ot \theta)\co \gamma))=\nabla_{B\ot H}^{\varphi_{B}}.
\end{equation}

Define $\varphi_{B}^{h}:H\ot B\rightarrow B$ by 
\begin{equation}
\label{vh}
\varphi_{B}^{h}=\mu_{B}\ot (\mu_{B}\ot h^{-1})\co (B\ot P_{\varphi_{B}}))\co (\theta\ot B)
\end{equation}
and $\sigma^{h}:H\ot H\rightarrow B$ by
\begin{equation}
\label{sh}
\sigma^{h}=\mu_{B}\co (B\ot h^{-1})\co \mu_{B\ot_{\varphi_{B}}^{\sigma}H}\co (\theta\ot 
\theta).
\end{equation}

Then, $\varphi_{B}^{h}$ is a measuring because 
\begin{itemize}
\item[ ]$\hspace{0.38cm}\mu_{B}\co (\varphi_{B}^{h}\ot \varphi_{B}^{h})\co (H\ot c_{H,B}\ot B)\co
(\delta_{H}\ot B\ot B) $

\item [ ]$=\mu_{B}\co (B\ot\mu_{B})\co  ((\mu_{B}\co (h\ot \varphi_{B})\co (\delta_{H}\ot B))\ot (\mu_{B}\co (u_{1}^{\varphi_{B}}\ot \varphi_{B})\co  (\delta_{H}\ot B))\ot h^{-1})\co (H\ot c_{H,B}\ot c_{H,B})$
\item[ ]$\hspace{0.38cm}\co (H\ot H\ot c_{H,B}\ot B)\co (((H\ot \delta_{H})\co \delta_{H})\ot B\ot B)  $ {\scriptsize ({\blue by naturality of $c$, coassociativity of  $\delta_{H}$, associativity}}
\item[ ]$\hspace{0.38cm}${\scriptsize {\blue  of $\mu_{B}$, and and (\ref{h1})})}

\item [ ]$=\mu_{B}\co (B\ot\mu_{B})\co (h\ot (\mu_{B}\co(\varphi_{B}\ot \varphi_{B})\co (H\ot c_{H,B}\ot B)\co
(\delta_{H}\ot B\ot B)) \ot h^{-1})\co (H\ot H\ot B\ot c_{H,B})$
\item[ ]$\hspace{0.38cm}\co (H\ot H\ot c_{H,B}\ot B)\co (((H\ot \delta_{H})\co \delta_{H})\ot B\ot B)$ {\scriptsize ({\blue by (\ref{nabla-fi}), coassociativity of  $\delta_{H}$ and associativity of $\mu_{B}$}}
\item[ ]$\hspace{0.38cm}${\scriptsize {\blue $\mu_{B}$})}

\item [ ]$=\varphi_{B}^{h}\co (H\ot \mu_{B})${\scriptsize ({\blue by and (b1) of Definition \ref{def} and naturality of  $c$}) } 
\end{itemize} 
and trivially, by the coassociativity of $\delta_{H}$ and the naturality of $c$, we have that 
$$P_{\varphi_{B}^{h}}=(\mu_{B}\ot H)\co (\mu_{B}\ot \gamma)\co (B\ot P_{\varphi_{B}})\co (\theta\ot B)$$ 

Therefore (\ref{gamma-theta-psi}) holds. Also, by the associativity of $\mu_{B}$, (\ref{nabla-nabla}) and the condition of gauge transformation
\begin{equation}
\label{gt1}
u_{1}^{\varphi_{B}^{h}}=h\ast h^{-1}
\end{equation}
holds. Then, as a consequence of the previous identity, we have that 
\begin{equation}
\label{tn5}
(\mu_{B}\ot H)\co (B\ot \gamma)\co \theta=(u_1^{{\varphi_{B}^h}}\ot H)\co \delta_{H}
\end{equation}
and  
\begin{equation}
\label{tn6}
(\mu_{B}\ot H)\co (B\ot ((\mu_{B}\ot H)\co (B\ot \gamma)\co \theta))=\nabla_{B\ot H}^{\varphi_{B}^h}.
\end{equation}
hold.

On the other hand for $F_{\sigma^{h}}$ we have 
\begin{itemize}
\item[ ]$\hspace{0.38cm}F_{\sigma^{h}} $
\item [ ]$=((\mu_{B}\co (\mu_{H}\ot h^{-1}))\ot H)\co (\mu_{B}\ot ((F_{\sigma}\ot \mu_{H})\co \delta_{H^{\ot 2}}))\co (B\ot P_{\varphi_{B}}\ot H)\co (\theta\ot \theta)$ {\scriptsize ({\blue by the naturality }}
\item[ ]$\hspace{0.38cm}${\scriptsize {\blue of  $c$ and the coassociativity of $\delta_{H}$})}

\item [ ]$=(\mu_{B}\ot H)\co (B\ot \gamma)\co \mu_{B\ot_{\varphi_{B}}^{\sigma}H}\co (\theta\ot \theta)$ {\scriptsize({\blue by (\ref{delta-sigmaHB})})} 
\end{itemize} 
and (\ref{gamma-theta-sigma}) also holds. Moreover, $ \sigma^{h}\ast u_{2}^{\varphi_{B}^{h}}=\sigma^{h} $, because 

\begin{itemize}
\item[ ]$\hspace{0.38cm} \sigma^{h}\ast u_{2}^{\varphi_{B}^{h}}$

\item [ ]$= \mu_{B}\co ((\mu_{B}\ot (\mu_{B} \ot h^{-1}))\ot (\varphi_{B}^{h}\co (H\ot \eta_{B})))\co (\mu_{B}\ot  ((F_{\sigma}\ot \mu_{H})\co \delta_{H^{\ot 2}}))\co (B\ot P_{\varphi_{B}}\ot H)\co (\theta\ot \theta)$ 
\item[ ]$\hspace{0.38cm}${\scriptsize ({\blue by the naturality of $c$, the coassociativity of  $\delta_{H}$ and the condition of morphism of right $H$-comodules for $\theta$})}

\item [ ]$=\mu_{B}\co ((\mu_{B}\ot (\mu_{B} \ot h^{-1}))\ot (\varphi_{B}^{h}\co (H\ot \eta_{B})))\co (\mu_{B}\ot  ((B\ot \delta_{H})\co F_{\sigma}))\co (B\ot P_{\varphi_{B}}\ot H)\co (\theta\ot \theta)$
\item[ ]$\hspace{0.38cm}${\scriptsize ({\blue by (\ref{delta-sigmaHB})})}

\item [ ]$= \mu_{B}\co (\mu_{B}\ot (h^{-1}\ast u_{1}^{\varphi_{B}^{h}}))\co (\mu_{B}\ot F_{\sigma})\co (B\ot P_{\varphi_{B}}\ot H)\co (\theta\ot \theta)$ {\scriptsize ({\blue by the associativity of $\mu_{B}$})}

\item [ ]$=\mu_{B}\co (\mu_{B}\ot (h^{-1}\ast h\ast h^{-1}))\co (\mu_{B}\ot F_{\sigma})\co (B\ot P_{\varphi_{B}}\ot H)\co (\theta\ot \theta)$ {\scriptsize ({\blue by (\ref{gt1})})}

\item [ ]$=\sigma^{h} $ {\scriptsize ({\blue by (\ref{h2}) and associativity of $\mu_{B}$}).}
\end{itemize}

By (i) of  Theorem \ref{teo-twis-cocy} to obtain that  $(B,H, P_{\varphi_{B}^{h}}, F_{\sigma^{h}})$ satisfies the twisted condition is enough  to prove that (\ref{t-sigma}) holds. Indeed, 

\begin{itemize}
\item[ ]$\hspace{0.38cm} \mu_{B}\co (B\ot \sigma^{h})\co (P_{\varphi_{B}^h}\ot H)\co (H\ot P_{\varphi_{B}^h})$

\item [ ]$=\mu_{B}\co (B\ot h^{-1})\co \mu_{B\ot_{\varphi_{B}}^{\sigma}H}\co (((\mu_{B}\ot H)\co (B\ot ((\mu_{B}\ot H)\co (B\ot \theta)\co \gamma)))\ot \theta)\co (\mu_{B}\ot H\ot H)$
\item[ ]$\hspace{0.38cm}\co (B\ot P_{\varphi_{B}}\ot H)\co (\theta\ot ((\mu_{B}\ot H)\co (\mu_{B}\ot \gamma)\co (B\ot P_{\varphi_{B}})\co (\theta\ot B))) $  {\scriptsize ({\blue by the associativity of $\mu_{B}$})}

\item [ ]$= \mu_{B}\co (B\ot h^{-1})\co \mu_{B\ot_{\varphi_{B}}^{\sigma}H}\co (\nabla_{B\ot H}^{\varphi_{B}}\ot \theta)\co (\mu_{B}\ot H\ot H)\co (B\ot P_{\varphi_{B}}\ot H) $
\item[ ]$\hspace{0.38cm}\co (\theta\ot ((\mu_{B}\ot H)\co (\mu_{B}\ot \gamma)\co (B\ot P_{\varphi_{B}})\co (\theta\ot B))) $ {\scriptsize ({\blue by  (\ref{tn4})})}

\item [ ]$=  \mu_{B}\co (B\ot h^{-1})\co \mu_{B\ot_{\varphi_{B}}^{\sigma}H}\co \co (\mu_{B}\ot H\ot \theta)\co (B\ot (\nabla_{B\ot H}^{\varphi_{B}}\co P_{\varphi_{B}})\ot H) $
\item[ ]$\hspace{0.38cm}\co (\theta\ot ((\mu_{B}\ot H)\co (\mu_{B}\ot \gamma)\co (B\ot P_{\varphi_{B}})\co (\theta\ot B))) $ {\scriptsize ({\blue by the associativity of $\mu_{B}$})}

\item [ ]$= \mu_{B}\co (B\ot h^{-1})\co \mu_{B\ot_{\varphi_{B}}^{\sigma}H}\co \co (\mu_{B}\ot H\ot \theta)\co (B\ot  P_{\varphi_{B}}\ot H) $
\item[ ]$\hspace{0.38cm}\co (\theta\ot ((\mu_{B}\ot H)\co (\mu_{B}\ot \gamma)\co (B\ot P_{\varphi_{B}})\co (\theta\ot B))) $  {\scriptsize ({\blue by  (\ref{nabla-psi})})}

\item [ ]$= \mu_{B}\co (\mu_{B}\ot h^{-1})\co (\mu_{B}\ot F_{\sigma})\co (B\ot ((\mu_{B}\ot H)\co (B\ot  P_{\varphi_{B}})\co (P_{\varphi_{B}}\ot B))\ot H)$
\item[ ]$\hspace{0.38cm}\co   (\theta\ot ((\mu_{B}\ot \theta)\co (\mu_{B}\ot \gamma)\co (B\ot P_{\varphi_{B}})\co (\theta\ot B)))$  {\scriptsize ({\blue by the definition of $\mu_{B\ot_{\varphi_{B}}^{\sigma}H}$})}

\item [ ]$=\mu_{B}\co (\mu_{B}\ot h^{-1})\co (\mu_{B}\ot F_{\sigma})\co (B\ot P_{\varphi_{B}}\ot H)$
\item[ ]$\hspace{0.38cm}\co (\theta\ot ((\mu_{B}\ot H)\co (B\ot ((\mu_{B}\ot H)\co (B\ot ((\mu_{B}\ot H)\co (B\ot \theta)\co \gamma))\co  P_{\varphi_{B}}))\co (\theta\ot B))$ {\scriptsize ({\blue by  (\ref{wmeas-wcp})}}
\item[ ]$\hspace{0.38cm}${\scriptsize {\blue  and the associativity of $\mu_{B}$ })}

\item [ ]$=\mu_{B}\co (\mu_{B}\ot h^{-1})\co (\mu_{B}\ot F_{\sigma})\co (B\ot P_{\varphi_{B}}\ot H)\co (\theta\ot ((\mu_{B}\ot H)\co (B\ot (\nabla_{B\ot H}^{\varphi_{B}}\co  P_{\varphi_{B}}))\co (\theta\ot B)))$ 
\item[ ]$\hspace{0.38cm}${\scriptsize ({\blue by (\ref{tn4})})}

\item [ ]$=\mu_{B}\co (\mu_{B}\ot h^{-1})\co (\mu_{B}\ot ((\mu_B\ot H)\co (B\ot F_{\sigma})\co (P_{\varphi_{B}}\ot H)\co (H\ot P_{\varphi_{B}})) \co (B\ot P_{\varphi_{B}} \ot H\ot B)\co (\theta\ot \theta \ot B)$ 
\item[ ]$\hspace{0.38cm}${\scriptsize ({\blue by (\ref{nabla-psi}), (\ref{wmeas-wcp}) and the associativity of $\mu_{B}$})}

\item [ ]$= \mu_{B}\co (\mu_{B}\ot h^{-1})\co (B\ot P_{\varphi_{B}})\co ((\mu_{B\ot_{\varphi_{B}}^{\sigma}H}\co (\theta\ot \theta))\ot B)$ {\scriptsize ({\blue by   (\ref{twcp}) })}

\item [ ]$=  \mu_{B}\co (\mu_{B}\ot h^{-1})\co (B\ot P_{\varphi_{B}})\co ((\nabla_{B\ot H}^{\varphi_{B}}\co \mu_{B\ot_{\varphi_{B}}^{\sigma}H}\co (\theta\ot \theta))\ot B)$ {\scriptsize ({\blue by   (\ref{norz})})}

\item [ ]$=\mu_{B}\co (\mu_{B}\ot h^{-1})\co ((\mu_{B}\co (B\ot (h^{-1}\ast h)))\ot P_{\varphi_{B}})\co (B\ot  \delta_{H}\ot B)\co ((\mu_{B\ot_{\varphi_{B}}^{\sigma}H}\co (\theta\ot \theta))\ot B) $ 
\item[ ]$\hspace{0.38cm}${\scriptsize ({\blue by   (\ref{nabla-nabla}) and the condition of gauge transformation})}

\item [ ]$=\mu_{B}\co (B\ot P_{\varphi_{B}^h})\co (F_{\sigma^h}\ot B)$ {\scriptsize ({\blue by the associativity of $\mu_{B}$ and the coassociativity of $\delta_{H}$}).}
\end{itemize}

Also, by (ii) of  Theorem \ref{teo-twis-cocy}, to obtain that  $(B,H, P_{\varphi_{B}^{h}}, F_{\sigma^{h}})$ satisfies the cocycle condition is enough  to prove that (\ref{c-sigma}) holds. Indeed, 

\begin{itemize}
\item[ ]$\hspace{0.38cm} \mu_{B}\co (B\ot \sigma^h)\co (P_{\varphi_{B}^h}\ot H)\co (H\ot F_{\sigma^h})$

\item [ ]$= \mu_{B}\co (B\ot h^{-1})\co \mu_{B\ot_{\varphi_{B}}^{\sigma}H}\co (\mu_{B}\ot H\ot \theta)\co (B\ot \mu_{B}\ot H\ot H)$
\item[ ]$\hspace{0.38cm}\co (B\ot B\ot ((\mu_{B}\ot H)\co (B\ot ((\mu_{B}\ot H)\co (B\ot \theta)\co \gamma))\co P_{\varphi_{B}})\ot H)\co (B\ot P_{\varphi_{B}}\ot\gamma) $
\item[ ]$\hspace{0.38cm} \co (\theta\ot (\mu_{B\ot_{\varphi_{B}}^{\sigma}H}\co (\theta\ot \theta))) $  {\scriptsize ({\blue by the associativity of $\mu_{B}$})}

\item [ ]$= \mu_{B}\co (B\ot h^{-1})\co \mu_{B\ot_{\varphi_{B}}^{\sigma}H}\co (\mu_{B}\ot H\ot \theta)\co (\mu_{B}\ot (\nabla_{B\ot H}^{\varphi_{B}}\co  P_{\varphi_{B}})\ot H)\co  (B\ot P_{\varphi_{B}}\ot\gamma)  $
\item[ ]$\hspace{0.38cm} \co (\theta\ot (\mu_{B\ot_{\varphi_{B}}^{\sigma}H}\co (\theta\ot \theta))) $ {\scriptsize ({\blue by  (\ref{tn4}) and the associativity of $\mu_{B}$})}

\item [ ]$=  \mu_{B}\co (B\ot h^{-1})\co \mu_{B\ot_{\varphi_{B}}^{\sigma}H}\co (\mu_{B}\ot H\ot \theta)\co (\mu_{B}\ot  P_{\varphi_{B}}\ot H)\co  (B\ot P_{\varphi_{B}}\ot\gamma)  $
\item[ ]$\hspace{0.38cm} \co (\theta\ot (\mu_{B\ot_{\varphi_{B}}^{\sigma}H}\co (\theta\ot \theta))) $ {\scriptsize ({\blue by by  (\ref{nabla-psi})})}

\item [ ]$=\mu_{B}\co (\mu_{B}\ot h^{-1})\co  (\mu_{B}\ot F_{\sigma})\co  (\mu_{B}\ot  P_{\varphi_{B}}\ot H)$
\item[ ]$\hspace{0.38cm}\co (\mu_{B}\ot  P_{\varphi_{B}}\ot ((\mu_{B}\ot H)\co (B\ot ((\mu_{B}\ot H)\co (B\ot \theta)\co \gamma))\co F_{\sigma}))$
\item[ ]$\hspace{0.38cm} \co (B\ot P_{\varphi_{B}} \ot P_{\varphi_{B}}\ot H)\co (\theta\ot \theta\ot \theta)$  {\scriptsize ({\blue by  (\ref{wmeas-wcp}), the definition of $\mu_{B\ot_{\varphi_{B}}^{\sigma}H}$ and the associativity of $\mu_{B}$})}

\item [ ]$=\mu_{B}\co (\mu_{B}\ot h^{-1})\co (\mu_{B}\ot ((\mu_{B}\ot H)\co (B\ot F_{\sigma})\co (P_{\varphi_{B}}\ot H)\co (H\ot (\nabla_{B\ot H}^{\varphi_{B}}\co F_{\sigma}))))$
\item[ ]$\hspace{0.38cm}\co (\mu_{B}\ot P_{\varphi_{B}}\ot H\ot H)\co (B\ot P_{\varphi_{B}} \ot P_{\varphi_{B}}\ot H)\co (\theta\ot \theta\ot \theta)  $  {\scriptsize ({\blue by (\ref{tn4}) and the associativity of $\mu_{B}$})}

\item [ ]$=\mu_{B}\co (\mu_{B}\ot h^{-1})\co (B\ot \mu_{B}\ot H)\co  (B\ot B\ot ((\mu_{B}\ot H)\co (B\ot F_{\sigma})\co (P_{\varphi_{B}}\ot H)\co (H\ot F_{\sigma})))$
\item[ ]$\hspace{0.38cm}\co (\mu_{B}\ot P_{\varphi_{B}}\ot H\ot H)\co (B\ot P_{\varphi_{B}} \ot P_{\varphi_{B}}\ot H)\co (\theta\ot \theta\ot \theta)$ {\scriptsize ({\blue by  (\ref{sigmaHB-2}) and the associativity of $\mu_{B}$ })}

\item [ ]$= \mu_{B}\co (\mu_{B}\ot h^{-1})\co (B\ot \mu_{B}\ot H)\co  (\mu_{B}\ot B\ot F_{\sigma})$
\item[ ]$\hspace{0.38cm}\co (B\ot B\ot ((\mu_{B}\ot H)\co (B\ot F_{\sigma})\co (P_{\varphi_{B}}\ot H)\co (H\ot P_{\varphi_{B}}))\ot H)\co (B\ot P_{\varphi_{B}}\ot H\ot \theta)\co (\theta\ot \theta\ot H)$ 
\item[ ]$\hspace{0.38cm}${\scriptsize ({\blue by (\ref{cwcp}) and the associativity of $\mu_{B}$})}

\item [ ]$=\mu_{B}\co (B\ot h^{-1})\co \mu_{B\ot_{\varphi_{B}}^{\sigma}H} \co ((\mu_{B\ot_{\varphi_{B}}^{\sigma}H}\co (\theta\ot \theta))\ot \theta)$  {\scriptsize ({\blue by (\ref{twcp}) and the associativity of $\mu_{B}$})}

\item [ ]$=\mu_{B}\co (B\ot h^{-1})\co \mu_{B\ot_{\varphi_{B}}^{\sigma}H} \co ((\nabla_{B\ot H}^{\varphi_{B}}\co \mu_{B\ot_{\varphi_{B}}^{\sigma}H}\co (\theta\ot \theta))\ot \theta)  $ {\scriptsize ({\blue by   (\ref{norz})})}

\item [ ]$=\mu_{B}\co (B\ot h^{-1})\co \mu_{B\ot_{\varphi_{B}}^{\sigma}H} \co (((\mu_{B}\ot H)\co (B\ot ((u_{1}^{\varphi_{B}}\ot H)\co \delta_{H}))) \co \mu_{B\ot_{\varphi_{B}}^{\sigma}H}\co (\theta\ot \theta))\ot \theta)$ 
\item[ ]$\hspace{0.38cm}${\scriptsize ({\blue by   (\ref{nabla-nabla})})}

\item [ ]$=\mu_{B}\co (B\ot h^{-1})\co \mu_{B\ot_{\varphi_{B}}^{\sigma}H} \co (((\mu_{B}\ot H)\co (B\ot ((\mu_{B}\ot H)\co (B\ot \theta)\co \gamma)) \co \mu_{B\ot_{\varphi_{B}}^{\sigma}H}\co (\theta\ot \theta))\ot \theta)   $ 
\item[ ]$\hspace{0.38cm}${\scriptsize ({\blue by   (\ref{tn4}) and the condition of gauge transformation})}

\item [ ]$= \mu_{B}\co (B\ot \sigma^h)\co (F_{\sigma^h}\ot B)$ {\scriptsize ({\blue by the associativity of $\mu_{B}$}).}
\end{itemize}

If we define $\nu^h$ by $\nu^h=(\mu_{B}\ot H)\co (B\ot \gamma)\co \nu$ we have that (\ref{gamma-theta-preunit}) holds trivially. Moreover $\nu^h$ is a preunit for $(B\ot H, \mu_{B\ot_{\varphi_{B}^h}^{\sigma^h}H})$ because (\ref{pre1-wcp}), (\ref{pre2-wcp}) and (\ref{pre3-wcp}) hold. Indeed, (\ref{pre1-wcp}) follows by
\begin{itemize}
\item[ ]$\hspace{0.38cm} \mu_{B}\co (B\ot F_{\sigma^h})\co (P_{\varphi_{B}^h}\ot H)\co (H\ot \nu^h)$

\item [ ]$=(\mu_{B}\ot H)\co (B\ot \gamma)\co \mu_{B\ot_{\varphi_{B}}^{\sigma}H}\co (\mu_{B}\ot H\ot \theta)$
\item[ ]$\hspace{0.38cm}\co (B\ot  ((\mu_{B}\ot H)\co (B\ot ((\mu_{B}\ot H)\co (B\ot \theta)\co \gamma))\co P_{\varphi_{B}})\ot H)\co (\theta \ot ((\mu_{B}\ot H)\co (B\ot \gamma)\co \nu))$  
\item[ ]$\hspace{0.38cm}${\scriptsize ({\blue by the associativity of $\mu_{B}$})}

\item [ ]$= (\mu_{B}\ot H)\co (\mu_{B}\ot \gamma)\co (\mu_{B}\ot F_{\sigma})\co (B\ot ((\mu_{B}\ot H)\co (B\ot P_{\varphi_{B}}))\ot H)\co  (B\ot  (\nabla_{B\ot H}^{\varphi_{B}}\co P_{\varphi_{B}})\ot \theta)$
\item[ ]$\hspace{0.38cm}\co (\theta \ot ((\mu_{B}\ot H)\co (B\ot \gamma)\co \nu))$  {\scriptsize ({\blue by  (\ref{tn4}), the definition of $\mu_{B\ot_{\varphi_{B}}^{\sigma}H}$ and the associativity of $\mu_{B}$})}

\item [ ]$=(\mu_{B}\ot H)\co (\mu_{B}\ot \gamma)\co (\mu_{B}\ot F_{\sigma})\co 
(B\ot ((\mu_{B}\ot H)\co (B\ot P_{\varphi_{B}})\co (P_{\varphi_{B}}\ot B))\ot \theta)$
\item[ ]$\hspace{0.38cm}\co (\theta \ot ((\mu_{B}\ot H)\co (B\ot \gamma)\co \nu))${\scriptsize ({\blue by (\ref{nabla-psi})})}

\item [ ]$= (\mu_{B}\ot H)\co (\mu_{B}\ot \gamma)\co (\mu_{B}\ot F_{\sigma})\co 
(B\ot P_{\varphi_{B}}\ot H) \co (\theta \ot ((\mu_{B}\ot H)\co (B\ot ((\mu_{B}\ot H)\co (B\ot \theta)\co \gamma))\co \nu)))$ 
\item[ ]$\hspace{0.38cm}${\scriptsize ({\blue by  (\ref{wmeas-wcp}) and the associativity of $\mu_{B}$})}

\item [ ]$=(\mu_{B}\ot H)\co (\mu_{B}\ot \gamma)\co (\mu_{B}\ot F_{\sigma})\co 
(B\ot P_{\varphi_{B}}\ot H) \co (\theta \ot (\nabla_{B\ot H}^{\varphi_{B}}\co \nu)))  $ {\scriptsize ({\blue by (\ref{tn4})})}

\item [ ]$=(\mu_{B}\ot H)\co (B \ot \gamma)\co  \nabla_{B\ot H}^{\varphi_{B}}\co \mu_{B\ot_{\varphi_{B}}^{\sigma}H}\co ((\nabla_{B\ot H}^{\varphi_{B}}\co\theta) \ot (\nabla_{B\ot H}^{\varphi_{B}}\co \nu))) $ {\scriptsize ({\blue by  (\ref{gamma-theta-idemp}) and (\ref{tn2})})}

\item [ ]$= (\mu_{B}\ot H)\co (B \ot \gamma)\co \nabla_{B\ot H}^{\varphi_{B}}\co\theta $ {\scriptsize ({\blue by the condition of unit for $p_{B\ot H}^{\varphi_{B}}\co \nu$})}

\item [ ]$= (\mu_{B}\ot H)\co (B \ot \gamma)\co \theta $ {\scriptsize ({\blue by (\ref{gamma-theta-idemp})})}

\item [ ]$= (u_{1}^{\varphi_{B}^{h}}\ot H)\co \delta_{H}$ {\scriptsize ({\blue by   (\ref{tn5})})}

\item [ ]$=\nabla_{B\ot H}^{\varphi_{B}^h} \co (\eta_{B}\ot H)$ {\scriptsize ({\blue by (\ref{tn5}) and the properties of $\eta_{B}$}).}
\end{itemize}

Also, (\ref{pre2-wcp}) holds because 

\begin{itemize}

\item[ ]$\hspace{0.38cm} \mu_{B}\co (B\ot F_{\sigma^h})\co (\nu^h\ot H)$

\item [ ]$=(\mu_{B}\ot H)\co (B\ot \gamma)\co \mu_{B\ot_{\varphi_{B}}^{\sigma}H}\co (((\mu_{B}\ot H)\co (B\ot ((\mu_{B}\ot H)\co (B\ot \theta)\co \gamma))\co \nu)\ot \theta) $  
\item[ ]$\hspace{0.38cm}${\scriptsize ({\blue by the associativity of $\mu_{B}$})}

\item [ ]$=(\mu_{B}\ot H)\co (B \ot \gamma)\co  \nabla_{B\ot H}^{\varphi_{B}}\co \mu_{B\ot_{\varphi_{B}}^{\sigma}H}\co ((\nabla_{B\ot H}^{\varphi_{B}}\co\theta) \ot (\nabla_{B\ot H}^{\varphi_{B}}\co \nu))) $ {\scriptsize ({\blue by  (\ref{gamma-theta-idemp}) and (\ref{tn2})})}

\item [ ]$= (\mu_{B}\ot H)\co (B \ot \gamma)\co \nabla_{B\ot H}^{\varphi_{B}}\co\theta $ {\scriptsize ({\blue by the condition of unit for $p_{B\ot H}^{\varphi_{B}}\co \nu$})}

\item [ ]$= (\mu_{B}\ot H)\co (B \ot \gamma)\co \theta $ {\scriptsize ({\blue by (\ref{gamma-theta-idemp})})}

\item [ ]$= (u_{1}^{\varphi_{B}^{h}}\ot H)\co \delta_{H}$ {\scriptsize ({\blue by   (\ref{tn5})})}

\item [ ]$=\nabla_{B\ot H}^{\varphi_{B}^h} \co (\eta_{B}\ot H)$ {\scriptsize ({\blue by (\ref{tn5}) and the properties of $\eta_{B}$}).}
\end{itemize}

Finally, 

\begin{itemize}

\item[ ]$\hspace{0.38cm} \mu_{B}\co (B\ot P_{\varphi_{B}^h} )\co (\nu^h\ot H)$

\item [ ]$=(\mu_{B}\ot H)\co (\mu_{B}\ot \gamma)\co  (B\ot P_{\varphi_{B}})\co (((\mu_{B}\ot H)\co (B\ot ((\mu_{B}\ot H)\co (B\ot \theta)\co \gamma))\co \nu)\ot B) $  
\item[ ]$\hspace{0.38cm}${\scriptsize ({\blue by the associativity of $\mu_{B}$})}

\item [ ]$=(\mu_{B}\ot H)\co (\mu_{B}\ot \gamma)\co  (B\ot P_{\varphi_{B}})\co ((\nabla_{B\ot H}^{\varphi_{B}} \co\nu)\ot B)$ {\scriptsize ({\blue by (\ref{tn4})})}

\item [ ]$= (\mu_{B}\ot H)\co (\mu_{B}\ot \gamma)\co  (B\ot P_{\varphi_{B}})\co (\nu\ot B)$ {\scriptsize ({\blue by (\ref{pre4-wcp})})}

\item [ ]$= (\mu_{B}\ot H)\co (\mu_{B} \ot \gamma)\co (B\ot \nu) $ {\scriptsize ({\blue by (\ref{pre3-wcp})})}

\item [ ]$=(\mu_{B}\ot H)\co (B\ot \nu^h)$ {\scriptsize ({\blue by the associativity of $\mu_{B}$}).}
\end{itemize}
and (\ref{pre3-wcp}) holds.

Therefore, as a consequence of the previous facts, we have a theorem that generalizes to the monoidal setting \cite[Theorem 5.4]{ana1}.
}
\end{apart}

\begin{teo}
\label{prin27}
Let $H$ be a  weak Hopf algebra, let  $\varphi_{B},\; \phi_{B}:H\ot B\rightarrow B$ be  measurings, and let $\sigma,\; \tau:H\ot H\rightarrow B$  be  morphisms such that $\sigma\ast u_{2}^{\varphi_{B}}=\sigma$, $\tau\ast u_{2}^{\phi_{B}}=\tau$. Assume that $\sigma$, $\tau$ satisfy the twisted  condition (\ref{t-sigma}), the 2-cocycle condition
(\ref{c-sigma}) and suppose that $\nu$ is a preunit for $\mu_{B\ot_{\varphi_{B}}^{\sigma}H}$, and $u$ is a preunit for $\mu_{B\ot_{\phi_{B}}^{\tau}H}$. The weak crossed products  $(B\ot H, \mu_{B\ot_{\varphi_{B}}^{\sigma}H})$ and $(B\ot H, \mu_{B\ot_{\phi_{B}}^{\tau}H})$ are equivalent if and only if there exists a gauge transformation $(h,h^{-1})$ for $\varphi_{B}$ such that 
$\phi_{B}=\varphi_{B}^h$, $\tau=\sigma^h$ and $u=\nu^{h}$.
\end{teo}

\section{Regular morphisms}

\begin{defin}
\label{reg-1}
{\rm Let $H$ be a  weak Hopf algebra and
$\varphi_{B}$ be a measuring. With $Reg_{\varphi_{B}}(H,B)$ we will denote the set of regular morphisms between $H$ and $B$, i.e., a morphism  $h:H\rightarrow B$ is a regular morphism if there exists a morphism $h^{-1}:H\rightarrow B$, called the convolution inverse of $h$,  such that the pair $(h, h^{-1})$ is a gauge transformation for $\varphi_{B}$ and  
\begin{equation}
\label{h-h-1-varphi} h\ast h^{-1}=u_{1}^{\varphi_{B}},
\end{equation}
holds. Then, $Reg_{\varphi_{B}}(H,B)$, with the convolution as a product, is a group with unit $u_{1}^{\varphi_{B}}$. 
}
\end{defin}

\begin{apart}
{\rm 
Let ${\mathcal P}_{\varphi_{B}}$ be the set of all pairs $(\phi_{B},\tau),$
where:
\begin{itemize}
\item[(i)]  The morphism $\phi_{B}:H\ot B\rightarrow B$ is a measuring satisfying $u_{1}^{\phi_{B}}=u_{1}^{\varphi_{B}}.$ 
\item[(ii)] The morphism $\tau:H\ot H\rightarrow B$ is such that 
$\tau=\tau\ast u_{2}^{\phi_{B}}$ and  the associated quadruple  $(B, H, P_{\phi_{B}}, F_{\tau})$ satisfies the twisted condition and  the cocycle condition.
\item[(iii)] The associated weak crossed product  $(B\ot H, \mu_{B\ot_{\phi_{B}}^{\tau}H})$ admits a preunit $\upsilon$. 
\end{itemize}

By the results proved in the previous section we know that $Reg_{\varphi_{B}}(H,B)$ acts on ${\mathcal P}_{\varphi_{B}}$. The action
$$R:Reg_{\varphi_{B}}(H,B)\times {\mathcal P}_{\varphi_{B}}\rightarrow {\mathcal P}_{\varphi_{B}},$$
is defined by 
\begin{equation}
\label{reg-acts}
R(h,(\phi_{B},\tau))=(\phi_{B}^{h},\tau^{h}).
\end{equation}
}
\end{apart}

\begin{prop}
\label{lemma119}
Let $(B,\varphi_{B})$,  $(B,\phi_{B})$ be a left weak $H$-module algebra and let $h:H\rightarrow B$ be a morphism such that $h\ast u_{1}^{\varphi_{B}}=h=u_{1}^{\phi_{B}}\ast h$. Then, the following assertions are equivalent:
\begin{itemize} 
\item[(i)] $h\co \eta_{H}=\eta_{B}$.
\item[(ii)] $h\co \Pi_{H}^{L}=u_{1}^{\phi_{B}}$.
\item[(iii)] $h\co \overline{\Pi}_{H}^{L}=u_{1}^{\varphi_{B}}$.
\end{itemize}

Moreover, if $h\ast \eta_{H}=\eta_{B}$, the  identity 
\begin{equation}
\label{equ-eta}
(\mu_{B}\ot H)\co (B\ot ((h\ot H)\co \delta_{H}\co \eta_{H}))=\nabla_{B\otimes H}^{\varphi_{B}}\co (B\ot \eta_{H}), 
\end{equation}
holds.

Also, if $g:H\rightarrow B$ is a morphism such that $g\ast u_{1}^{\phi_{B}}=g=u_{1}^{\varphi_{B}}\ast g$, the following assertions are equivalent:
\begin{itemize} 
\item[(iv)] $g\co \eta_{H}=\eta_{B}$.
\item[(v)] $g\co \Pi_{H}^{L}=u_{1}^{\varphi_{B}}$.
\item[(vi)] $g\co \overline{\Pi}_{H}^{L}=u_{1}^{\phi_{B}}$.
\end{itemize}

Then, if $g\ast \eta_{H}=\eta_{B}$, the  identity
\begin{equation}
\label{equ-eta-1}
(\mu_{B}\ot H)\co (B\ot ((g\ot H)\co \delta_{H}\co \eta_{H}))=\nabla_{B\otimes H}^{\phi_{B}}\co (B\ot \eta_{H})
\end{equation}
holds.

Moreover, if there exists $h^{-1}:H\rightarrow B$ such that $(h,h^{-1})$ is a gauge transformation for $\varphi_{B}$ and  $h\ast h^{-1}=u_{1}^{\phi_{B}}$ holds,  we have $h\ast \eta_{H}=\eta_{B}$ iff $h^{-1}\ast \eta_{H}=\eta_{B}.$
\end{prop}

\begin{proof}  By the properties of $\eta_{H}$ and (b2) of Definition \ref{def}, we obtain that (ii)$\Rightarrow$(i), and (iii)$\Rightarrow$(i). Also, (i)$\Rightarrow$(ii) holds because, 

\begin{itemize}
\item[ ]$\hspace{0.38cm}h\co \Pi_{H}^{L} $
\item [ ]$=(u_{1}^{\phi_{B}}\ast h)\co \Pi_{H}^{L}  $ {\scriptsize ({\blue by  $h=u_{1}^{\phi_{B}}\ast h$})}
\item [ ]$=\mu_{B}\co (u_{1}^{\phi_{B}}\ot h)\co ((((\varepsilon_{H}\co \mu_{H})\ot H)\co (H\ot c_{H,H})\co (\delta_{H}\ot H))\ot H)\co (H\ot c_{H,H})\co ((\delta_{H}\co \eta_{H})\ot H)$ 
\item[ ]$\hspace{0.38cm}${\scriptsize ({\blue by  naturality of $c$ and coassociativity of $\delta_{H}$})}
\item [ ]$=\mu_{B}\co  ((u_{1}^{\phi_{B}}\co \mu_{H}\co (H\ot \Pi_{H}^{L}))\ot h)\co (H\ot c_{H,H})\co ((\delta_{H}\co \eta_{H})\ot H)$ {\scriptsize ({\blue by (\ref{d-pi1})})}

\item [ ]$=\mu_{B}\co  ((\phi_{B}\co (H\ot u_{1}^{\phi_{B}}))\ot h)\co (H\ot c_{H,H})\co ((\delta_{H}\co \eta_{H})\ot H)$ {\scriptsize ({\blue by (b3) of Definition \ref{def} and (\ref{B6})})}

\item [ ]$=\mu_{B}\co  ((\phi_{B}\co (\overline{\Pi}_{H}^{L}\ot B))\ot h)\co (H\ot c_{H,H})\co ((\delta_{H}\co \eta_{H})\ot u_{1}^{\phi_{B}})$ {\scriptsize ({\blue by the naturality of $c$ and (\ref{edpi})})}

\item [ ]$=\mu_{B}\co  ((\mu_{B}\co c_{B,B}\co
(u_{1}^{\phi_{B}}\ot B))\ot h)\co (H\ot c_{H,H})\co ((\delta_{H}\co \eta_{H})\ot u_{1}^{\phi_{B}})$ {\scriptsize ({\blue by (\ref{B2})})}

\item [ ]$=\mu_{B}\co c_{B,B}\co (((u_{1}^{\phi_{B}}\ast h)\co \eta_{H})\ot u_{1}^{\phi_{B}})
$ {\scriptsize ({\blue by naturality of $c$ and associativity of $\mu_{B}$})}

\item [ ]$=\mu_{B}\co c_{B,B}\co (( h\co \eta_{H})\ot u_{1}^{\phi_{B}})
$ {\scriptsize ({\blue by $h=u_{1}^{\phi_{B}}\ast h$})}

\item [ ]$= u_{1}^{\phi_{B}}
$ {\scriptsize ({\blue by (i), naturality of $c$, and properties of $\eta_{B}$}).}

\end{itemize}  

On the other hand, (i)$\Rightarrow$(iii) because 

\begin{itemize}
\item[ ]$\hspace{0.38cm}h\co \overline{\Pi}_{H}^{L}$
\item [ ]$=(h\ast u_{1}^{\varphi_{B}})\co \overline{\Pi}_{H}^{L}  $ {\scriptsize ({\blue by  $h=h\ast u_{1}^{\varphi_{B}}$})}
\item [ ]$=\mu_{B}\co (h\ot u_{1}^{\varphi_{B}})\co (H\ot ((H\ot (\varepsilon_{H}\co \mu_{H}))\co (\delta_{H}\ot H)))\co ((\delta_{H}\co \eta_{H})\ot H)$ {\scriptsize ({\blue by the  coassociativity of $\delta_{H}$})}

\item [ ]$=\mu_{B}\co (h\ot (u_{1}^{\varphi_{B}}\co \mu_{H}\co (H\ot \overline{\Pi}_{H}^{L})))\co 
((\delta_{H}\co \eta_{H})\ot H)$ {\scriptsize ({\blue by (\ref{d-pi21})})}

\item [ ]$=\mu_{B}\co (h\ot \varphi_{B})\co 
((\delta_{H}\co \eta_{H})\ot u_{1}^{\varphi_{B}})$ {\scriptsize ({\blue by (b3)  of Definition  \ref{def} and (\ref{B5})})}

\item [ ]$=\mu_{B}\co (h\ot (\varphi_{B}\co (\Pi_{H}^{L}\ot B))\co 
((\delta_{H}\co \eta_{H})\ot u_{1}^{\varphi_{B}})$ {\scriptsize ({\blue by (\ref{edpi})})}

\item [ ]$=\mu_{B}\co (h\ot (\mu_{B}\co (u_{1}^{\varphi_{B}}\ot B))\co 
((\delta_{H}\co \eta_{H})\ot u_{1}^{\varphi_{B}})$ {\scriptsize ({\blue  by  (\ref{B1})})}

\item [ ]$=\mu_{B}\co  (((h\ast u_{1}^{\varphi_{B}})\co \eta_{H})\ot u_{1}^{\varphi_{B}})
$ {\scriptsize ({\blue by associativity of $\mu_{B}$})}

\item [ ]$=\mu_{B}\co (( h\co \eta_{H})\ot u_{1}^{\varphi_{B}})
$ {\scriptsize ({\blue by $h=h\ast u_{1}^{\varphi_{B}}$})}

\item [ ]$= u_{1}^{\varphi_{B}}
$ {\scriptsize ({\blue by (i), and properties of $\eta_{B}$}).}

\end{itemize}  

As a consequence of these equivalences, we obtain (\ref{equ-eta}) because 

 \begin{itemize}
\item[ ]$\hspace{0.38cm}(\mu_{B}\ot H)\co (B\ot ((h\ot H)\co \delta_{H}\co \eta_{H}))$
\item [ ]$= (\mu_{B}\ot H)\co (B\ot (((h\co \overline{\Pi}_{H}^{L})\ot H)\co \delta_{H}\co \eta_{H}))$ {\scriptsize ({\blue by (\ref{edpi})})}
\item [ ]$=\nabla_{B\otimes H}^{\varphi_{B}}\co (B\ot \eta_{H})$ {\scriptsize ({\blue by (iii) and (\ref{nabla-nabla})})}
\end{itemize}

The proof for the equivalences associated to $g$ and (\ref{equ-eta-1}) are similar and we leave the details to the reader. 

Finally, assume that there exists $h^{-1}:H\rightarrow B$ such that $(h,h^{-1})$ is a gauge transformation for $\varphi_{B}$ and  $h\ast h^{-1}=u_{1}^{\phi_{B}}$ holds. If $h\co \eta_{H}=\eta_{B}$, 
we have
\begin{itemize}
\item[ ]$\hspace{0.38cm}h^{-1}\co\eta_{H}$
\item [ ]$=(u_{1}^{\varphi_{B}}\ast h^{-1})\co \eta_{H}$ {\scriptsize ({\blue by $u_{1}^{\varphi_{B}}\ast h^{-1}=h^{-1}$})}
\item [ ]$=((h\co \overline{\Pi}_{H}^{L})\ast h^{-1})\co \eta_{H} ${\scriptsize ({\blue by (ii)})}
\item [ ]$=u_{1}^{\phi_{B}}\co \eta_{H}$ {\scriptsize ({\blue by (\ref{edpi}) and $u_{1}^{\phi_{B}}=h\ast h^{-1}$})}
\item [ ]$= \eta_{B}$ {\scriptsize ({\blue by (b2) of Definition \ref{def}}).}
\end{itemize}

Conversely, if $h^{-1}\co \eta_{H}=\eta_{B}$, by similar arguments (in this case $g=h$),
$$h\co
\eta_{H}=(h\ast u_{1}^{\varphi_{B}})\co \eta_{H}=(h\ast (h^{-1}\co
\Pi_{H}^{L}))\co \eta_{H}=u_{1}^{\phi_{B}}\co \eta_{H}=\eta_{B}.
$$

\end{proof}

As a particular instance of the previous proposition we have the following corollary.

\begin{cor}
\label{lemma219}
Let $(B,\varphi_{B})$ be a left weak $H$-module algebra and let $h:H\rightarrow B$ be a morphism such that $h\ast u_{1}^{\varphi_{B}}=h=u_{1}^{\varphi_{B}}\ast h$. Then, the following assertions are equivalent:
\begin{itemize} 
\item[(i)] $h\co \eta_{H}=\eta_{B}$.
\item[(ii)] $h\co \Pi_{H}^{L}=u_{1}^{\varphi_{B}}$.
\item[(iii)] $h\co \overline{\Pi}_{H}^{L}=u_{1}^{\varphi_{B}}$.
\end{itemize}

Moreover, if $h\in Reg_{\varphi_{B}}(H,B)$ with convolution inverse $h^{-1}:H\rightarrow B$, we have $h\ast \eta_{H}=\eta_{B}$ iff $h^{-1}\ast \eta_{H}=\eta_{B}.$ Then, under these conditions, if $h\ast \eta_{H}=\eta_{B}$,  the following assertions hold and are equivalent:
\begin{itemize} 
\item[(iv)] $h^{-1}\co \eta_{H}=\eta_{B}$.
\item[(v)] $h^{-1}\co \Pi_{H}^{L}=u_{1}^{\varphi_{B}}$.
\item[(vi)] $h^{-1}\co \overline{\Pi}_{H}^{L}=u_{1}^{\varphi_{B}}$.
\end{itemize}

\end{cor}

\begin{defin}
\label{reg-1+}
{\rm Let $H$ be a  weak Hopf algebra and
$\varphi_{B}:H\ot B\rightarrow B$ be a measuring. With
$$Reg_{\varphi_{B}}^{t}(H,B)$$ we will denote the set of morphisms
$h:H\rightarrow B$ in $Reg_{\varphi_{B}}(H,B)$ such that 
\begin{equation}
\label{h-total}  h\co \eta_{H}=\eta_{B}.
\end{equation}
Then, if $(B, \varphi_{B})$ is a left weak $H$-module algebra, by Corollary \ref{lemma219}, $Reg_{\varphi_{B}}^{t}(H,B)$ is a subgroup of $Reg_{\varphi_{B}}(H,B)$. 
}
\end{defin}

\begin{rem}
\label{rem+}
{\rm 
The set $Reg_{\varphi_{B}}^{t}(H,B)$ also acts on ${\mathcal P}_{\varphi_{B}}$, i.e.,  we have a map 
$$R^{\prime}:Reg_{\varphi_{B}}^{t}(H,B)\times {\mathcal P}_{\varphi_{B}}\rightarrow {\mathcal P}_{\varphi_{B}},$$
 defined by 
$$R^{\prime}(h,(\phi_{B},\tau))=R(h,(\phi_{B},\tau)),$$
where $R$ is the action defined in (\ref{reg-acts}). 

Note that, if $(B, \varphi_{B})$ is a left weak $H$-module algebra, the measuring $\varphi_{B}^{h}$ defined in (\ref{vh}) satisfies (b2) of Definition \ref{def} because
we have
\begin{itemize}
\item[ ]$\hspace{0.38cm}\varphi_{B}^{h}\co (\eta_{H}\ot B)$
\item [ ]$=  \mu_{B}\co ((\mu_{B}\co ((h\co \overline{\Pi}_{H}^{L})\ot B))\ot h^{-1})\co (H\ot P_{\varphi_{B}})\co ((\delta_{H}\co \eta_{H})\ot B)$ {\scriptsize ({\blue by  (\ref{edpi})})}

\item [ ]$= \mu_{B}\co ((\mu_{B}\co (u_{1}^{\varphi_{B}}\ot B))\ot h^{-1})\co (H\ot P_{\varphi_{B}})\co ((\delta_{H}\co \eta_{H})\ot B) $ {\scriptsize ({\blue by  (i) $\Rightarrow$ (iii)  of Corollary \ref{lemma219}})}

\item [ ]$= \mu_{B}\co (\mu_{B}\ot h^{-1})\co (B\ot P_{\varphi_{B}})\co ((P_{\varphi_{B}}\co (\eta_{H}\ot \eta_{B}))\ot B) $ {\scriptsize ({\blue by (\ref{eta-psi})})}

\item [ ]$= \mu_{B}\co (B\ot h^{-1})\co P_{\varphi_{B}}\co (\eta_{H}\ot B) $ {\scriptsize ({\blue by (\ref{wmeas-wcp}) and properties of $\eta_{B}$})}

\item [ ]$=\mu_{B}\co (\varphi_{B}\ot B)\co (H\ot c_{B,B})\co (((H\ot (h^{-1}\co \Pi_{H}^{L}))\co \delta_{H}\co \eta_{H})\ot B) $ {\scriptsize ({\blue by (\ref{edpi}) and naturality of $c$})}

\item [ ]$=\mu_{B}\co (B\ot u_{1}^{\varphi_{B}}) \co P_{\varphi_{B}}\co (\eta_{H}\ot B) $ {\scriptsize ({\blue  by  (v)  of Corollary \ref{lemma219} and the naturality of $c$})}

\item [ ]$= \varphi_{B}\co (\eta_{H}\ot B)
$ {\scriptsize ({\blue by (\ref{psiHB-1}) for  $\phi_{B}$})}

\item [ ]$=id_{B} 
$ {\scriptsize ({\blue by (b2) of Definition \ref{def}}).}
\end{itemize}  
}
\end{rem}

\begin{teo}
\label{equiv-hh-al}
Let $H$ be a  weak Hopf algebra, let $(B,\varphi_{B})$, $(B,\phi_{B})$ be  left weak $H$-module algebras and let $\sigma, \tau:H\ot H\rightarrow B$  be  morphisms such that $\sigma\ast u_{2}^{\varphi_{B}}=\sigma$, $\tau\ast u_{2}^{\phi_{B}}=\tau$. Assume that $\sigma$, $\tau$ satisfy the twisted condition (\ref{t-sigma}), the 2-cocycle condition
(\ref{c-sigma}) and  the normal condition (\ref{normal-sigma}). 
The following assertions are equivalent:
\begin{itemize}
\item[(i)] The weak crossed products  $(B\ot H, \mu_{B\ot_{\varphi_{B}}^{\sigma}H})$ and $(B\ot H, \mu_{B\ot_{\phi_{B}}^{\tau}H})$ are equivalent.
\item[(ii)] There exists a  gauge transformation $(h,h^{-1})$  for $\varphi_{B}$ such that (\ref{h6}), 
\begin{equation}
\label{n1} h\co \eta_{H}=\eta_{B},
\end{equation}
\begin{equation}
\label{n2} \mu_{B}\co (B\ot  h)\co P_{\phi_{B}}=\mu_{B}\co (h\ot \varphi_{B})\co (\delta_{H}\ot B), 
\end{equation}
\begin{equation}
\label{n3} \mu_{B}\co (B\ot  h)\co F_{\tau}=(\mu_{B}\ot H)\co (\mu_{B}\ot \sigma)\co (B\ot P_{\varphi_{B}}\ot H)\co (((h\ot H)\co \delta_{H})\ot 
((h\ot H)\co \delta_{H})), 
\end{equation}
\end{itemize}
hold.
\end{teo}
 
\begin{proof} We first prove (i)$\Rightarrow$(ii). By Corollary \ref{crossed-product1}, we know that $(B\ot H, \mu_{B\ot_{\varphi_{B}}^{\sigma}H})$ and $(B\ot H, \mu_{B\ot_{\phi_{B}}^{\tau}H})$ are weak crossed products with preunits $\nu=\nabla_{B\otimes H}^{\varphi_{B}}\co (\eta_{B}\ot \eta_{H})$, $u=\nabla_{B\otimes H}^{\phi_{B}}\co (\eta_{B}\ot \eta_{H})$, respectively.   As in the proof of Theorem \ref{equiv-hh} define $\theta, \gamma$ by (\ref{tg}) and $h$, $h^{-1}$ by (\ref{h-h-1}). Then, using that $T,S$ are morphisms of left $B$-modules and (\ref{th-gh-1}) we have the following identities:
\begin{equation}
\label{Te-Se}
(B\ot \varepsilon_{H})\co T=\mu_{B}\co (B\ot h^{-1}),\;\; (B\ot \varepsilon_{H})\co S=\mu_{B}\co (B\ot h). 
\end{equation}

By  (i)$\Rightarrow$(ii) of Theorem \ref{equiv-hh}, the pair $(h,h^{-1})$ is a gauge transformation for $\varphi_{B}$ and the identities (\ref{h4}), (\ref{h5}) and (\ref{h3}) hold. Therefore,  by Remark (\ref{hh-1}) we obtain that (\ref{h6}) holds.

On the other hand, we obtain (\ref{n1})  because 
\begin{itemize}
\item[ ]$\hspace{0.38cm}h\co \eta_{H}$
\item [ ]$=(B\ot \varepsilon_{H})\co S\co  (\eta_{B}\ot \eta_{H}) $ {\scriptsize ({\blue by (\ref{Te-Se})})}
\item [ ]$=  (B\ot \varepsilon_{H})\co \nabla_{B\otimes H}^{\phi_{B}}\co (\eta_{B}\ot \eta_{H})$ {\scriptsize ({\blue by (\ref{uT})})}
\item [ ]$=   u_{1}^{\phi_{B}}\co \eta_{H} $ {\scriptsize ({\blue by naturality of $c$ and counit properties})}
\item [ ]$= \eta_{B}$ {\scriptsize ({\blue by (b2) of Definition \ref{def}}).}
\end{itemize}

Now, by the proof of (i)$\Rightarrow$(ii) of Theorem \ref{equiv-hh} we know that $S$ is multiplicative, i.e., 
\begin{equation}
\label{Smul}
S\co \mu_{B\ot_{\phi_{B}}^{\tau}H}=\mu_{B\ot_{\varphi_{B}}^{\sigma}H}\co (S\ot S).
\end{equation}

Then composing with $\eta_{B}\ot H\ot B\ot \eta_{H}$ in the previous identity we have 
\begin{itemize}
\item[ ]$\hspace{0.38cm}  S\co \mu_{B\ot_{\phi_{B}}^{\tau}H}\co (\eta_{B}\ot H\ot B\ot \eta_{H}) $
\item [ ]$= (\mu_{B}\ot H)\co (\mu_{B}\ot \theta)\co (B\ot (F_{\tau}\co (H\ot \eta_{H})))\co P_{\phi_{B}} $ {\scriptsize ({\blue by (\ref{tg-1}) and  properties of $\eta_{B}$})}
\item [ ]$= (\mu_{B}\ot H)\co (\mu_{B}\ot \theta)\co (B\ot (\nabla_{B\ot H}^{\phi_{B}}\co (\eta_{B}\ot H)))\co P_{\phi_{B}}$ {\scriptsize ({\blue by (\ref{sigma-preunit1})})}
\item [ ]$= (\mu_{B}\ot H)\co (B\ot \theta)\co \nabla_{B\otimes H}^{\phi_{B}}\co P_{\phi_{B}}$ {\scriptsize ({\blue by  the left $B$-linearity of $\nabla_{B\otimes H}^{\phi_{B}}$ and  properties of $\eta_{B}$})}
\item [ ]$= (\mu_{B}\ot H)\co (B\ot \theta)\co P_{\phi_{B}}$ {\scriptsize ({\blue by (\ref{nabla-psi}).})}
\end{itemize}
and 
\begin{itemize}
\item[ ]$\hspace{0.38cm}  \mu_{B\ot_{\varphi_{B}}^{\sigma}H}\co (S\ot S)\co (\eta_{B}\ot H\ot B\ot \eta_{H}) $
\item [ ]$=  \mu_{B\ot_{\varphi_{B}}^{\sigma}H}\co (((h\ot H)\co \delta_{H})\ot (((\mu_{B}\ot H)\co (B\ot (( h\ot H)\co \delta_{H}\co \eta_{H}))) $ {\scriptsize ({\blue by (\ref{tg-1}),  (\ref{th-gh-1}), and  properties}}
\item[ ]$\hspace{0.38cm}${\scriptsize {\blue of $\eta_{B}$})}
\item [ ]$=  \mu_{B\ot_{\varphi_{B}}^{\sigma}H}\co (((h\ot H)\co \delta_{H})\ot (\nabla_{B\otimes H}^{\varphi_{B}}\co (B\ot \eta_{H})))   $ {\scriptsize ({\blue by (\ref{equ-eta})})}
\item [ ]$= \mu_{B\ot_{\varphi_{B}}^{\sigma}H}\co (((h\ot H)\co \delta_{H})\ot B\ot \eta_{H}))) $ {\scriptsize ({\blue by (\ref{otra-prop})})}
\item [ ]$=  (\mu_{B}\ot H)\co (\mu_{B}\ot (\nabla_{B\otimes H}^{\varphi_{B}}\co (\eta_{B}\ot H)))  \co (B\ot P_{\varphi_{B}})\co (((h\ot H)\co \delta_{H})\ot B)$ {\scriptsize ({\blue by (\ref{sigma-preunit1})})}
\item [ ]$=  (\mu_{B}\ot H)\co  (B\ot (\nabla_{B\otimes H}^{\varphi_{B}}\co P_{\varphi_{B}}))\co (((h\ot H)\co \delta_{H})\ot B)$ {\scriptsize ({\blue by  the left $B$-linearity of $\nabla_{B\otimes H}^{\varphi_{B}}$ and  properties}}
\item[ ]$\hspace{0.38cm}${\scriptsize {\blue of $\eta_{B}$})}
\item [ ]$=   (\mu_{B}\ot H)\co  (B\ot P_{\varphi_{B}})\co (((h\ot H)\co \delta_{H})\ot B)$ {\scriptsize ({\blue by the left $B$-linearity of $\nabla_{B\otimes H}^{\varphi_{B}}$ and (\ref{nabla-psi})}).}
\end{itemize}

Therefore, 
\begin{equation}
\label{Smul-1}
(\mu_{B}\ot H)\co (B\ot \theta)\co P_{\phi_{B}}= (\mu_{B}\ot H)\co  (B\ot P_{\varphi_{B}})\co (((h\ot H)\co \delta_{H})\ot B)
\end{equation}
holds and, composing with $B\ot \varepsilon_{H}$ in (\ref{Smul-1}), we obtain (\ref{n2}) by the naturality of $c$ and (\ref{e-psi}).

Finally, composing with $\eta_{B}\ot H\ot \eta_{B}\ot H$ in (\ref{Smul}) we have 
\begin{itemize}
\item[ ]$\hspace{0.38cm}  S\co \mu_{B\ot_{\phi_{B}}^{\tau}H}\co (\eta_{B}\ot H\ot \eta_{B}\ot H) $
\item [ ]$= (\mu_{B}\ot H)\co (B\ot ((h\ot H)\co \delta_{H}))\co F_{\tau}$ {\scriptsize ({\blue by (\ref{aw1}), (\ref{tg-1}), (\ref{th-gh-1}) and  the properties of $\eta_{B}$})}
\end{itemize}
and 
\begin{itemize}
\item[ ]$\hspace{0.38cm}  \mu_{B\ot_{\varphi_{B}}^{\sigma}H}\co (S\ot S)\co (\eta_{B}\ot H\ot \eta_{B}\ot H)$
\item [ ]$= \mu_{B\ot_{\varphi_{B}}^{\sigma}H}\co (((h\ot H)\co \delta_{H})\ot ((h\ot H)\co \delta_{H}))$ {\scriptsize ({\blue by (\ref{tg-1}), properties of $\eta_{B}$, and (\ref{th-gh-1})}).}
\end{itemize}

As a consequence, 
\begin{equation}
\label{Sm2}
 (\mu_{B}\ot H)\co (B\ot ((h\ot H)\co \delta_{H}))\co F_{\tau}= \mu_{B\ot_{\varphi_{B}}^{\sigma}H}\co (((h\ot H)\co \delta_{H})\ot ((h\ot H)\co \delta_{H}))
\end{equation}
holds and, composing with $B\ot \varepsilon_{H}$, we obtain (\ref{n3}) by the counit properties and (\ref{sigmaHB-3}).

Conversely, consider that (ii) holds.  To prove that  (ii)$\Rightarrow$(i), following (ii)$\Rightarrow$(i) of Theorem \ref{equiv-hh}, we only need to obtain the equalities  (\ref{h4}), (\ref{h5}) and (\ref{h3}). Indeed, note that by Proposition  \ref{lemma119}, $h^{-1}\co \eta_{H}=\eta_{B}$ because (\ref{n1}) holds. Then,  (\ref{h4}) holds because
\begin{itemize}
\item[ ]$\hspace{0.38cm} \mu_{B}\co ((\mu_{B}\co (h\ot \varphi_{B})\co (\delta_{H}\ot B))\ot h^{-1})\co (H\ot c_{H,B})\co (\delta_{H}\ot B)$
\item[ ]$\hspace{0.38cm} \mu_{B}\co (B\ot (\mu_{B}\co  (h\ot h^{-1})))\co (P_{\phi_{B}}\ot H)\co  (H\ot c_{H,B})\co (\delta_{H}\ot B)$ {\scriptsize ({\blue  by (\ref{n2}) and the associativity of}}
\item[ ]$\hspace{0.38cm}$ {\scriptsize {\blue  $\mu_{B}$})}
\item [ ]$= \mu_{B}\co (B\ot (h\ast h^{-1}))\co P_{\phi_{B}} $ {\scriptsize ({\blue by the naturality of $c$ and coassociativity of $\delta_{H}$})}
\item [ ]$= \mu_{B}\co (B\ot u_{1}^{\phi_{B}})\co  P_{\phi_{B}}  $ {\scriptsize ({\blue by (\ref{h6})})}
\item [ ]$=\phi_{B}$ {\scriptsize ({\blue by (\ref{psiHB-1}) for $\phi_{B}$}).}
\end{itemize}

On the other hand, 
\begin{itemize}
\item[ ]$\hspace{0.38cm}\mu_{B}\co (B\ot h^{-1})\co \mu_{B\ot_{\varphi_{B}}^{\sigma}H}\co (((h\ot H)\co \delta_{H})\ot  ((h\ot H)\co \delta_{H}))$
\item [ ]$= \mu_{B}\co (((\mu_{B}\ot H)\co (\mu_{B}\ot \sigma)\co (B\ot P_{\varphi_{B}}\ot H)\co (((h\ot H)\co \delta_{H})\ot  ((h\ot H)\co \delta_{H})))\ot (h^{-1}\co \mu_{H}))$
\item[ ]$\hspace{0.38cm}\co  \delta_{H^{\ot 2}}$ {\scriptsize ({\blue by naturality of $c$ and coassociativity of $\delta_{H}$})}
\item [ ]$= \mu_{B}\co (B\ot (\mu_{B}\co (h\ot h^{-1})))\co  (F_{\tau}\ot \mu_{H})\co  \delta_{H^{\ot 2}}  $ {\scriptsize ({\blue by (\ref{n3}) and associativity of $\mu_{B}$})}
\item [ ]$= \mu_{B}\co (B\ot (h\ast h^{-1}))\co  F_{\tau}$ {\scriptsize ({\blue by (\ref{delta-sigmaHB})})}
\item [ ]$= \mu_{B}\co (B\ot u_{1}^{\phi_{B}})\co  F_{\tau}$ {\scriptsize ({\blue by (\ref{h6})})}
\item [ ]$=\tau$ {\scriptsize ({\blue by (\ref{sigmaHB-1})})}
\end{itemize}
and then (\ref{h5}) holds. Finally, we obtain (\ref{h3}) because
\begin{itemize}
\item[ ]$\hspace{0.38cm}((\mu_{B}\co (B\ot h^{-1}))\ot H)\co (B\ot \delta_{H})\co \nabla_{B\otimes H}^{\varphi_{B}}\co (\eta_{B}\ot \eta_{H})$
\item [ ]$= ((u_{1}^{\varphi_{B}}\ast h^{-1})\ot H)\co \delta_{H}\co \eta_{H} $ {\scriptsize ({\blue by (\ref{nabla-eta}) and by the coassociativity of $\delta_{H}$})}
\item [ ]$= ( h^{-1}\ot H)\co \delta_{H}\co \eta_{H} $ {\scriptsize ({\blue by the gauge transformation condition})}
\item [ ]$= ( (h^{-1}\co \overline{\Pi}_{H}^{L})\ot H)\co \delta_{H}\co \eta_{H} $ {\scriptsize ({\blue by (\ref{edpi})})}
\item [ ]$= ( u_{1}^{\phi_{B}}\ot H)\co \delta_{H}\co \eta_{H} $ {\scriptsize ({\blue by (vi) of Proposition \ref{lemma119}})}
\item [ ]$=\nabla_{B\otimes H}^{\phi_{B}}\co (\eta_{B}\ot \eta_{H})$ {\scriptsize ({\blue by  (\ref{nabla-eta})}).}
\end{itemize}
\end{proof}

\begin{apart}
{\rm 
As a consequence of the previous theorem, it is possible to define a groupoid, denoted by ${\mathcal G}_{H}^{B}$ whose objects are pairs 
$$(\varphi_{B},\sigma),$$ 
where $(B,\varphi_{B})$ is a left weak $H$-module algebra, $\sigma:H\ot H\rightarrow B$ is a morphism such that 
$$
u_{2}^{\varphi_{B}}\ast \sigma=\sigma=\sigma\ast u_{2}^{\varphi_{B}} 
$$
and the associated quadruple ${\Bbb B}_{H}$ satisfies the twisted, cocycle and normal conditions. 

A morphism between two objects  $(\varphi_{B},\sigma)$, $(\phi_{B},\tau)$ of ${\mathcal G}_{H}^{B}$ is defined by a morphism $h:H\rightarrow B$ for which there exists a morphism $h^{-1}:H\rightarrow B$ such that $(h,h^{-1})$ is a gauge transformation for $\varphi_{B}$ satisfying the conditions (ii) of Theorem \ref{equiv-hh-al}.  If $h:(\varphi_{B},\sigma)\rightarrow (\phi_{B},\tau)$, $g:(\phi_{B},\tau)\rightarrow (\chi_{B},\omega)$ are morphisms in ${\mathcal G}_{H}^{B}$, the composition, denoted by $g\circledcirc h$, is defined by 
$$g\circledcirc h=g\ast h.$$

The previous composition is well defined because, if $l=g\ast h$ and $l^{-1}=h^{-1}\ast g^{-1}$, it is easy to show that 
$(l,l^{-1})$ is a gauge transformation for $\varphi_{B}$ and  $l\ast l^{-1}=u_{1}^{\chi_{B}}$. Also, 
\begin{itemize}
\item[ ]$\hspace{0.38cm}l\co \eta_{H}$
\item [ ]$=(g\ast (h\co \Pi_{H}^{L})) \co \eta_{H}$ {\scriptsize ({\blue by (\ref{edpi})})}
\item [ ]$=(g\ast u_{1}^{\phi_{B}}) \co \eta_{H} $ {\scriptsize ({\blue by (ii) of Proposition \ref{lemma119}})}
\item [ ]$= g\co \eta_{H}  $ {\scriptsize ({\blue by the condition of gauge transformation})}
\item [ ]$= \eta_{B}  $ {\scriptsize ({\blue by (\ref{n1})})}
\end{itemize}
and then (\ref{n1}) holds for $l$. The equality (\ref{n2}) for $l$ follows by 
\begin{itemize}
\item[ ]$\hspace{0.38cm} \mu_{B}\co (B\ot l)\co P_{\chi_{B}}$
\item [ ]$=\mu_{B}\co ((\mu_{B}\co (B\ot g)\co P_{\chi_{B}})\ot h)\co (H\ot c_{H,B})\co (\delta_{H}\ot B) $ {\scriptsize ({\blue by the associativity of $\mu_{B}$, the coassociativity}}
\item[ ]$\hspace{0.38cm}${\scriptsize {\blue of $\delta_{H}$ and the naturality of $c$})}
\item [ ]$=\mu_{B}\co ((\mu_{B}\co (g\ot \phi_{B})\co (\delta_{H}\ot B))\ot h)\co (H\ot c_{H,B})\co (\delta_{H}\ot B) $ {\scriptsize ({\blue by (\ref{n2})})}
\item [ ]$=\mu_{B}\co ( g\ot ((\mu_{B}\co (h\ot \varphi_{B})\co (\delta_{H}\ot B)))\co (\delta_{H}\ot B)  $ {\scriptsize ({\blue by the coassociativity of $\delta_{H}$})}
\item [ ]$= \mu_{B}\co ( g\ot (\mu_{B}\co (B\ot h)\co P_{\phi_{B}}))\co (\delta_{H}\ot B)  $ {\scriptsize ({\blue by (\ref{n2})})}
\item [ ]$=(\mu_{B}\co (l\ot \phi_{B})\co (\delta_{H}\ot B) $ {\scriptsize ({\blue by the associativity of $\mu_{B}$ and the coassociativity of $\delta_{H}$})}
\end{itemize}
and (\ref{n3}) holds because
\begin{itemize}
\item[ ]$\hspace{0.38cm} \mu_{B}\co (B\ot l)\co F_{\omega}$
\item [ ]$=\mu_{B}\co ((\mu_{B}\co (B\ot g)\co F_{\omega})\ot (h\co \mu_{H}))\co \delta_{H^{\ot 2}}$ {\scriptsize ({\blue by the associativity of $\mu_{B}$, the coassociativity $\delta_{H^{\ot 2}}$})}
\item [ ]$=\mu_{B}\co ((\mu_{B}\co (\mu_{B}\ot \tau)\co (B\ot P_{\phi_{B}}\ot H)\co (((g\ot H)\co \delta_{H})\ot ((g\ot H)\co \delta_{H})))\ot (h\co \mu_{H}))\co \delta_{H^{\ot 2}}$ 
\item[ ]$\hspace{0.38cm}${\scriptsize ({\blue by (\ref{n3})})}
\item [ ]$=\mu_{B}\co ((\mu_{B}\co (B\ot \phi_{B}))\ot  (\mu_{B}\co (B\ot h)\co F_{\tau}))\co (B\ot  H\ot c_{H,B}\ot H)\co (((g\ot \delta_{H})\co \delta_{H})\ot ((g\ot H)\co \delta_{H}))$ 
\item[ ]$\hspace{0.38cm}${\scriptsize ({\blue by the associativity of $\mu_{B}$, the coassociativity of $\delta_{H}$ and the naturality of $c$})}
\item [ ]$=\mu_{B}\co ((\mu_{B}\co (B\ot \phi_{B}))\ot  (\mu_{B}\co (\mu_{B}\ot \sigma)\co (B\ot P_{\varphi_{B}}\ot H)\co (((h\ot H)\co \delta_{H})\ot ((h\ot H)\co \delta_{H}))))$
\item[ ]$\hspace{0.38cm}\co (B\ot  H\ot c_{H,B}\ot H)\co (((g\ot \delta_{H})\co \delta_{H})\ot ((g\ot H)\co \delta_{H}))   $ {\scriptsize ({\blue by the associativity of $\mu_{B}$ and (\ref{n3})})}
\item [ ]$=\mu_{B}\co ((\mu_{B}\co (B\ot (\mu_{B}\co (B\ot h) \co P_{\phi_{B}})))\ot (\mu_{B}\co (B\ot \sigma)\co (P_{\varphi_{B}}\ot H)))$
\item[ ]$\hspace{0.38cm}\co (((g\ot \delta_{H})\co \delta_{H})\ot ((((g\ot h)\co \delta_{H})\ot H)\co \delta_{H}))  $ {\scriptsize ({\blue by the associativity of $\mu_{B}$, the coassociativity of $\delta_{H}$ and}}
\item[ ]$\hspace{0.38cm}${\scriptsize {\blue the naturality of $c$})}
\item [ ]$=\mu_{B}\co ((\mu_{B}\co (B\ot (\mu_{B}\co (h\ot \varphi_{B}) \co (\delta_{H}\ot B))))\ot (\mu_{B}\co (B\ot \sigma)\co (P_{\varphi_{B}}\ot H)))$
\item[ ]$\hspace{0.38cm}\co (((g\ot \delta_{H})\co \delta_{H})\ot ((((g\ot h)\co \delta_{H})\ot H)\co \delta_{H}))  $ {\scriptsize ({\blue by the associativity of $\mu_{B}$ and (\ref{n2})})}
\item [ ]$=\mu_{B}\co (B\ot (\mu_{B}\co ((\mu_{B}\co (\varphi_{B}\ot \varphi_{B})\co (H\ot c_{H,B}\ot B)\co (\delta_{H}\ot B\ot B))\ot \sigma)))$ 
\item[ ]$\hspace{0.38cm}\co (B\ot H\ot ( (B\ot c_{H,B})\co (c_{H,B}\ot B))\ot H)\co (((g\ot \delta_{H})\co \delta_{H})\ot ((((g\ot h)\co \delta_{H})\ot H)\co \delta_{H}))$ 
\item[ ]$\hspace{0.38cm}${\scriptsize ({\blue by the associativity of $\mu_{B}$, the coassociativity of $\delta_{H}$ and the naturality of $c$})}
\item [ ]$= \mu_{B}\co (\mu_{B}\ot \sigma)\co (B\ot P_{\varphi_{B}}\ot H)\co (((l\ot H)\co \delta_{H})\ot ((l\ot H)\co \delta_{H}))$ {\scriptsize ({\blue by the associativity of $\mu_{B}$ and}}
\item[ ]$\hspace{0.38cm}${\scriptsize {\blue (b1) of Definition \ref{def}}).}
\end{itemize}

The identity of $(\varphi_{B},\sigma)$ is $id_{(\varphi_{B},\sigma)}=u_{1}^{\varphi_{B}}$ because $h\ast u_{1}^{\varphi_{B}}=h$ and $(u_{1}^{\varphi_{B}},u_{1}^{\varphi_{B}})$ is a gauge transformation for $\varphi_{B}$ satisfying (\ref{h-h-1-varphi}). The equality (\ref{n1}) follows from (b2) of Definition \ref{def},  (\ref{n2}) is a consequence of the naturality of $c$ and (b1) of Definition \ref{def} and  (\ref{n3}) holds because:
\begin{itemize}
\item[ ]$\hspace{0.38cm}(\mu_{B}\ot H)\co (\mu_{B}\ot \sigma)\co (B\ot P_{\varphi_{B}}\ot H)\co (((u_{1}^{\varphi_{B}}\ot H)\co \delta_{H})\ot  ((u_{1}^{\varphi_{B}}\ot H)\co \delta_{H}))$
\item [ ]$=\mu_{B}\co ((\mu_{B}\co ( u_{1}^{\varphi_{B}}\ot\varphi_{B})\co (\delta_{H}\ot   u_{1}^{\varphi_{B}}))\ot \sigma)\co \delta_{H^{\ot 2}}$
{\scriptsize ({\blue by naturality of $c$ and coassociativity of $\delta_{H}$})}
\item [ ]$=\mu_{B}\co ((\varphi_{B}\co (H\ot  u_{1}^{\varphi_{B}}) )\ot \sigma)\co \delta_{H}^{\ot 2}$ {\scriptsize ({\blue by (\ref{nabla-fi})})}
\item [ ]$=u_{2}^{\varphi_{B}}\ast \sigma  $ {\scriptsize ({\blue by  (b3) of Definition \ref{def}})}
\item [ ]$= \sigma $ {\scriptsize ({\blue by (\ref{c-sigma})})}
\item [ ]$=\mu_{B}\co (B\ot   u_{1}^{\varphi_{B}})\co F_{\sigma}$ {\scriptsize ({\blue by  (\ref{sigmaHB-1})}).}
\end{itemize}

As a consequence, $h$ is an isomorphism with inverse $h^{-1}$ and $(g\circledcirc h)^{-1}=h^{-1}\circledcirc g^{-1}$. Therefore, ${\mathcal G}_{H}^{B}$ is a groupoid.
}
\end{apart}

\begin{apart}
{\rm 
Let $H$ be a  weak Hopf algebra, let $(B,\varphi_{B})$  be a  left weak $H$-module algebra and let $\sigma:H\ot H\rightarrow B$  be  a morphism such that $\sigma\ast u_{2}^{\varphi_{B}}=\sigma$. Assume that $\sigma$ satisfies the twisted condition (\ref{t-sigma}), the 2-cocycle condition (\ref{c-sigma}) and  the normal condition (\ref{normal-sigma}). Let $h$ be a morphism in $Reg_{\varphi_{B}}^{t}(H,B)$. Then $(h, h^{-1})$ is a gauge transformation for $\varphi_{B}$ such that (\ref{h-h-1-varphi}) and (\ref{n1}) hold. Define $\varphi_{B}^{h}$ and $\sigma{h}$ as in (\ref{vh}) and (\ref{sh}) respectively. Then, by  (\ref{basic}), $\varphi_{B}^{h}$ is a measuring such that (\ref{gt1}) holds. Therefore $u_{1}^{\varphi_{B}}=u_{1}^{\varphi_{B}^h}$ and then 
\begin{equation}
\label{nn2}
\nabla_{B\ot H}^{\varphi_{B}^h}=\nabla_{B\ot H}^{\varphi_{B}}.
\end{equation}

Moreover, by Remark \ref{rem+}, we know that $\varphi_{B}^{h}$ satisfies (b2) of Definition \ref{def}. On the other hand, $\sigma^{h}$ is such that $\sigma^{h}\ast u_{2}^{\varphi_{B}^{h}}=\sigma^{h}$ and satisfies  the twisted condition (\ref{t-sigma}), the cocycle condition (\ref{c-sigma}) and  $$\nu^{h}=(\mu_{B}\ot H)\co (B\ot ((h^{-1}\ot H)\co \delta_{H}))\co \nabla_{B\ot H}^{\varphi_{B}}\co (\eta_{B}\ot \eta_{H})$$ is a preunit for the associated weak crossed product $(B\ot H, \mu_{B\ot_{\varphi_{B}^{h}}^{\sigma^h}H})$. Note that 

\begin{itemize}

\item[ ]$\hspace{0.38cm} \nu^{h}$

\item [ ]$=((u_{1}^{\varphi_{B}}\ast h^{-1})\ot H)\co \delta_{H}\co \eta_{H}$ {\scriptsize ({\blue by the coassociativity of $\delta_{H}$})}

\item [ ]$=(h^{-1}\ot H)\co \delta_{H}\co \eta_{H} $ {\scriptsize ({\blue by the condition of gauge transformation})}

\item [ ]$= ((h^{-1}\co \overline{\Pi}_{H}^{L})\ot H)\co \delta_{H}\co \eta_{H}  $ {\scriptsize ({\blue by (\ref{edpi})})}

\item [ ]$= (u_{1}^{\varphi_{B}}\ot H)\co \delta_{H}\co \eta_{H}  $ {\scriptsize ({\blue by (vi) of Corollary (\ref{lemma219})})}

\item [ ]$= \nabla_{B\ot H}^{\varphi_{B}}\co (\eta_{B}\ot \eta_{H}) $ {\scriptsize ({\blue by (\ref{nabla-eta}})}

\item [ ]$= \nabla_{B\ot H}^{\varphi_{B}^h}\co (\eta_{B}\ot \eta_{H})$ {\scriptsize ({\blue by (\ref{nn2})}).}

\end{itemize}

Therefore, 
$$
\nu^{h}=\nabla_{B\ot H}^{\varphi_{B}^h}\co (\eta_{B}\ot \eta_{H})=\nabla_{B\ot H}^{\varphi_{B}}\co (\eta_{B}\ot \eta_{H})=\nu.
$$

Also, $\varphi_{B}^h$ satisfies (\ref{B2}) because 
\begin{itemize}
\item[ ]$\hspace{0.38cm}\varphi_{B}^h\co (\overline{\Pi}_{H}^{L}\ot B) $
\item [ ]$=\mu_{B}\ot (\mu_{B}\ot h^{-1})\co (B\ot P_{\varphi_{B}}))\co (((h\co \overline{\Pi}_{H}^{L})\ot H)\co \delta_{H}\co \overline{\Pi}_{H}^{L})  \ot B) $ {\scriptsize ({\blue by (\ref{d-pibar})})}
\item [ ]$=\mu_{B}\ot (\mu_{B}\ot h^{-1})\co (u_{1}^{\varphi_{B}}\ot P_{\varphi_{B}}))\co ((\delta_{H}\co \overline{\Pi}_{H}^{L})  \ot B)  $ {\scriptsize ({\blue by (ii) of Corollary \ref{lemma219}})}
\item [ ]$=\mu_{B}\ot (B\ot h^{-1})\co  P_{\varphi_{B}}\co (\overline{\Pi}_{H}^{L}  \ot B)  $ {\scriptsize ({\blue by (\ref{nabla-fiAH})})}
\item [ ]$=\mu_{B}\co ((\varphi_{B}\co  (\overline{\Pi}_{H}^{L}\ot B))\ot h^{-1})\co (H\ot c_{H,B})\co  ((\delta_{H}\co \overline{\Pi}_{H}^{L})  \ot B)   $ {\scriptsize ({\blue by  (\ref{d-pibar})})}
\item [ ]$= \mu_{B}\co ((\mu_{B}\co c_{B,B}\co (u_{1}^{\varphi_{B}}))\ot h^{-1})\co (H\ot c_{H,B})\co  ((\delta_{H}\co \overline{\Pi}_{H}^{L})  \ot B) $ {\scriptsize ({\blue by (\ref{B2}) for $\varphi_{B}$})}
\item [ ]$=\mu_{B}\co (B\ot (u_{1}^{\varphi_{B}}\ast h^{-1}))\co c_{H,B}\co   (\overline{\Pi}_{H}^{L}  \ot B) $ {\scriptsize ({\blue by the associativity of $\mu_{B}$ and the naturality of $c$})}
\item [ ]$= \mu_{B}\co c_{B,B}\co ( (h^{-1}\co \overline{\Pi}_{H}^{L})\ot B)  $ {\scriptsize ({\blue by the naturality of $c$ and the condition of gauge transformation})}
\item [ ]$= \mu_{B}\co c_{B,B}\co (u_{1}^{\varphi_{B}}\ot B)   $ {\scriptsize ({\blue by  (vi) of Corollary \ref{lemma219}})}
\item [ ]$= \mu_{B}\co c_{B,B}\co (u_{1}^{\varphi_{B}^h}\ot B)  $ {\scriptsize ({\blue by $u_{1}^{\varphi_{B}}=u_{1}^{\varphi_{B}^h}$})}
\end{itemize}

As a consequence, using the same proof than in Remark \ref{miracle}, we obtain that (\ref{nvar-eta}) holds for $\varphi_{B}^h$. Therefore, by Corollary \ref{norma-sigma-prop4}, we have that $\sigma^{h}$ satisfies the normal condition (\ref{normal-sigma}), i.e.,  $$\sigma^{h}\co (\eta_{H}\ot H)=u_{1}^{\varphi_{B}^h}=\sigma^{h}\co (H\ot \eta_{H}).$$

Finally, if $B$ is commutative  and $H$ is coconmutative, the equality  
\begin{equation}
\label{cbch}
\mu_{B}\co (B\ot h)\co P_{\varphi_{B}}=\mu_{B}\co (h\ot \varphi_{B})\co (\delta_{H}\ot B)
\end{equation}
holds and then 
\begin{itemize}

\item[ ]$\hspace{0.38cm} \varphi_{B}^h$

\item [ ]$= \mu_{B}\co ((\mu_{B}\co (h\ot \varphi_{B})\co (\delta_{H}\ot B))\ot h^{-1})\co (H\ot c_{H,B})\co (\delta_{H}\ot B)$ {\scriptsize ({\blue by the coassociativity of $\delta_{H}$})}

\item [ ]$= \mu_{B}\co ((\mu_{B}\co (B\ot h)\co P_{\varphi_{B}}))\ot h^{-1})\co (H\ot c_{H,B})\co (\delta_{H}\ot B) $ {\scriptsize ({\blue by (\ref{cbch})})}

\item [ ]$= \mu_{B}\co (B\ot (h\ast h^{-1}))\co P_{\varphi_{B}}  $ {\scriptsize ({\blue by the coassociativity of $\delta_{H}$, the associativity of $\mu_{B}$ and the naturality of $c$})}

\item [ ]$= \mu_{B}\co (B\ot u_{1}^{\varphi_{B}})\co P_{\varphi_{B}}  $ {\scriptsize ({\blue by (\ref{h-h-1-varphi})})}

\item [ ]$=\varphi_{B} $ {\scriptsize ({\blue by (\ref{psiHB-1}}).}

\end{itemize}

}
\end{apart}

\section{Hom-products, invertible morphisms and centers} 

In this subsection, for a weak Hopf algebra $H$ and an algebra $B$, we will explore a product in $ Hom_{\sf C}(H^{\ot n}\ot  
 B, B)$ that will permit us to extend some results about the factorization through the center of $B$, given in \cite{Montl} for Hopf algebras, to the weak Hopf algebra setting.

\begin{defin}
\label{hom-pro}
{\rm Let  $H$ be a weak Hopf algebra and let $\varphi_B:H\ot B\to B$ be a  measuring. Let $\varphi$ and $\psi\in Hom_{\sf C}(H^{\ot n}\ot  B, B)$. We define the product 
$$\wedge:Hom_{\sf C}(H^{\ot n}\ot  B, B)\times Hom_{\sf C}(H^{\ot n}\ot  B, B)\rightarrow Hom_{\sf C}(H^{\ot n}\ot  B, B) $$ between $\varphi$ and $\psi$ as 
$$ \varphi \wedge \psi=\varphi\co (H^{\ot n}\ot \psi)\co (\delta_{H^{\ot n}}\ot B).$$

Obviously, $\wedge$ is an associative product because $\delta_{H^{\ot n}}$ is coassociative.

We say that a morphism $\varphi\in Hom_{\sf C}(H^{\ot n}\ot  B, B)$ is $\varphi_{B}$-invertible  if there exists a morphism $\varphi^{\dagger}\in Hom_{\sf C}(H^{\ot n}\ot  B, B)$ such that 
\begin{equation}
\label{dagger}
\varphi \wedge \varphi^{\dagger}= \mu_B\circ (u_{n}^{\varphi_{B}}\ot B).
\end{equation}

}
\end{defin}

\begin{prop}
\label{prp}
Let $H$ be a weak Hopf algebra and let  $\varphi_B:H\ot B\to B$ be a  measuring. For $\omega:H^{\ot n}\rightarrow B$ define 
$$\overline{\omega}=\mu_{B}\co (\omega\ot B),\;\;\; \overline{\omega}^{op}=\mu_{B}\co c_{B,B}\co (\omega\ot B).$$

Then, if $\omega$, $\theta\in Hom_{\sf C}(H^{\ot n}, B)$ and $\gamma\in Hom_{\sf C}(H^{\ot n}\ot  B, B)$ the following equalities hold:
\begin{itemize}
\item[(i)] $\overline{\omega}\wedge \overline{\theta}=\overline{\omega\ast \theta}$.
\item[(ii)] If $H$ is cocommutative, $\overline{\omega}^{op}\wedge \gamma=\mu_{B}\co (\gamma\ot \omega)\co (H^{\ot n}\ot c_{H^{\ot n},B})\co (\delta_{H^{\ot n}}\ot B)$. 
\item[(iii)] If $H$ is cocommutative, $\overline{\omega}^{op}\wedge \overline{\theta}^{op}=\overline{\theta\ast \omega}^{op}$.
\item[(iv)]  If $H$ is cocommutative, $\overline{\omega}^{op}\wedge \overline{\theta}= \overline{\theta}\wedge \overline{\omega}^{op}$.
\item[(v)]  If $(B, \varphi_B)$ is a left weak $H$-module algebra,  $\overline{u_{n}^{\varphi_{B}}}\wedge \varphi_{B}^{\ot n}=\varphi_{B}^{\ot n}$.
\item[(vi)] If $H$ is cocommutative and  $(B, \varphi_B)$ is a left weak $H$-module algebra,   $\overline{u_{n}^{\varphi_{B}}}^{op}\wedge \varphi_{B}^{\ot n}=\varphi_{B}^{\ot n}$.
\item[(vii)] $\overline{\varphi_{B}\co (H\ot \omega)}\wedge (\varphi_{B}\co (H\ot \gamma))=\varphi_{B}\co (H\ot  (\overline{\omega}\wedge \gamma))$.
\item[(viii)] If $H$ is cocommutative, $\overline{\varphi_{B}\co (H\ot \omega)}^{op}\wedge (\varphi_{B}\co (H\ot \gamma))=\varphi_{B}\co (H\ot  (\overline{\omega}^{op}\wedge \gamma))$.
\end{itemize}
\end{prop}

\begin{proof} The proof of (i) follows directly from the associativity of $\mu_{B}$. If $H$ is cocommutative, so is $H^{\ot n}$ and, by the naturality of $c$, we obtain (ii). By similar reasoning and using the associativity of $\mu_{B}$ we obtain (iii) and (iv). On the other hand, 
$$\overline{u_{n}^{\varphi_{B}}}\wedge \varphi_{B}^{\ot n}$$
$$=\mu_{B}\co (\varphi_{B}^{\ot n}\ot \varphi_{B}^{\ot n})\co (H^{\ot n}\ot c_{H^{\ot n},B}\ot B)\co (\delta_{H^{\ot n}}\ot \eta_{B}\ot B)\stackrel{{\scriptsize \blue  (\ref{vp-n})}}{=}\varphi_{B}^{\ot n}\co (H^{\ot n}\ot (\mu_{B}\co (\eta_{B}\ot B)))=\varphi_{B}^{\ot n},$$
and then (v) holds. Similarly, using that $H^{\ot n}$ is cocommutative, the naturality of $c$ and  (\ref{vp-n}) we prove (vi). 

The identity, (vii) follows from 
\begin{itemize}

\item[ ]$\hspace{0.38cm} \overline{\varphi_{B}\co (H\ot \omega)}\wedge (\varphi_{B}\co (H\ot \gamma))$

\item [ ]$= \mu_{B}\co (\varphi_{B}\ot \varphi_{B})\co (H\ot c_{H,B}\ot B)\co (\delta_{H}\ot ((\omega\ot \theta)\co (\delta_{H^{\ot n}}\ot B)))$ {\scriptsize ({\blue by the naturality of $c$})}

\item [ ]$=\varphi_{B}\co (H\ot  (\overline{\omega}\wedge \gamma)) $ {\scriptsize ({\blue by (\ref{vp-n})})}

\end{itemize}
and, similarly, using that $H$ is cocommutative, we obtain (viii).
\end{proof}

\begin{rem}
{\rm The equivalence of measurings (or, in particular, of weak actions) in terms of gauge transformations acquires a new meaning in terms of this product. Actually, if $H$ is cocommutative, the action described in (\ref{reg-acts}) on a measuring $\phi_B$ can be seen as a conjugation by gauge transformations in the following way:
$$
\phi_B^h = \overline{h}\wedge\overline{h^{-1}}^{op}\wedge \phi_B.
$$
	
Moreover observe that for a cocommutative weak Hopf algebra $H$ and measurings $\varphi_B$ and $\phi_B$ satisfying conditions of Theorem \ref{equiv-hh}, we can re-write equality (\ref{h4}) using the Hom-product as
$$
\phi_B = \overline{h}\wedge\overline{h^{-1}}^{op}\wedge \varphi_B.
$$

Also in this way, equality (\ref{n2}) of Theorem \ref{equiv-hh-al} can be interpreted as $
\overline{h}^{op}\wedge\phi_B = \overline{h}\wedge\varphi_B,$ in coherence with the action of gauge transformations as a conjugation given above. 
}
\end{rem}

\begin{prop}
\label{w-t} 
Let $H$ be a cocommutative weak Hopf algebra and  let $(B,\varphi_{B})$ be a left weak $H$-module algebra. A morphism $\sigma:H^{\ot 2}\rightarrow B$ satisfies the twisted condition (\ref{t-sigma}) if and only if 
\begin{equation}
\label{t-sigmac}
\overline{\sigma}^{op}\wedge \varphi_{B}^{\ot 2}=\overline{\sigma}\wedge (\varphi_{B}\co (\mu_{H}\ot B))
\end{equation}
holds.

\end{prop}

\begin{proof} The proof  follows from the following facts. First, note that by definition of $F_{\sigma}$, we have that 
$$\mu_{B}\co (B\ot \varphi_{B})\co  (F_{\sigma}\ot B)=\overline{\sigma}\wedge (\varphi_{B}\co (\mu_{H}\ot B)).$$

On the other hand, if $H$ is cocommutative, by the naturality of $c$, we have 
$$\mu_{B}\co (B\ot \sigma)\co  (P_{\varphi_{B}}\ot H)\co (H\ot P_{\varphi_{B}}) =\overline{\sigma}^{op}\wedge \varphi_{B}^{\ot 2}.$$

\end{proof}

\begin{defin}
\label{reg-n}
{\rm  Let $H$ be a weak Hopf algebra and
$(B,\varphi_{B})$ be a left weak $H$-module algebra. For $n\geq 1$,
with $Reg_{\varphi_{B}}(H^{\ot n},B)$ we will denote the set of
morphisms $\sigma:H^{\ot n}\rightarrow B$ such that there exists a
morphism $\sigma^{-1}:H^{\ot n}\rightarrow B$ (the convolution inverse of
$\sigma$) satisfying the following equalities:
\begin{equation}
\label{s1}
\sigma\ast \sigma^{-1}= \sigma^{-1}\ast
\sigma=u_{n}^{\varphi_{B}}, 
\end{equation}
\begin{equation}
\label{s2}
\sigma\ast \sigma^{-1}\ast \sigma=\sigma,\;\;
\sigma^{-1}\ast \sigma\ast \sigma^{-1}=\sigma^{-1}.
\end{equation}

Note that, for $n=1$, we recover the  group $Reg_{\varphi_{B}}(H,B)$ introduced in Definition \ref{reg-1}. For any $n$, $Reg_{\varphi_{B}}(H^{n},B)$ is a group with unit element $u_{n}^{\varphi_{B}}$ because by (\ref{idem-n}) we know that $u_{n}^{\varphi_{B}}\ast u_{n}^{\varphi_{B}}=u_{n}^{\varphi_{B}}$. Also,
if $B$ is commutative and $H$ is cocommutative, we have that  $Reg_{\varphi_{B}}(H^{\ot n},B)$ is an abelian group.

We denote by $Reg_{\varphi_{B}}(H_{L},B)$  the set of morphisms
$g:H_{L}\rightarrow B$ such that there exists a morphism
$g^{-1}:H_{L}\rightarrow B$ (the convolution inverse of $g$)
satisfying
$$g\ast g^{-1}=g^{-1}\ast g=u_{0}^{\varphi_{B}},\;\; g\ast g^{-1}
\ast g=g,\;\; g^{-1}\ast g\ast g^{-1}=g^{-1},$$  where
$u_{0}^{\varphi_{B}}=u_{1}^{\varphi_{B}}\co i_{H}^L$. Then by (\ref{B3}) we have  $u_{1}^{\varphi_{B}}=u_{0}^{\varphi_{B}}\co p_{H}^L$.
}
\end{defin}

\begin{defin}
\label{center}
{\rm For an algebra $B$ we define the center of $B$ as a subobject
$Z(B)$  of $B$ with a monomorphism 
$z_{B}:Z(B)\rightarrow B$ satisfying the identitity 
\begin{equation}
\label{cen1}
\mu_{B}\co (B\ot z_{B})=\mu_{B}\co c_{B,B}\co (B\ot z_{B})
\end{equation}
and such that, if $f:A\rightarrow B$ is a morphism for which $\mu_{B}\co (B\ot
f)=\mu_{B}\co c_{B,B}\co (B\ot f)$ holds, there exists an unique morphism
$f^{\prime}:A\rightarrow Z(B)$ satisfying $z_{B}\co f^{\prime}=f.$ As a consequence, we obtain that $
Z(B)$ is a commutative algebra, where $\eta_{Z(B)}$ is the unique morphism such that 
\begin{equation}
\label{zbeta}
z_{B}\co \eta_{Z(B)}=\eta_{B} 
\end{equation}
and $\mu_{Z(B)}$ is the unique morphism such that
\begin{equation}
\label{zbpro}
z_{B}\co \mu_{Z(B)}=\mu_{B}\co (z_{B}\ot z_{B}).
\end{equation}

For example, if ${\sf C}$ is a
closed category with equalizers and $\alpha_{B}$ and $\beta_{B}$ are the unit and the counit,
respectively, of the ${\sf C}$-adjunction $B\otimes -\dashv
[B,-]:{\sf C}\rightarrow {\sf C}$, the center of $B$ can
be obtained by the  following equalizer diagram:
$$
\setlength{\unitlength}{3mm}
\begin{picture}(30,4)
\put(3,2){\vector(1,0){4}} \put(11,2.5){\vector(1,0){10}}
\put(11,1.5){\vector(1,0){10}} \put(0,2){\makebox(0,0){$Z(B)$}} \put(9,2){\makebox(0,0){$B$}} \put(24,2){\makebox(0,0){$[B,B]$}} \put(5.5,3){\makebox(0,0){$z_{B}$}}
\put(16,3.5){\makebox(0,0){$\vartheta_{B}$}}
\put(16,0.15){\makebox(0,0){$\theta_{B}$}}
\end{picture}
$$
where $\vartheta_{B}=[B, \mu_{B}]\circ \alpha_{B}(B)$ and
$\theta_{B}=[B, \mu_{B}\circ c_{B,B}]\circ \alpha_{B}(B)$. Then in the category of modules over a commutative ring the center is an equalizer object.

Finally, note that by (\ref{cen1}), composing with the symmetry isomorphism we obtain 
$$
\mu_{B}\co ( z_{B}\ot B)=\mu_{B}\co c_{B,B}\co ( z_{B}\ot B).
$$
}
\end{defin}

\begin{ej}
\label{u1-cen}
{\rm Let $H$ be cocommutative weak Hopf algebra and let $(B,\varphi_{B})$ be a left weak $H$-module algebra. Then, $\Pi_{H}^{L}=\overline{\Pi}_{H}^{L}$ and by (\ref{B1}) and (\ref{B2}) we have that $\mu_{B}\circ c_{B,B}\circ (u_{1}^{\varphi_{B}}\ot B)=\mu_{B}\circ (u_{1}^{\varphi_{B}}\ot B)$. Then, $u_{1}^{\varphi_{B}}$ factors through $Z(B)$. Therefore, there exists an unique morphism $v_{1}^{\varphi_{B}}:H\rightarrow Z(B)$ such that 
$$
z_{B}\co v_{1}^{\varphi_{B}}=u_{1}^{\varphi_{B}}.
$$

Then, taking into account the equality (\ref{u}), we obtain
$$\mu_{B}\circ c_{B,B}\circ (u_{n}^{\varphi_{B}}\ot B)=\mu_{B}\circ c_{B,B}\circ ((u_{1}^{\varphi_{B}}\co m_{H}^{\ot n})\ot B)=\mu_{B}\circ ((u_{1}^{\varphi_{B}}\co m_{H}^{\ot n})\ot B)=\mu_{B}\circ (u_{n}^{\varphi_{B}}\ot B)$$
and, as a consequence, $u_{n}^{\varphi_{B}}$ factors through $Z(B)$. Therefore, there exists an unique morphism $v_{n}^{\varphi_{B}}:H\ot H\rightarrow Z(B)$ such that 
$$
z_{B}\co v_{n}^{\varphi_{B}}=u_{n}^{\varphi_{B}}.
$$
}
\end{ej}

\begin{rem}
\label{rmexc}
{\rm Let $H$ be a weak Hopf algebra. Let $\omega:H^{\ot n}\rightarrow B$ be a morphism. Then, $\omega$ factors through the center of $B$ if and only if $\overline{\omega}=\overline{\omega}^{op}$. Therefore, if $H$ is cocommutative and $(B,\varphi_{B})$ is a left weak $H$-module algebra, $\overline{u_{n}^{\varphi_{B}}}=\overline{u_{n}^{\varphi_{B}}}^{op}$ for all $n\geq 1$.

Also, if $\omega$ factors through the center of $B$, $\omega\ast \tau=\tau \ast \omega$ for all morphism $\tau:H^{\ot n}\rightarrow B$.

}
\end{rem}

\begin{prop}
\label{varphi-i}
Let $H$ be a cocommutative weak Hopf algebra. Let $(B,\varphi_{B})$ be a left weak $H$-module algebra and $\sigma \in Reg_{\varphi_{B}}(H^{\ot 2},B)$ satisfying the twisted condition (\ref{t-sigma}).  Then, $\varphi_{B}$ is $\varphi_{B}$-invertible.
\end{prop}

\begin{proof} Let $h_{\sigma}$ and $h_{\sigma^{-1}}$ be the morphisms defined by 
$$
h_{\sigma}=\sigma\co (H\ot \lambda_{H})\co \delta_{H}, \;\;\;h_{\sigma^{-1}}=\sigma^{-1}\co (H\ot \lambda_{H})\co \delta_{H}.
$$

Then, $h_{\sigma}\in Reg_{\varphi_{B}}(H,B)$ and $h_{\sigma}^{-1}=h_{\sigma^{-1}}$. Indeed, first note that 

\begin{itemize}

\item[ ]$\hspace{0.38cm}h_{\sigma}\ast h_{\sigma^{-1}} $

\item [ ]$= \mu_{B}\co (\sigma\ot \sigma^{-1})\co (H\ot c_{H,H}\ot H) \co (\delta_{H}\ot ((\lambda_{H}\ot \lambda_{H})\co \delta_{H}))\co \delta_{H} $ {\scriptsize ({\blue by the coassociativity and}}
\item[ ]$\hspace{0.38cm}${\scriptsize {\blue the cocommutativity of $\delta_{H}$ and the naturality of $c$})}

\item [ ]$=(\sigma\ast \sigma^{-1})\co (H\ot \lambda_{H})\co \delta_{H}$ {\scriptsize ({\blue the  cocommutativity of $\delta_{H}$ and (\ref{anti})})}

\item [ ]$= u_{1}^{\varphi_{B}}\co \mu_{H}\co  (H\ot \lambda_{H})\co \delta_{H} $ {\scriptsize ({\blue by $\sigma \in Reg_{\varphi_{B}}(H^{\ot 2},B)$})}

\item [ ]$= u_{1}^{\varphi_{B}}\co \Pi_{H}^{L}  $ {\scriptsize ({\blue by the definition of $\Pi_{H}^{L}$})}

\item [ ]$= u_{1}^{\varphi_{B}}  $ {\scriptsize ({\blue by (\ref{B3})})}

\end{itemize}
and similarly, $h_{\sigma^{-1}}\ast h_{\sigma}=u_{1}^{\varphi_{B}}$.  Also, by the coassociativity and the cocommutativity of $\delta_{H}$, the naturality of $c$,  (\ref{anti}) and $\sigma \in Reg_{\varphi_{B}}(H^{\ot 2},B)$ we have that 
$$h_{\sigma}\ast h_{\sigma^{-1}}\ast h_{\sigma}=(\sigma\ast \sigma^{-1}\ast \sigma)\co  (H\ot \lambda_{H})\co \delta_{H}=h_{\sigma}$$
and $h_{\sigma^{-1}}\ast h_{\sigma}\ast h_{\sigma^{-1}}=h_{\sigma^{-1}}$ hold.

Now, let $\varphi_{\sigma}$ be the morphism defined by 
$$
\varphi_{\sigma}=\mu_{B}\co (\mu_{B}\ot B)\co (h_{\sigma}\ot B\ot h_{\sigma^{-1}})\co (H\ot c_{H,H})\co (\delta_{H}\ot H).
$$

Then, $\varphi_{\sigma}$ is $\varphi_{B}$-invertible with inverse defined by 
$$
\varphi_{\sigma}^{\dagger}=\mu_{B}\co (\mu_{B}\ot B)\co (h_{\sigma^{-1}}\ot B\ot h_{\sigma})\co (H\ot c_{H,H})\co (\delta_{H}\ot H).
$$

Indeed: 
\begin{itemize}

\item[ ]$\hspace{0.38cm}\varphi_{\sigma}\wedge \varphi_{\sigma}^{\dagger}$

\item [ ]$= \mu_{B}\co (B\ot (\mu_{B}\co c_{B,B}))\ot (((h_{\sigma}\ast h_{\sigma^{-1}})\ot (h_{\sigma}\ast h_{\sigma^{-1}})\co \delta_{H})\ot B) $ {\scriptsize ({\blue by the coassociativity and the}}
\item[ ]$\hspace{0.38cm}${\scriptsize {\blue   cocommutativity of $\delta_{H}$, the naturality of $c$ and the associativity of $\mu_{B}$})}

\item [ ]$= \mu_{B}\co (u_{1}^{\varphi_{B}}\ot (\mu_{B}\co c_{B,B}\co (u_{1}^{\varphi_{B}}\ot B)))\co (\delta_{H}\ot B)$ {\scriptsize ({\blue by $h_{\sigma} \in Reg_{\varphi_{B}}(H,B)$})}

\item [ ]$=\mu_{B}\co (( u_{1}^{\varphi_{B}}\ast u_{1}^{\varphi_{B}})\ot B)$ {\scriptsize ({\blue by the factorization of $u_{1}^{\varphi_{B}}$ through the center of $B$})}

\item [ ]$= \mu_{B}\co (u_{1}^{\varphi_{B}} \ot B)$ {\scriptsize ({\blue by (\ref{idem-n})}).}

\end{itemize}

On the other hand, let be the morphism $\varphi_{B}\wedge (\varphi_{B}\co (\lambda_{H}\ot B))$. For this morphism we have the following

\begin{itemize}

\item[ ]$\hspace{0.38cm}\overline{h_{{\sigma}}}^{op}\wedge (\varphi_{B}\wedge (\varphi_{B}\co (\lambda_{H}\ot B))) $

\item [ ]$= \mu_{B}\co (\varphi_{B}\ot \sigma)\co (H\ot c_{H,B}\ot H)\co (\delta_{H}\ot (c_{H,B}\co (H\ot \varphi_{B})\co ((\delta_{H}\co \lambda_{H})\ot B)))\co (\delta_{H}\ot B)  $ {\scriptsize ({\blue by the}}
\item[ ]$\hspace{0.38cm}${\scriptsize {\blue  coassociativity and the cocommutativity of $\delta_{H}$, the naturality of $c$ and (\ref{anti})})}

\item [ ]$= \mu_{B}\co (B\ot \sigma)\co (P_{\varphi_{B}}\ot H)\co (H\ot  P_{\varphi_{B}})\co (((H\ot \lambda_{H})\co\delta_{H})\ot B)$ {\scriptsize ({\blue by the cocommutativity of $\delta_{H}$ and}}
\item[ ]$\hspace{0.38cm}${\scriptsize {\blue the naturality of $c$})}

\item [ ]$=\mu_{B}\co (B\ot \varphi_{B})\co (F_{\sigma}\ot B)\co  (((H\ot \lambda_{H})\co\delta_{H})\ot B)$ {\scriptsize ({\blue by (\ref{t-sigma})})}

\item [ ]$=\mu_{B}\co (B\ot \varphi_{B})\co (((h_{\sigma}\ot \Pi_{H}^{L})\co \delta_{H})\ot B) $ {\scriptsize ({\blue by the cocommutativity of $\delta_{H}$, the naturality of $c$ and (\ref{anti})})}

\item [ ]$= \mu_{B}\co (B\ot \mu_{B})\co (((h_{\sigma}\ot u_{1}^{\varphi_{B}})\co \delta_{H})\ot B) $ {\scriptsize ({\blue by (\ref{B1})})}

\item [ ]$=\mu_{B}\co ((h_{\sigma}\ast u_{1}^{\varphi_{B}} )\ot B)$ {\scriptsize ({\blue by  the associativity of $\mu_{B}$})}

\item [ ]$= \overline{h_{\sigma}} $ {\scriptsize ({\blue by $h_{\sigma} \in Reg_{\varphi_{B}}(H,B)$})} 
\end{itemize}
and, as a consequence, 
\begin{equation}
\label{hsig2}
\varphi_{\sigma}=\varphi_{B}\wedge (\varphi_{B}\co (\lambda_{H}\ot B))
\end{equation}
holds because, on the one hand, 
\begin{itemize}

\item[ ]$\hspace{0.38cm}\overline{h_{{\sigma}^{-1}}}^{op}\wedge (\overline{h_{{\sigma}}}^{op}\wedge (\varphi_{B}\wedge (\varphi_{B}\co (\lambda_{H}\ot B)))) $

\item [ ]$=\overline{h_{\sigma}\ast h_{\sigma^{-1}}}^{op} \wedge (\varphi_{B}\wedge (\varphi_{B}\co (\lambda_{H}\ot B)))$ {\scriptsize ({\blue by associativity of $\wedge$ and (iii) of Proposition \ref{prp}})}

\item [ ]$=\overline{u_{1}^{\varphi{B}}}^{op} \wedge (\varphi_{B}\wedge (\varphi_{B}\co (\lambda_{H}\ot B))) $ {\scriptsize ({\blue by $h_{\sigma} \in Reg_{\varphi_{B}}(H,B)$})}

\item [ ]$= (\overline{u_{1}^{\varphi{B}}}^{op} \wedge \varphi_{B})  \wedge (\varphi_{B}\co (\lambda_{H}\ot B))$ {\scriptsize ({\blue by associativity of $\wedge$})}

\item [ ]$= \varphi_{B}\wedge (\varphi_{B}\co (\lambda_{H}\ot B))$ {\scriptsize ({\blue by (vi) of Proposition \ref{prp}})}

\end{itemize}

and, on the other hand, by the cocommutativity of $\delta_{H}$, the naturality of $c$ and the associativity of $\mu_{B}$
$$\overline{h_{{\sigma}^{-1}}}^{op}\wedge \overline{h_{\sigma}}=\varphi_{\sigma}.$$

Finally, define the morphism $\varphi_{B}^{\dagger}$ by 
$$
\varphi_{B}^{\dagger}=(\varphi_{B}\co (\lambda_{H}\ot B))\wedge \varphi_{\sigma^{-1}}.
$$

Then, by (\ref{hsig2}) we have 
$$\varphi_{B}\wedge \varphi_{B}^{\dagger}=\varphi_{\sigma}\wedge \varphi_{\sigma}^{\dagger}=\mu_{B}\co (u_{1}^{\varphi_{B}} \ot B)$$
and therefore  $\varphi_{B}$ is $\varphi_{B}$-invertible with inverse $\varphi_{B}^{\dagger}$.
\end{proof}

\begin{prop}
\label{varphi-ic}
Let $H$ be a cocommutative weak Hopf algebra. Let $(B,\varphi_{B})$ be a left weak $H$-module algebra and suppose that $\varphi_{B}$ is $\varphi_{B}$-invertible.  Then, $\varphi_{B}^{\ot n}$ is $\varphi_{B}^{\ot n}$-invertible.
\end{prop}

\begin{proof} By assumption the asserttion is true for $n=1$. We will  proceed by induction assuming that it is true for $n-1$, i.e., $\varphi_{B}^{\ot (n-1)}$ is $\varphi_{B}^{\ot (n-1)}$-invertible with inverse $\varphi_{B}^{\ot (n-1) \dagger}$, and we will prove it for $n$. Indeed: Define $\varphi_{B}^{\ot n \dagger}$ by 
$$\varphi_{B}^{\ot n \dagger}=\varphi_{B}^{\ot (n-1) \dagger}\co (H^{\ot (n-1)}\ot \varphi_{B}^{\dagger})\co (c_{H, H^{\ot (n-1)}}\ot B).$$

Then, 

\begin{itemize}

\item[ ]$\hspace{0.38cm} \varphi_{B}^{\ot n}\wedge \varphi_{B}^{\ot n \dagger}$

\item [ ]$=\varphi_{B}\co (H\ot ( \varphi_{B}^{\ot (n-1)}\wedge \varphi_{B}^{\ot (n-1) \dagger}))\co (H\ot H^{\ot (n-1)}\ot \varphi_{B}^{\dagger})\co (H\ot c_{H, H^{\ot (n-1)}}\ot B)\co (\delta_{H}\ot H^{\ot (n-1)}\ot B)$ 
\item[ ]$\hspace{0.38cm}${\scriptsize ({\blue by naturality of $c$})}

\item [ ]$= \varphi_{B}\co (H\ot (\mu_{B}\co (u_{n-1}^{\varphi_{B}}\ot B)))\co (H\ot H^{\ot (n-1)}\ot \varphi_{B}^{\dagger})\co (H\ot c_{H, H^{\ot (n-1)}}\ot B)\co (\delta_{H}\ot H^{\ot (n-1)}\ot B) $ 
\item[ ]$\hspace{0.38cm}${\scriptsize ({\blue by the induction hypothesis})}

\item [ ]$=\varphi_{B}\co (H\ot \mu_{B})\co (H\ot \varphi_{B}^{\dagger} \ot B)\co (\delta_{H}\ot (c_{B,B}\co (u_{n-1}^{\varphi_{B}}\ot B))) $ {\scriptsize ({\blue by the factorization of $u_{n-1}^{\varphi_{B}}$ through}}
\item[ ]$\hspace{0.38cm}${\scriptsize {\blue the center of $B$ and the naturality of $c$})}

\item [ ]$=\mu_{B}\co (\varphi_{B}\ot \varphi_{B})\co (H\ot c_{H,B}\ot B)\co  (\delta_{H}\ot \varphi_{B}^{\dagger} \ot B)\co (\delta_{H}\ot (c_{B,B}\co (u_{n-1}^{\varphi_{B}}\ot B))) $ {\scriptsize ({\blue by  (b1) of}}
\item[ ]$\hspace{0.38cm}${\scriptsize {\blue Definition \ref{def}})}

\item [ ]$= \mu_{B}\co ((\varphi_{B}\wedge \varphi_{B}^{\dagger})\ot \varphi_{B})\co (H\ot c_{H,B}\ot B)\co (\delta_{H}\ot (c_{B,B}\co (u_{n-1}^{\varphi_{B}}\ot B)))$ {\scriptsize ({\blue by naturality of $c$ and the}}
\item[ ]$\hspace{0.38cm}${\scriptsize {\blue  coassociativity and cocommuativity of $\delta_{H}$})} 

\item [ ]$=\mu_{B}\co ((\mu_{B}\co (u_{1}^{\varphi_{B}}\ot B))\ot \varphi_{B})\co (H\ot c_{H,B}\ot B)\co (\delta_{H}\ot (c_{B,B}\co (u_{n-1}^{\varphi_{B}}\ot B)))  $ {\scriptsize ({\blue by the $\varphi_{B}$-invertivility}}
\item[ ]$\hspace{0.38cm}${\scriptsize {\blue  of $\varphi_{B}$})}

\item [ ]$= \mu_{B}\co (B\ot (\mu_{B}\co (u_{1}^{\varphi_{B}}\ot \varphi_{B})\co (\delta_{H}\ot B)))\co (c_{H,B}\ot B)\co (H\ot  (c_{B,B}\co (u_{n-1}^{\varphi_{B}}\ot B)))$ {\scriptsize ({\blue by the}}
\item[ ]$\hspace{0.38cm}${\scriptsize {\blue  naturality of $c$, the associativity of $\mu_{B}$ and the factorization of $u_{1}^{\varphi_{B}}$ through the center of $B$})}

\item [ ]$= \mu_{B}\co (B\ot \varphi_{B})\co (c_{H,B}\ot B)\co (H\ot  (c_{B,B}\co (u_{n-1}^{\varphi_{B}}\ot B))) $ {\scriptsize ({\blue by (\ref{nabla-fi})})}

\item [ ]$= \mu_{B}\co c_{B,B}\co (u_{n}^{\varphi_{B}}\ot B) $ {\scriptsize ({\blue by the naturality of $c$})}

\item [ ]$= \mu_{B}\co (u_{n}^{\varphi_{B}}\ot B)  $ {\scriptsize ({\blue by the factorization of $u_{n}^{\varphi_{B}}$ through the center of $B$})}

\end{itemize}
and, therefore $\varphi_{B}^{\ot n}$ is $\varphi_{B}^{\ot n}$-invertible.

\end{proof}

\begin{prop}
\label{varphi-ic}
Let $H$ be a cocommutative weak Hopf algebra. Let $(B,\varphi_{B})$ be a left weak $H$-module algebra and suppose that $\varphi_{B}$ is $\varphi_{B}$-invertible.  Then, $\omega\in Reg_{\varphi_{B}}(H^{\ot n}, B)$ satisfies 
\begin{equation}
\label{sig-cen}
\overline{\omega}\wedge \varphi_{B}^{\ot n}=\overline{\omega}^{op}\wedge \varphi_{B}^{\ot n}
\end{equation}
if and only if  it factors through the center of $B$.
\end{prop}

\begin{proof} Assume that (\ref{sig-cen}) holds. Then, by the associativity of $\mu_{B}$ and $\omega\in Reg_{\varphi_{B}}(H^{\ot n}, B)$, we have 
$$ \overline{\omega}\wedge \varphi_{B}^{\ot n}\wedge  \varphi_{B}^{\ot n \dagger}=\overline{\omega}\wedge  (\mu_{B}\co (u_{n}^{\varphi_{B}}\ot B))=\mu_{B}\co ((\omega \ast u_{n}^{\varphi_{B}})\ot B)=\overline{\omega}.$$

On the other hand, 

\begin{itemize}

\item[ ]$\hspace{0.38cm} \overline{\omega}^{op}\wedge \varphi_{B}^{\ot n} \wedge  \varphi_{B}^{\ot n \dagger}$

\item [ ]$=\overline{\omega}^{op}\wedge (\mu_{B}\co (u_{n}^{\varphi_{B}}\ot B)) $ {\scriptsize ({\blue by the $\varphi_{B}^{\ot n}$-invertivility})}

\item [ ]$=\mu_{B}\co c_{B,B}\co (\omega\ot (\mu_{B}\co c_{B,B}\co (u_{n}^{\varphi_{B}}\ot B)))\co (\delta_{H^{\ot n}}\ot B)$ {\scriptsize ({\blue by the  factorization of $u_{n}^{\varphi_{B}}$ through the center}}
\item[ ]$\hspace{0.38cm}${\scriptsize {\blue of $B$})}

\item [ ]$=\mu_{B}\co (B\ot (\mu_{B}\co c_{B,B}))\co (c_{B,B}\ot B)\co (B\ot c_{B,B})\co (((\omega\ot u_{n}^{\varphi_{B}})\co   \delta_{H^{\ot n}})\ot B)$ {\scriptsize ({\blue by the naturality of}}
\item[ ]$\hspace{0.38cm}${\scriptsize ({\blue $c$ and the associativity of $\mu_{B}$})}

\item [ ]$= \mu_{B}\co c_{B,B}\co ((u_{n}^{\varphi_{B}}\ast w)\ot B)$ {\scriptsize ({\blue by the naturality of $c$ and the cocommutativity of $\delta_{H}$})}

\item [ ]$=\overline{\omega}^{op}$ {\scriptsize ({\blue by $\omega\in Reg_{\varphi_{B}}(H^{\ot n}, B)$})}.

\end{itemize}

Therefore, $\overline{\omega}=\overline{\omega}^{op}$ and, as a consequence, $\omega$ factors through the center of $B$.

Conversely, if  $\omega$ factors through the center of $B$, by Remark \ref{rmexc}, 
 we have that $\overline{\omega}=\overline{\omega}^{op}$ and then (\ref{sig-cen}) holds trivially.

\end{proof}

\begin{prop}
\label{varphi-ic1}
Let $H$ be a cocommutative weak Hopf algebra. Let $(B,\varphi_{B})$ be a left weak $H$-module algebra and suppose that $\varphi_{B}$ is $\varphi_{B}$-invertible.  Then, if  $\omega\in Reg_{\varphi_{B}}(H^{\ot n}, B)$ satisfies (\ref{sig-cen}), $\omega^{-1}$ also satisfies (\ref{sig-cen}). Then, as a consequence, $\omega^{-1}$ factors through the center of $B$.
\end{prop}

\begin{proof} By the equalities of Proposition \ref{prp}, Remark \ref{rmexc} and Proposition \ref{varphi-ic} we have the following:
$$\overline{\omega}\wedge \overline{\omega^{-1}}=\overline{\omega\ast \omega^{-1}}=\overline{u_{n}^{\varphi_{B}}}=\overline{u_{n}^{\varphi_{B}}}^{op}=\overline{\omega^{-1}\ast \omega}^{op}=\overline{\omega}^{op}\wedge \overline{\omega^{-1}}^{op}=\overline{\omega}\wedge \overline{\omega^{-1}}^{op}.$$

Then, as a consequence, we have that
$$\overline{\omega^{-1}}\wedge \varphi_{B}^{\ot n}=\overline{\omega^{-1}}\wedge \overline{u_{n}^{\varphi_{B}}}\wedge \varphi_{B}^{\ot n}=\overline{\omega^{-1}}\wedge \overline{u_{n}^{\varphi_{B}}}^{op}\wedge \varphi_{B}^{\ot n}= \overline{u_{n}^{\varphi_{B}}}^{op}\wedge \overline{\omega^{-1}}\wedge \varphi_{B}^{\ot n}=\overline{u_{n}^{\varphi_{B}}}\wedge \overline{\omega^{-1}}\wedge \varphi_{B}^{\ot n}$$
$$=\overline{\omega^{-1}}\wedge \overline{\omega}\wedge \overline{\omega^{-1}}\wedge \varphi_{B}^{\ot n}=\overline{\omega^{-1}}\wedge \overline{\omega}\wedge \overline{\omega^{-1}}^{op}\wedge \varphi_{B}^{\ot n}=\overline{u_{n}^{\varphi_{B}}}\wedge \overline{\omega^{-1}}^{op}\wedge \varphi_{B}^{\ot n}=\overline{\omega^{-1}}^{op}\wedge \overline{u_{n}^{\varphi_{B}}}\wedge \varphi_{B}^{\ot n}$$
$$=\overline{\omega^{-1}}^{op}\wedge \varphi_{B}^{\ot n}.$$

Therefore, $\omega^{-1}$  satisfies (\ref{sig-cen}) and, by the previous proposition, $\omega^{-1}$ factors through the center of $B$.

\end{proof}

\begin{prop}
\label{Hm-alg}
Let $H$ be a cocommutative weak Hopf algebra. Let $(B,\varphi_{B})$ be a left weak $H$-module algebra and $\sigma \in Reg_{\varphi_{B}}(H^{\ot 2},B)$ satisfying the twisted condition (\ref{t-sigma}). Then, $\varphi_{B}$ induces a left $H$-module algebra structure on the center of $B$, where the action $\varphi_{Z(B)}:H\ot Z(B)\rightarrow Z(B)$ is the factorication of $\varphi_{B}\co (H\ot z_{B}): H\ot Z(B)\rightarrow B$ through the center of $B$.
\end{prop}

\begin{proof} First note that, by (\ref{cen1}), (b1) of Definition \ref{def} and the naturality of $c$, we obtain that the identity 
\begin{equation}
\label{H1}
\mu_{B}\co (\varphi_{B}\ot \varphi_{B})\co (H\ot c_{H,B}\ot B)\co (\delta_{B}\ot z_{B}\ot B)=\mu_{B}\co c_{B,B}\co (\varphi_{B}\ot \varphi_{B})\co (H\ot c_{H,B}\ot B)\co (\delta_{B}\ot z_{B}\ot B)
\end{equation}
holds. Then, on the one hand, 
\begin{itemize}
\item[ ]$\hspace{0.38cm} \mu_{B}\co (\varphi_{B}\ot (\mu_{B}\co (u_{1}^{\varphi_{B}}\ot B)))\co (H\ot c_{H,B}\ot B)\co (\delta_{B}\ot z_{B}\ot B)$

\item [ ]$= \mu_{B}\co (\mu_{B}\ot B)\co (\varphi_{B}\ot c_{B,B})\co (H\ot c_{B,B}\ot B)\co ((( H\ot u_{1}^{\varphi_{B}})\co \delta_{H})\ot z_{B}\ot B) $ {\scriptsize ({\blue by (\ref{cen1}) and the}}
\item[ ]$\hspace{0.38cm}${\scriptsize {\blue naturality of $c$})}

\item [ ]$=\mu_{B}\co (\mu_{B}\ot B)\co (\varphi_{B}\ot c_{B,B})\co (H\ot c_{B,B}\ot B)\co (((c_{B,H}\co ( u_{1}^{\varphi_{B}}\ot H)\co \delta_{H})\ot z_{B}\ot B)  $ {\scriptsize ({\blue by the}}
\item[ ]$\hspace{0.38cm}${\scriptsize {\blue  cocommutativity of $\delta_{H}$ and the naturality of $c$})}

\item [ ]$=\mu_{B}\co (B\ot (\mu_{B}\co c_{B,B}\co (u_{1}^{\varphi_{B}}\ot B)))\co ((c_{H,B}\co (H\ot \varphi_{B})\co (\delta_{H}\ot z_{B}))\ot B) $ {\scriptsize ({\blue by the naturality of $c$}}
\item[ ]$\hspace{0.38cm}${\scriptsize {\blue  and the associativity of $\mu_{B}$})}

\item [ ]$= \mu_{B}\co ((\mu_{B}\co c_{B,B}\co (u_{1}^{\varphi_{B}}\ot \varphi_{B})\co (\delta_{H}\ot z_{B}))\ot B)$ {\scriptsize ({\blue by the naturality of $c$, the associativity of $\mu_{B}$ and (\ref{cen1})})}

\item [ ]$=\mu_{B}\co ((\mu_{B}\co (u_{1}^{\varphi_{B}}\ot \varphi_{B})\co (\delta_{H}\ot z_{B}))\ot B) $ {\scriptsize ({\blue by (\ref{cen1})})}

\item [ ]$= \mu_{B}\co ((\varphi_{B}\co (H\ot z_{B}))\ot B)$ {\scriptsize ({\blue by (\ref{nabla-fi})})}

\end{itemize}

and, on the other hand, 

\begin{itemize}

\item[ ]$\hspace{0.38cm} \mu_{B}\co (\varphi_{B}\ot (\mu_{B}\co (u_{1}^{\varphi_{B}}\ot B)))\co (H\ot c_{H,B}\ot B)\co (\delta_{H}\ot z_{B}\ot B)$

\item [ ]$= \mu_{B}\co (\varphi_{B}\ot (\varphi_{B}\wedge \varphi_{B}^{\dagger}))\co (H\ot c_{H,B}\ot B)\co (\delta_{H}\ot z_{B}\ot B) $ {\scriptsize ({\blue by the $\varphi_{B}$-invertivility of $\varphi_{B}$})}

\item [ ]$= \mu_{B}\co (\varphi_{B}\ot \varphi_{B})\co (H\ot c_{H,B}\ot \varphi_{B}^{\dagger})\co (\delta_{H}\ot c_{H,B}\ot B)\co (\delta_{H}\ot z_{B}\ot B) $ {\scriptsize ({\blue by the  coassociativity of $\delta_{H}$}}
\item[ ]$\hspace{0.38cm}${\scriptsize {\blue  and the naturality of $c$})}

\item [ ]$= \mu_{B}\co c_{B,B}\co (\varphi_{B}\ot \varphi_{B})\co (H\ot c_{H,B}\ot \varphi_{B}^{\dagger})\co (\delta_{H}\ot c_{H,B}\ot B)\co (\delta_{H}\ot z_{B}\ot B) $ {\scriptsize ({\blue by (\ref{H1})})}

\item [ ]$= \mu_{B}\co c_{B,B}\co (\varphi_{B}\ot (\varphi_{B}\wedge \varphi_{B}^{\dagger}))\co (H\ot c_{H,B}\ot B)\co (\delta_{H}\ot z_{B}\ot B)$ {\scriptsize ({\blue by the naturality of $c$ and the}}
\item[ ]$\hspace{0.38cm}${\scriptsize {\blue  coassociativity of $\delta_{H}$})}

\item [ ]$= \mu_{B}\co (\mu_{B}\ot B)\co (B\ot c_{B,B})\co ((c_{B,B}\co (\varphi_{B}\ot u_{1}^{\varphi_{B}})\co (H\ot c_{H,B})\co (\delta_{H}\ot z_{B}))\ot B)$ {\scriptsize ({\blue by the}}
\item[ ]$\hspace{0.38cm}${\scriptsize {\blue $\varphi_{B}$-invertivility of $\varphi_{B}$ and the naturality of $c$})}

\item [ ]$=\mu_{B}\co (\mu_{B}\ot B)\co (B\ot c_{B,B})\co (((u_{1}^{\varphi_{B}}\ot \varphi_{B})\co (\delta_{H}\ot z_{B}))\ot B) $ {\scriptsize ({\blue by the naturality of $c$ and the}}
\item[ ]$\hspace{0.38cm}${\scriptsize {\blue  cocommutativity of $\delta_{H}$})}

\item [ ]$=\mu_{B}\co (B\ot \mu_{B})\co (c_{B,B}\ot B)\co (B\ot c_{B,B})\co (((u_{1}^{\varphi_{B}}\ot \varphi_{B})\co (\delta_{H}\ot z_{B}))\ot B) $ {\scriptsize ({\blue by (\ref{cen1}) and the}}
\item[ ]$\hspace{0.38cm}${\scriptsize {\blue  associativity of $\mu_{B}$})}

\item [ ]$= \mu_{B}\co c_{B,B}\co \co ((\mu_{B}\co (u_{1}^{\varphi_{B}}\ot \varphi_{B})\co (\delta_{H}\ot z_{B}))\ot B)$ {\scriptsize ({\blue by the naturality of $c$})}

\item [ ]$=\mu_{B}\co c_{B,B}\co ((\varphi_{B}\co (H\ot z_{B}))\ot B) $ {\scriptsize ({\blue by (\ref{nabla-fi})}).}

\end{itemize}

Therefore, as a consequence of the previous equalities, we have that 
$$\mu_{B}\co ((\varphi_{B}\co (H\ot z_{B}))\ot B)=\mu_{B}\co c_{B,B}\co ((\varphi_{B}\co (H\ot z_{B}))\ot B)$$
and this implies that there exists a unique morphism $\varphi_{Z(B)}:H\ot Z(B)\rightarrow Z(B)$ such that 
\begin{equation}
\label{zba}
z_{B}\co \varphi_{Z(B)}=\varphi_{B}\co (H\ot z_{B}).
\end{equation}

The pair $(Z(B), \varphi_{Z(B)})$ is a left $H$-module algebra. Indeed, using hat $z_{B}$ is a monomorphism we have that 
$\varphi_{Z(B)}\co (\eta_{H}\ot Z(B))=id_{Z(B)}$ because by (\ref{zba}) $z_{B}\co\varphi_{Z(B)}
\co (\eta_{H}\ot Z(B))=z_{B}$ holds. Also, we have the identity
\begin{equation}
\label{zb1}
\mu_{B}\co ((\varphi_{B}\co (H\ot \varphi_{B}))\ot B)\co 
(H\ot H\ot (c_{B,B}\co (\sigma\ot B)))\co (\delta_{H^{\ot 2}}\ot z_{B})=\mu_{B}\co (B\ot \varphi_{B})\co 
(G_{\sigma}\ot z_{B})
\end{equation}
since
\begin{itemize}
\item[ ]$\hspace{0.38cm} \mu_{B}\co ((\varphi_{B}\co (H\ot \varphi_{B}))\ot B)\co 
(H\ot H\ot (c_{B,B}\co (\sigma\ot B)))\co (\delta_{H^{\ot 2}}\ot z_{B})$

\item[ ]$= \mu_{B}\co (B\ot \sigma)\co (P_{\varphi_{B}}\ot H)\co (H\ot P_{\varphi_{B}})\co (H\ot H\ot z_{B})$ {\scriptsize ({\blue by  the naturality of $c$})}

\item[ ]$= \mu_{B}\co (B\ot \varphi_{B})\co (F_{\sigma}\ot z_{B})$ {\scriptsize ({\blue by (\ref{t-sigma})})}

\item [ ]$= \mu_{B}\co c_{B,B}\co (B\ot (z_{B}\co \varphi_{Z(B)}))\co (F_{\sigma}\ot Z(B))$ {\scriptsize ({\blue by (\ref{cen1}) and (\ref{zba})})}

\item [ ]$=\mu_{B}\co (\varphi_{B}\ot B)\co (H\ot c_{B,B})\ot (((\mu_{H}\ot \sigma)\co c_{H^{\ot 2},H^{\ot 2}}\co \delta_{H^{\ot 2}}) \ot z_{B}) $ {\scriptsize ({\blue by  the naturality of $c$})}

\item [ ]$= \mu_{B}\co (\varphi_{B}\ot B)\co (H\ot c_{B,B})\ot (G_{\sigma}\ot z_{B}) $ {\scriptsize ({\blue by  the cocommutativity of $\delta_{H^{\ot 2}}$}).}

\end{itemize}

Then, 
\begin{itemize}
\item[ ]$\hspace{0.38cm} \mu_{B}\co ((\mu_{B}\co ((\varphi_{B}\co (H\ot \varphi_{B}))\ot B)\co 
(H\ot H\ot (c_{B,B}\co (\sigma\ot B)))\co (\delta_{H^{\ot 2}}\ot B))\ot B)$
\item[ ]$\hspace{0.38cm} \co (H\ot H\ot (c_{B,B}\co (\sigma^{-1}\ot B)))\co (\delta_{H^{\ot 2}}\ot z_{B})$

\item[ ]$= \mu_{B}\co ((\varphi_{B}\co (H\ot \varphi_{B}))\ot B)\co (H\ot H\ot (c_{B,B}\co ((\sigma\ast \sigma^{-1})\ot B)))\co (\delta_{H^{\ot 2}}\ot z_{B})$ {\scriptsize ({\blue by the coassociativity}}
\item[ ]$\hspace{0.38cm}${\scriptsize {\blue of  $\delta_{H^{\ot 2}}$})}

\item[ ]$=\mu_{B}\co ((\varphi_{B}\co (H\ot \varphi_{B}))\ot B)\co (H\ot H\ot (c_{B,B}\co (u_{2}^{\varphi_{B}}\ot B)))\co (\delta_{H^{\ot 2}}\ot z_{B}) $ {\scriptsize ({\blue by $\sigma\in Reg_{\varphi_{B}}(H^{\ot 2}, B)$ })}

\item [ ]$=\mu_{B}\co (\varphi_{B}\ot \varphi_{B})\co (H\ot c_{H,B}\ot B)\co (\delta_{H}\ot ((\varphi_{B}\ot \varphi_{B})\co (H\ot c_{H,B}\ot B)\co (\delta_{H}\ot z_{B}\ot \eta_{B})))
${\scriptsize ({\blue by}}
\item[ ]$\hspace{0.38cm}${\scriptsize {\blue the naturality of $c$})}

\item [ ]$=\varphi_{B}\co (H\ot (\varphi_{B}\co (H\ot (\mu_{B}\co (z_{B}\ot \eta_{B}))))) $ {\scriptsize ({\blue by (b1) of Definition \ref{def}})}

\item [ ]$=\varphi_{B}\co (H\ot (\varphi_{B}\co (H\ot z_{B})))$ {\scriptsize ({\blue by the properties of $\eta_{B}$})}

\item [ ]$=z_{B}\co \varphi_{Z(B)}\co (H\ot \varphi_{Z(B)}) $ {\scriptsize ({\blue by (\ref{zba})})}

\end{itemize}
and, on the other hand, 
\begin{itemize}
\item[ ]$\hspace{0.38cm}  \mu_{B}\co ((\mu_{B}\co (B\ot \varphi_{B})\co 
(G_{\sigma}\ot B))\ot B)\co (H\ot H\ot (c_{B,B}\co (\sigma^{-1}\ot B)))\co (\delta_{H^{\ot 2}}\ot z_{B})$
\item[ ]$=\mu_{B}\co (\varphi_{B}\ot B)\co (\mu_{H}\ot  (c_{B,B}\co ((\sigma\ast \sigma^{-1})\ot B)))\co (\delta_{H^{\ot 2}}\ot z_{B}) $ {\scriptsize ({\blue by the coassociativity of  $\delta_{H^{\ot 2}}$})} 

\item[ ]$=\mu_{B}\co (\varphi_{B}\ot B)\co (\mu_{H}\ot  (c_{B,B}\co (u_{2}^{\varphi_{B}}\ot B)))\co (\delta_{H^{\ot 2}}\ot z_{B}) $ {\scriptsize ({\blue by $\sigma\in Reg_{\varphi_{B}}(H^{\ot 2}, B)$ })}

\item [ ]$=\mu_{B}\co (\varphi_{B}\ot u_{1}^{\varphi_{B}})\co (H\ot c_{H,B})\co (((\mu_{H}\ot \mu_{H})\co \delta_{H^{\ot 2}})\ot z_{B}) $ {\scriptsize ({\blue by the naturality of $c$})}

\item [ ]$=\mu_{B}\co (\varphi_{B}\ot \varphi_{B})\co (H\ot c_{H,B}\ot B)\co ((\delta_{H}\co \mu_{H})\ot z_{B}\ot \eta_{B}) $ {\scriptsize ({\blue by (a1) of Definition \ref{wha}})}

\item [ ]$=\varphi_{B}\co (\mu_{H}\ot (\mu_{B}\co (z_{B}\ot \eta_{B}))  $ {\scriptsize ({\blue by (b1) of Definition \ref{def}})}

\item [ ]$=\varphi_{B}\co (\mu_{H}\ot z_{B})$ {\scriptsize ({\blue by the properties of $\eta_{B}$})}

\item [ ]$= z_{B}\co \varphi_{Z(B)}\co (\mu_{H}\ot Z(B))$ {\scriptsize ({\blue by  (\ref{zba})}).}

\end{itemize}

Therefore, as a consequence of (\ref{zb1}),  we have that $\varphi_{Z(B)}\co (H\ot \varphi_{Z(B)})=\varphi_{Z(B)}\co (\mu_{H}\ot Z(B))$ and this implies that  $(Z(B), \varphi_{Z(B)})$ is a left $H$-module. Finally, it is  a left $H$-module algebra because composing with the monomorphism $z_{B}$ we have 
\begin{itemize}
\item[ ]$\hspace{0.38cm} z_{B}\co \mu_{Z(B)}\co (\varphi_{Z(B)}\ot \varphi_{Z(B)})\co (H\ot c_{H, Z(B)}\ot Z(B))\co (\delta_{H}\ot Z(B)\ot Z(B))$
\item[ ]$= \mu_{B}\co (\varphi_{B}\ot \varphi_{B})\co (H\ot c_{H,B}\ot B)\co (\delta_{H}\ot z_{B}\ot z_{B})$ {\scriptsize ({\blue by (\ref{zbpro}), (\ref{zba}) and the naturality of $c$})} 

\item [ ]$=\varphi_{B}\co (H\ot (\mu_{B}\co (z_{B}\ot \eta_{B})) $ {\scriptsize ({\blue by (b1) of Definition \ref{def}})}

\item [ ]$=\varphi_{B}\co (H\ot (z_{B}\co \mu_{Z(B)}))$ {\scriptsize ({\blue by  (\ref{zbpro})})}

\item [ ]$= z_{B}\co \varphi_{Z(B)}\co (H\ot \mu_{Z(B)})$ {\scriptsize ({\blue by  (\ref{zba})})}

\end{itemize}
and 
\begin{itemize}
\item[ ]$\hspace{0.38cm} z_{B}\co \varphi_{Z(B)}\co (H\ot u_{1}^{\varphi_{Z(B)}})$
\item[ ]$= \varphi_{B}\co (H\ot u_{1}^{\varphi_{B}})$ {\scriptsize ({\blue by (\ref{zbeta}), (\ref{zba})})} 

\item [ ]$=u_{1}^{\varphi_{B}}\co \mu_{H}$ {\scriptsize ({\blue by (b3) of Definition \ref{def}})}

\item [ ]$=z_{B}\co u_{1}^{\varphi_{Z(B)}}\co \mu_{H}$ {\scriptsize ({\blue by  (\ref{zbeta}), (\ref{zba})}).}

\end{itemize}

\end{proof}

\begin{rem}
{\rm Note that, under the conditions of the previous proposition, the equality 
\begin{equation}
\label{udd}
z_{B}\co u_{1}^{\varphi_{Z(B)}}=u_{1}^{\varphi_{B}}
\end{equation}
holds.
}
\end{rem}

\begin{apart}
{\rm Let $H$ be a cocommutative weak Hopf algebra. By  \cite[Theorem 3.1]{nmra6} we know that, if $(A, \varphi_{A})$ is a left weak $H$-module algebra and $\sigma\in Reg_{\varphi_{A}}(H^{\ot 2},A)$ satisfies the twisted condition (\ref{t-sigma}), $(A,\varphi_{A})$ is a left $H$-module algebra if and only if the morphism $\sigma$ factorizes through the center of $A$. Moreover, by \cite[Corollary 3.1]{nmra6}, $(A,\varphi_{A})$ is a left $H$-module algebra if and only if the morphism $u_{2}^{\varphi_{A}}$ satisfies the twisted condition (\ref{t-sigma}).
}
\end{apart} 

\begin{prop}
\label{talsig} 
Let $H$ be a cocommutative weak Hopf algebra and let $(B,\varphi_{B})$ be a left weak $H$-module algebra. Let $\sigma\in Reg_{\varphi_{B}}(H^{\ot 2},B)$ satisfying the twisted condition (\ref{t-sigma}). Then, $\alpha\in Reg_{\varphi_{B}}(H^{\ot 2},B)$ satisfies the twisted condition (\ref{t-sigma}) if, and only if, there exists $\tau\in Reg_{\varphi_{Z(B)}}(H^{\ot 2},Z(B))$ such that 
\begin{equation}
\label{talsig-eq}
\alpha=(z_{B}\co \tau)\ast \sigma.
\end{equation}
\end{prop}

\begin{proof} Suppose that $\alpha$ satisfies the twisted condition  (\ref{t-sigma}). We will see that $\sigma\ast \alpha^{-1}$ factors through the center of $B$. Following Proposition \ref{varphi-ic}, to prove it we will see that $\overline{\sigma\ast \alpha^{-1}}\wedge \varphi_{B}^{\ot 2}=\overline{\sigma\ast \alpha^{-1}}^{op}\wedge \varphi_{B}^{\ot 2}.$ First, note that
$$
\varphi_{B}\co (\mu_{H}\ot B)=\overline{\sigma^{-1}}\wedge \overline{\sigma}^{op}\wedge \varphi_{B}^{\ot 2}
$$
because 
\begin{itemize}
\item[ ]$\hspace{0.38cm} \varphi_{B}\co (\mu_{H}\ot B)$
\item[ ]$= \overline{u_{2}^{\varphi_{B}}}\wedge (\varphi_{B}\co (\mu_{H}\ot B))$ {\scriptsize ({\blue by (a1) of Definition \ref{wha} and (\ref{nabla-fi})})} 

\item [ ]$=\overline{\sigma^{-1}}\wedge \overline{\sigma}\wedge (\varphi_{B}\co (\mu_{H}\ot B))$ {\scriptsize ({\blue by $\sigma\in Reg_{\varphi_{B}}(H^{\ot 2},B)$ and (i) of Proposition \ref{prp}})}

\item [ ]$=\overline{\sigma}\ast \overline{\alpha^{-1}}^{op}\wedge \varphi_{B}^{\ot 2}$ {\scriptsize ({\blue by  (\ref{t-sigmac})}).}
\end{itemize}

Thus, for $\alpha$ we have the same identity and then
\begin{equation}
\label{talsig2}
\overline{\sigma^{-1}}\wedge \overline{\sigma}^{op}\wedge \varphi_{B}^{\ot 2}=\overline{\alpha^{-1}}\wedge \overline{\alpha}^{op}\wedge \varphi_{B}^{\ot 2}
\end{equation}
holds. As a consequence, 
\begin{equation}
\label{talsig3}
\overline{\sigma}^{op}\wedge \varphi_{B}^{\ot 2}= \overline{\alpha}^{op}\wedge  \overline{\sigma\ast \alpha^{-1}} \wedge  \varphi_{B}^{\ot 2}
\end{equation}
also holds since
\begin{itemize}
\item[ ]$\hspace{0.38cm}\overline{\sigma}^{op}\wedge \varphi_{B}^{\ot 2} $
\item[ ]$= \overline{\sigma}^{op}\wedge \overline{u_{2}^{\varphi_{B}}}\wedge  \varphi_{B}^{\ot 2} $ {\scriptsize ({\blue by (v) of Proposition \ref{prp}})} 

\item [ ]$=\overline{u_{2}^{\varphi_{B}}}\wedge  \overline{\sigma}^{op}\wedge  \varphi_{B}^{\ot 2}$ {\scriptsize ({\blue by  (iv) of Proposition \ref{prp}})}

\item [ ]$=\overline{\sigma}\wedge \overline{\sigma^{-1}}\wedge  \overline{\sigma}^{op}\wedge  \varphi_{B}^{\ot 2}$ {\scriptsize ({\blue by $\sigma\in Reg_{\varphi_{B}}(H^{\ot 2},B)$ and (i) of Proposition \ref{prp}})}

\item[ ]$=\overline{\sigma}\wedge \overline{\alpha^{-1}}\wedge  \overline{\alpha}^{op}\wedge  \varphi_{B}^{\ot 2} $ {\scriptsize ({\blue by (\ref{talsig2})})} 

\item [ ]$=\overline{\alpha}^{op}\wedge  \overline{\sigma\ast \alpha^{-1}} \wedge  \varphi_{B}^{\ot 2}$ {\scriptsize ({\blue by  (iv) and (i) of Proposition \ref{prp}}).}

\end{itemize}

Therefore, 
\begin{itemize}
\item[ ]$\hspace{0.38cm}\overline{\sigma\ast \alpha^{-1}}^{op} \wedge \varphi_{B}^{\ot 2} $

\item[ ]$=\overline{\alpha^{-1}}^{op} \wedge \overline{\sigma}^{op}\wedge \varphi_{B}^{\ot 2} $
{\scriptsize ({\blue by (iii) of Proposition \ref{prp}})} 

\item[ ]$=\overline{\alpha^{-1}}^{op}\wedge  \overline{\alpha}^{op}\wedge  \overline{\sigma\ast \alpha^{-1}} \wedge  \varphi_{B}^{\ot 2} $ {\scriptsize ({\blue by \ref{talsig3}})} 

\item [ ]$=\overline{u_{2}^{\varphi_{B}}}^{op}\wedge \overline{\sigma\ast \alpha^{-1}} \wedge  \varphi_{B}^{\ot 2}$ {\scriptsize ({\blue by $\alpha\in Reg_{\varphi_{B}}(H^{\ot 2},B)$ and (iii) of Proposition \ref{prp}})}

\item [ ]$=\overline{u_{2}^{\varphi_{B}}}\wedge \overline{\sigma\ast \alpha^{-1}} \wedge  \varphi_{B}^{\ot 2} $ {\scriptsize ({\blue by Remark \ref{rmexc}})}

\item[ ]$= \overline{\sigma\ast \alpha^{-1}} \wedge \overline{u_{2}^{\varphi_{B}}}\wedge \varphi_{B}^{\ot 2} $ {\scriptsize ({\blue by the factorization of $u_{2}^{\varphi_{B}}$ through the center of $B$})} 

\item [ ]$= \overline{\sigma\ast \alpha^{-1}} \wedge \varphi_{B}^{\ot 2}$ {\scriptsize ({\blue by  (v)  of Proposition \ref{prp}})}
\end{itemize}
and this implies that $\sigma\ast \alpha^{-1}$ factors through the center of $B$.  Then, by Proposition \ref{varphi-ic1}, the morphism $(\sigma\ast \alpha^{-1})^{-1}=\alpha\ast \sigma^{-1}$ also factors through the center of $B$. If $\tau$ is the factorization, we have that $z_{B}\co\tau=\alpha\ast \sigma^{-1}$. Then, (\ref{talsig-eq}) holds.

Conversely, if (\ref{talsig-eq}) holds for $\tau\in Reg_{\varphi_{Z(B)}}(H^{\ot 2},Z(B)),$  we have that 
\begin{itemize}
\item[ ]$\hspace{0.38cm} \overline{\alpha}^{op}\wedge \varphi_{B}^{\ot 2}$
\item[ ]$=\overline{(z_{B}\co \tau)\ast \sigma}^{op}\wedge \varphi_{B}^{\ot 2}  $ {\scriptsize ({\blue by (\ref{talsig-eq})})} 
\item[ ]$= \overline{ \sigma}^{op}\wedge \overline{z_{B}\co \tau}^{op}\wedge\varphi_{B}^{\ot 2} $ {\scriptsize ({\blue by (iii) of Proposition \ref{prp}})} 
\item[ ]$=\overline{ \sigma}^{op}\wedge \overline{z_{B}\co \tau}\wedge\varphi_{B}^{\ot 2}  $ {\scriptsize ({\blue by the factorization through the center of $B$})} 
\item[ ]$= \overline{z_{B}\co \tau}\wedge \overline{ \sigma}^{op}\wedge \varphi_{B}^{\ot 2}  $ {\scriptsize ({\blue by (iv) of Proposition \ref{prp}})} 
\item[ ]$= \overline{z_{B}\co \tau}\wedge \overline{ \sigma}\wedge (\varphi_{B}\co (\mu_{H}\ot B))$ {\scriptsize ({\blue by (\ref{t-sigmac})})} 
\item[ ]$=\overline{(z_{B}\co \tau)\co \sigma}\wedge (\varphi_{B}\co (\mu_{H}\ot B))  $ {\scriptsize ({\blue by (i) of Proposition \ref{prp}})} 
\end{itemize}
and, therefore, $\alpha$ satisfies the twisted condition.
\end{proof}

\section{Cohomological  Obstructions in a Weak Setting} 

In the beginning of this section we review the basic facts about the Sweedler cohomology in a weak setting. This cohomology was introduced in \cite{nmra5} as a generalization of the classical Sweedler cohomology for Hopf algebras \cite{Moss}. The groups $Reg_{\varphi_{B}}(H_{L},B)$ and $Reg_{\varphi_{B}}(H^{\ot n}, B)$, introduced in the previous section, will be  the objects of the corresponding cosimplicial complex.

\begin{apart}
\label{cohlogy}
{\rm 
Let $H$ be a cocommutative weak Hopf algebra and let $(B,\varphi_{B})$ be a left weak $H$-module algebra. Following  \cite{nmra5} we define the coface operators as the group morphisms 
$$\partial_{0,i}:Reg_{\varphi_{B}}(H_{L},B)\rightarrow Reg_{\varphi_{B}}(H,B), \;\; i\in\{0,1\}$$
$$\partial_{0,0}(g)=\varphi_{B}\co (H\ot (g\co p_{L}\co \Pi_{H}^{R}))\co \delta_{H},\;\;
\partial_{0,1}(g)=g\co p_{L},$$
$$\partial_{k,i}:Reg_{\varphi_{B}}(H^{\ot k},B)\rightarrow Reg_{\varphi_{B}}(H^{\ot (k+1)},B), \;\;k\geq 1,\;\; i\in\{0,1,\cdots,k+1\}$$
$$ \partial_{k,i}(\sigma)=\left\{ \begin{array}{l}
 \varphi_{B}\co (H\ot \sigma), \;\;i=0\\
 \\
 \sigma\co (H^{i-1}\ot\mu_{H}\ot H^{k-i}),
 \;\;i\in\{1,\cdots,k\} \\

 \\
\sigma\co (H^{\ot (k-1)}\ot (\mu_{H}\co (H\ot \Pi_{H}^{L}))), \;\;i=k+1 ,
  \end{array}\right.
$$

On the other hand, we define the codegeneracy operators by
$s_{1,0}:Reg_{\varphi_{B}}(H,B)\rightarrow Reg_
{\varphi_{B}}(H_{L},B),$
$$s_{1,0}(h)=h\co i_{L}, $$
 and
$s_{k+1,i}:Reg_{\varphi_{B}}(H^{\ot (k+1)},B)\rightarrow Reg_
{\varphi_{B}}(H^{\ot k},B), \;\;k\geq 1,\;\; i\in\{0,1,\cdots,k\},$
\vspace{0.05cm}
$$
s_{k+1,i}(\sigma)=\sigma\co (H^{\ot i}\ot \eta_{H}\ot H^{\ot (k-i)}).
$$

Taking into account the codegeneracy operators, we define the groups 
$$Reg_{\varphi_{B}}^{+}(H^{\ot (k+1)}, B)=\bigcap_{i=0}^{k}Ker(s_{k+1,i}),$$
$$Reg_{\varphi_{B}}^{+}(H_{L}, B)=\{g \in Reg_{\varphi_{B}}(H_{L}, B)\; ;
\; g\co\eta_{H_{L}}=\eta_{B}\}.$$

Note that $Reg_{\varphi_{B}}^{+}(H, B)$ is the group $Reg_{\varphi_{B}}^{t}(H, B)$ introduced in Definition \ref{reg-1+}  because 
$$Reg_{\varphi_{B}}^{+}(H, B)=Ker(s_{1,0})=\{h \in Reg_{\varphi_{B}}(H, B)\; ;
\; h\co i_{H}^{L}=u_{0}^{\varphi_{B}}\}$$
$$=\{h \in Reg_{\varphi_{A}}(H, B)\; ;
\; h\co \Pi_{H}^{L}=u_{1}^{\varphi_{B}}\}=\{h \in Reg_{\varphi_{B}}(H, B)\; ;
\; h\co\eta_{H}=\eta_{B}\}=Reg_{\varphi_{B}}^{t}(H, B).$$

Also, $Reg_{\varphi_{B}}^{+}(H^{\ot 2}, B)$ is the subgroup of $Reg_{\varphi_{B}}(H^{\ot 2}, B)$  formed by the elements satisfying the normal condition (\ref{normal-sigma}) because
$$Reg_{\varphi_{B}}^{+}(H^{\ot 2}, B)=Ker(s_{2,0})\cap Ker(s_{2,1})$$
$$=\{\sigma \in Reg_{\varphi_{B}}(H^{\ot 2}, B)\;; \; \sigma\co
(\eta_{H}\ot H)=\sigma\co (H\ot \eta_{H})=u_{1}^{\varphi_{B}}\}$$
and finally 
$$Reg_{\varphi_{B}}^{+}(H^{\ot 3}, B)=Ker(s_{3,0})\cap Ker(s_{3,1})\cap Ker(s_{3,2})$$
$$=\{\sigma \in Reg_{\varphi_{B}}(H^{\ot 3}, B)\;; \; \sigma\co
(\eta_{H}\ot H\ot H)=\sigma\co (H\ot \eta_{H}\ot H)=\sigma\co (H\ot H\ot \eta_{H})=u_{2}^{\varphi_{B}}\}.$$

}
\end{apart}

\begin{apart}
\label{cohlogy0}
{\rm 
Let $H$ be a cocommutative weak Hopf algebra. If $(A,\varphi_{A})$ is a left  $H$-module algebra, by \cite{nmra5} the groups
$Reg_{\varphi_{A}}(H_{L},A)$ and $Reg_{\varphi_{A}}(H^{\ot n}, A)$, $n \geq, 
1$ are the objects of a cosimplicial complex of groups with the previous coface and codegeneracy
operators. In this case,  $$D^{k}_{\varphi_{A}}=\partial_{k,0}\ast
\partial_{k,1}^{-1}\ast \cdots \ast
\partial_{k,k+1}^{{(-1)}^{k+1}}$$ denote the coboundary morphisms of
the cochain complex
$$Reg_{\varphi_{A}}(H_{L},A)\stackrel{D^{0}_{\varphi_{A}}}\longrightarrow
Reg_{\varphi_{A}}(H,A)\stackrel{D^{1}_{\varphi_{A}}}\longrightarrow
Reg_{\varphi_{A}}(H^{\ot 2},A)
\stackrel{D^{2}_{\varphi_{A}}}\longrightarrow \cdots $$
$$
\cdots \stackrel{D^{k-1}_{\varphi_{A}}}\longrightarrow
Reg_{\varphi_{A}}(H^{\ot k},A)\stackrel{D^{k}_{\varphi_{A}}}\longrightarrow
Reg_{\varphi_{A}}(H^{\ot (k+1)},A)
\stackrel{D^{k+1}_{\varphi_{A}}}\longrightarrow \cdots$$
associated to the cosimplicial complex
$Reg_{\varphi_{A}}(H^{\ot \bullet}, A)$. 
}
\end{apart}

\begin{apart}
\label{cohlogy1}
{\rm 
Let $H$ be a weak Hopf algebra. If $(B,\varphi_{B})$ is a left weak $H$-module algebra and if $\sigma\in Reg_{\varphi_{B}}(H^{\ot 2}, B)$, by \cite[Proposition 5.5]{Gu2-Val},  the morphism  $E(\sigma):H^{\ot 3}\rightarrow B$ defined by $E(\sigma)=\sigma\ot \varepsilon_{H}$ satisfies the following identities:
\begin{equation}
\label{G2V}
E(\sigma)\ast u_{3}^{\varphi_{B}}=u_{3}^{\varphi_{B}}\ast E(\sigma)=\partial_{2,3}(\sigma).
\end{equation}

Then, using that $\partial_{2,3}$ is a group morphism, we have 
$$u_{3}^{\varphi_{B}}=\partial_{2,3}(\sigma)^{-1}\ast u_{3}^{\varphi_{B}}\ast E(\sigma)=\partial_{2,3}(\sigma)^{-1}\ast E(\sigma).$$

Therefore, 
\begin{equation}
\label{G3V}
\partial_{2,3}(\sigma)= E(\sigma).
\end{equation}

Similarly, 
\begin{equation}
\label{G31V}
\partial_{2,3}(\sigma^{-1})=E(\sigma^{-1})\ast u_{3}^{\varphi_{B}}=u_{3}^{\varphi_B}\ast E(\sigma)= E(\sigma^{-1}),
\end{equation}
where $E(\sigma^{-1})=\sigma^{-1}\ot \varepsilon_{H}$.

Then, if $(B,\varphi_{B})$ is a left $H$-module algebra and $H$ is cocommutative,  by (\ref{G2V}), the second coboundary morphism of the cosimplicial complex $Reg_{\varphi_{B}}(H^{\ot \bullet}, B)$ admits the following form:
\begin{equation}
\label{G2V1}
D^{2}_{\varphi_{B}}(\sigma)=\partial_{2,0}(\sigma)\ast
\partial_{2,1}(\sigma^{-1})\ast \partial_{2,2}(\sigma)\ast E(\sigma^{-1}).
\end{equation}
}
\end{apart}

\begin{apart}
\label{cohlogy2}
{\rm 
Let $H$ be a cocommutative weak Hopf algebra. If $(A, \varphi_{A})$ is  a commutative left $H$-module
algebra, $(Reg_{\varphi_{A}}(H^{\ot \bullet},A),D^{\bullet}_{\varphi_{A}})$ gives
the Sweedler cohomology of $H$ in $(A,\varphi_{A})$. Therefore,  the
kth group will be defined by
$$\displaystyle {\mathcal H}_{\varphi_{A}}^k (H,A)={\frac{Ker(D^{k}_{\varphi_{A}})}{Im(D^{k-1}_{\varphi_{A}})}}$$ for
$k\geq 1$. 

The normalized cochain subcomplex  of $A$, denoted by  
$(Reg_{\varphi_{A}}^{+}(H^{\ot \bullet},B),D^{\bullet +}_{\varphi_{A}})$, is
defined by the groups $Reg_{\varphi_{A}}^{+}(H^{\ot (k+1)}, A)$, $Reg_{\varphi_{A}}^{+}(H_{L}, A)$
with $D^{k +}_{\varphi_{A}}$ the restriction of $D^{k}_{\varphi_{A}}$
to $Reg_{\varphi_{A}}^{+}(H^{\ot \bullet}, A)$.

We have that $(Reg_{\varphi_{A}}^{+}(H^{\ot \bullet}, A), D^{\bullet
+}_{\varphi_{A}})$,  is a subcomplex of
$(Reg_{\varphi_{A}}(H^{\ot \bullet},A),D^{\bullet}_{\varphi_{A}})$ and 
the injection map between $(Reg_{\varphi_{A}}^{+}(H^{\ot \bullet}, A), D^{\bullet
+}_{\varphi_{A}})$ and 
$(Reg_{\varphi_{A}}(H^{\ot \bullet},A),D^{\bullet}_{\varphi_{A}})$ induces an isomorphism of cohomology $$I_{\varphi_{A}}^{k}:\displaystyle {\mathcal H}_{\varphi_{A}}^k (H,A)\rightarrow {\mathcal H}_{\varphi_{A}}^{k+} (H,A).$$ 

}
\end{apart}

\begin{apart}
\label{cohlogy3}
{\rm  Assume that $H$ is a  weak Hopf algebra,
let $(B,\varphi_{B})$ be a left weak $H$-module algebra and let $\sigma,
\tau\in Reg_{\varphi_{B}}^{+}(H^{\ot 2},B)$ satisfying the twisted
condition (\ref{t-sigma}) and the 2-cocycle condition
(\ref{c-sigma}). Then by Theorem \ref{equiv-hh-al},   $(B\ot H, \mu_{B\ot_{\varphi_{B}}^{\sigma}H})$ and $(B\ot H, \mu_{B\ot_{\varphi_{B}}^{\tau}H})$ are equivalent if, and only if, there exists $h\in Reg_{\varphi_{B}}^{+}(H,B)$ satisfying (\ref{cbch}) and (\ref{n3}). Then, if $H$ is cocommutative,  by \cite[Corollary 4.8, Theorem 4.9]{nmra5}, $(B\ot H, \mu_{B\ot_{\varphi_{B}}^{\sigma}H})$ and $(B\ot H, \mu_{B\ot_{\varphi_{B}}^{\tau}H})$ are equivalent if, and only if, there exists $h\in Reg_{\varphi_{B}}^{+}(H,B)$ such that the equalities (\ref{cbch}) and 
\begin{equation}
\label{cohomologous-1} \sigma\ast \partial_{1,1}(h)=\partial_{1,0}(h)\ast \partial_{1,2}(h) \ast \tau,
\end{equation}
hold. Note that the equality (\ref{cbch}) is always true if $B$
is commutative and $H$ is cocommutative. Then, under these conditions, if $(B,\varphi_{B})$
is a  left $H$-module algebra, the equivalence between  two weak crossed products 
$(B\ot H, \mu_{B\ot_{\varphi_{B}}^{\sigma}H})$ and $(B\ot H, \mu_{B\ot_{\varphi_{B}}^{\tau}H})$ is determined by the existence of $h$ in $Reg_{\varphi_{B}}^{+}(H,B)$
satifying the equality (\ref{cohomologous-1}). In this case, if $\sigma$ and $\tau\in Reg_{\varphi_{B}}^{+}(H^{\ot 2},B )$,  (\ref{cohomologous-1}) is equivalent to say that 
$$\sigma\ast \tau^{-1}\in Im(D^{1+}_{\varphi_{B}}),$$ 
i.e., $[\sigma]=[\tau]$ in ${\mathcal H}_{\varphi_{B}}^{2+} (H,B)$.
}
\end{apart}

\begin{apart}
\label{cohlogy4}
{\rm Let  $H$ be a  weak Hopf algebra, let $(B,\varphi_{B})$ be a left weak $H$-module algebra and let $\sigma \in Reg_{\varphi_{B}}(H^{\ot 2},B)$. Then, using the coface operators, it is an easy exercise to prove that $\sigma$ satisfy the cocycle condition (\ref{c-sigma}) if and only if 
\begin{equation}
\label{deric-coc}
\partial_{2,0}(\sigma)\ast \partial_{2,2}(\sigma)=\partial_{2,3}(\sigma)\ast \partial_{2,1}(\sigma)
\end{equation}
holds. Then, by (\ref{G2V}), we have that $\sigma$ satisfy the cocycle condition (\ref{c-sigma}) if and only if $\sigma$ satisfies the equality
\begin{equation}
\label{deric-coc1}
\partial_{2,0}(\sigma)\ast \partial_{2,2}(\sigma)=E(\sigma)\ast \partial_{2,1}(\sigma).
\end{equation}
}
\end{apart}

\begin{defin}
{\rm Let  $H$ be a  cocommutative weak Hopf algebra, let $(B,\varphi_{B})$ be a left weak $H$-module algebra and let $\sigma \in Reg_{\varphi_{B}}(H^{\ot 2},B)$.  We define the pre-obstruction of $\sigma$ as the morphism $w_{\sigma}:H^{\ot 3}\rightarrow B$, where 
$$
w_{\sigma}=\partial_{2,0}(\sigma)\ast \partial_{2,2}(\sigma)\ast \partial_{2,1}(\sigma)^{-1}\ast \partial_{2,3}(\sigma)^{-1}.
$$

The using that $\partial_{2,1}(\sigma)^{-1}=\partial_{2,1}(\sigma^{-1})$ and $\partial_{2,3}(\sigma)^{-1}=\partial_{2,3}(\sigma^{-1})$, by the previous considerations, we have that  $\sigma$ satisfies the cocycle condition (\ref{c-sigma}) if and only if  $w_{\sigma}=u_{3}^{\varphi_{B}}$. Also, note that by (\ref{G31V}), we have 
$$
w_{\sigma}=\partial_{2,0}(\sigma)\ast \partial_{2,2}(\sigma)\ast \partial_{2,1}(\sigma^{-1})\ast E(\sigma^{-1}).
$$

Note that $\omega_{\sigma}\in Reg_{\varphi_{B}}(H^{\ot 3},B)$ and it is easy to show that 
\begin{equation}
\label{WSR}
\omega_{\sigma}=S_{\sigma}\ast R_{\sigma},
\end{equation}
where 
$$S_{\sigma}=\partial_{2,0}(\sigma)\ast \partial_{2,2}(\sigma)=\mu_{B}\co (B\ot \sigma)\co (P_{\varphi_{B}}\ot H)\co (H\ot F_{\sigma})$$
and 
$$R_{\sigma}=\partial_{2,1}(\sigma^{-1})\ast E(\sigma^{-1})=\mu_{B}\co (\sigma^{-1}\ot B)\co (H\ot c_{B,H})\co (G_{\sigma^{-1}}\ot H),$$
are morphisms in $Reg_{\varphi_{B}}(H^{\ot 3}, B)$, where 
$$S_{\sigma}^{-1}=\mu_{B}\co (\sigma^{-1}\ot \varphi_{B})\co (H\ot c_{H,H}\ot B)\co (\delta_{H}\ot G_{\sigma^{-1}})$$
and 
$$R_{\sigma}^{-1}=\mu_{B}\co (B\ot \sigma)\co (F_{\sigma}\ot H).$$
}
\end{defin}

\begin{prop}
\label{34w}
Let  $H$ be a  cocommutative weak Hopf algebra, let $(B,\varphi_{B})$ be a left weak $H$-module algebra and let $\sigma \in Reg_{\varphi_{B}}(H^{\ot 2},B)$. Let $\omega_{\sigma}$ be the pre-obstruction of $\sigma$. Then 
\begin{equation}
\label{34w1}
\partial_{3,4}(\omega_{\sigma})=\omega_{\sigma}\otimes \varepsilon_{H}.
\end{equation}
\end{prop}

\begin{proof}
Note that by (\ref{WSR}) we know that 
$\omega_{\sigma}=S_{\sigma}\ast R_{\sigma}$. Then, if we prove that 
$$\partial_{3,4}(S_{\sigma})=S_{\sigma}\otimes \varepsilon_{H}, \;\;\; \partial_{3,4}(R_{\sigma})=R_{\sigma}\otimes \varepsilon_{H}$$
hold, we obtain (\ref{34w1}) since
$$\partial_{3,4}(\omega_{\sigma})=\partial_{3,4}(S_{\sigma}\ast R_{\sigma})=\partial_{3,4}(S_{\sigma})\ast \partial_{3,4}(R_{\sigma})= (S_{\sigma}\otimes \varepsilon_{H})\ast  (R_{\sigma}\otimes \varepsilon_{H})= (S_{\sigma}\ast R_{\sigma})\ot \varepsilon_{H}=\omega_{\sigma}\otimes \varepsilon_{H}.$$

First note that  
\begin{itemize}
\item[ ]$\hspace{0.38cm}\partial_{3,4}(S_{\sigma})$

\item[ ]$= \mu_{B}\co (B\ot \sigma)\co (P_{\varphi_{B}}\ot H)\co (H\ot ((\sigma\ot \mu_{H})\co (H\ot c_{H,H}\ot H)\co (\delta_{H}\ot (((\mu_{H}\co (H\ot \Pi_{H}^{L}))\ot (\mu_{H}$
\item[ ]$\hspace{0.38cm}\co  (H\ot \Pi_{H}^{L})))\co \delta_{H^{\ot 2}})))) $ {\scriptsize ({\blue by  (i) of \cite[Proposition 2.6]{nmra5}, (a1) of Definition \ref{wha} and the naturality of $c$})}

\item[ ]$= \mu_{B}\co (B\ot \partial_{3,2}(\partial_{2,3}(\sigma)))\co ((P_{\varphi_{B}}\co (H\ot \partial_{2,3}(\sigma)))\ot H\ot H\ot H)\co (H\ot 
 \delta_{H^{\ot 3}})$ {\scriptsize ({\blue by the naturality of $c$}}
\item[ ]$\hspace{0.38cm}${\scriptsize {\blue and the associativity of $\mu_{H}$})}

\item[ ]$=S_{\sigma}\otimes \varepsilon_{H} $ {\scriptsize ({\blue by (\ref{G3V}) and the naturality of $c$}).}
\end{itemize}

Finally, $\partial_{3,4}(R_{\sigma})=R_{\sigma}\otimes \varepsilon_{H}$ holds because:

\begin{itemize}
\item[ ]$\hspace{0.38cm}\partial_{3,4}(R_{\sigma})$

\item[ ]$= \mu_{B}\co (\partial_{2,3}(\sigma^{-1})\ot B)\co (H\ot H\ot c_{B,H})\co (H\ot c_{B,H}\ot H)\co (G_{\sigma^{-1}}\ot H\ot H)$ {\scriptsize ({\blue by the naturality of $c$}}

\item[ ]$= R_{\sigma}\otimes \varepsilon_{H}$  {\scriptsize ({\blue by  the naturality of $c$ and (\ref{G3V}) for $\sigma^{-1}$}).}

\end{itemize}

\end{proof}

\begin{prop}
\label{34w}
Let  $H$ be a cocommutative weak Hopf algebra, let $(B,\varphi_{B})$ be a left weak $H$-module algebra and let $\sigma \in Reg_{\varphi_{B}}(H^{\ot 2},B)$ satisfying the twisted condition (\ref{t-sigma}). Then, the pre-obstruction of $\sigma$ factors through the center of $B$.

\end{prop}

\begin{proof} We will use Proposition \ref{varphi-ic} to obtain that $\omega_{\sigma}$ factors through the center of $B$. To prove that 
$$\overline{\omega_{\sigma}}\wedge \varphi_{B}^{\ot 3}=\overline{\omega_{\sigma}}^{op}\wedge \varphi_{B}^{\ot 3}$$
we first see 
\begin{equation}
\label{wegd-1}
\overline{\partial_{2,0}(\sigma)\ast \partial_{2,2}(\sigma)}^{op}\wedge \varphi_{B}^{\ot 3}=\overline{\partial_{2,0}(\sigma)\ast \partial_{2,2}(\sigma)}\wedge (\varphi_{B}\co (m_{H}^{\ot 3}\ot B))
\end{equation}
and 
\begin{equation}
\label{wegd-2}
\overline{\partial_{2,3}(\sigma)\ast \partial_{2,1}(\sigma)}^{op}\wedge \varphi_{B}^{\ot 3}=\overline{\partial_{2,3}(\sigma)\ast \partial_{2,1}(\sigma)}\wedge (\varphi_{B}\co (m_{H}^{\ot 3}\ot B)).
\end{equation}

Indeed: 
\begin{itemize}
\item[ ]$\hspace{0.38cm}\overline{\partial_{2,0}(\sigma)\ast \partial_{2,2}(\sigma)}^{op}\wedge \varphi_{B}^{\ot 3} $

\item[ ]$=\overline{\partial_{2,2}(\sigma)}^{op}\wedge  \overline{\partial_{2,0}(\sigma)}^{op}\wedge \varphi_{B}^{\ot 3}  $ {\scriptsize ({\blue by (iii) of Proposition \ref{prp}})}

\item[ ]$=\overline{\partial_{2,2}(\sigma)}^{op}\wedge (\varphi_{B}\co (H\ot (\overline{\sigma}^{op}\wedge \varphi_{B}^{\ot 2})))$  {\scriptsize ({\blue by  the naturality of $c$, (b1) of Definition \ref{def} and the cocommutativity}}
\item[ ]$\hspace{0.38cm}${\scriptsize {\blue of $\delta_{H}$})}

\item[ ]$=\overline{\partial_{2,2}(\sigma)}^{op}\wedge (\varphi_{B}\co (H\ot (\overline{\sigma}\wedge (\varphi_{B}\co (\mu_{H}\ot B)))))$ {\scriptsize ({\blue by (\ref{t-sigmac})})}

\item[ ]$=\overline{\partial_{2,2}(\sigma)}^{op}\wedge \overline{\partial_{2,0}(\sigma)}\wedge (\varphi_{B}\co (H\ot (\varphi_{B}\co (\mu_{H}\ot B))))$  {\scriptsize ({\blue by  (vii) of Proposition \ref{prp}})}

\item[ ]$=\overline{\partial_{2,0}(\sigma)}\wedge \overline{\partial_{2,2}(\sigma)}^{op}\wedge (\varphi_{B}^{\ot 2}\co (H\ot \mu_{H}\ot B))$ {\scriptsize ({\blue by  (iv) of Proposition \ref{prp}})}

\item[ ]$=\overline{\partial_{2,0}(\sigma)}\wedge \overline{\partial_{2,2}(\sigma)}\wedge (\varphi_{B}\co (m_{H}^{\ot 3}\ot B)) $  {\scriptsize ({\blue by (a1) of Definition \ref{wha}, the naturality of $c$ and (\ref{t-sigmac})})}

\item[ ]$=\overline{\partial_{2,0}(\sigma)\ast \partial_{2,2}(\sigma)}\wedge (\varphi_{B}\co (m_{H}^{\ot 3}\ot B))  $  {\scriptsize ({\blue by  the naturality of $c$})} 

\end{itemize}
and 
\begin{itemize}
\item[ ]$\hspace{0.38cm} \overline{\partial_{2,3}(\sigma)\ast \partial_{2,1}(\sigma)}^{op}\wedge \varphi_{B}^{\ot 3}$

\item[ ]$=\overline{\partial_{2,1}(\sigma)}^{op} \wedge \overline{\partial_{2,3}(\sigma)}^{op}\wedge \varphi_{B}^{\ot 3} $ {\scriptsize ({\blue by (iii) of Proposition \ref{prp}})}

\item[ ]$=\overline{\partial_{2,1}(\sigma)}^{op} \wedge ((\overline{\sigma}^{op}\wedge \varphi_{B}^{\ot 2})\co (H\ot H\ot \varphi_{B}))$  {\scriptsize ({\blue by  the naturality of $c$, the counit properties and (\ref{G3V})})}

\item[ ]$= \overline{\partial_{2,1}(\sigma)}^{op} \wedge ((\overline{\sigma}\wedge (\varphi_{B}\co (\mu_{H}\ot B)))\co (H\ot H\ot \varphi_{B})) $ {\scriptsize ({\blue by (\ref{t-sigmac})})}

\item[ ]$=\overline{\partial_{2,1}(\sigma)}^{op} \wedge  \overline{\partial_{2,3}(\sigma)}\wedge (\varphi_{B}\co (\mu_{H}\ot \varphi_{B}))$  {\scriptsize ({\blue by the naturality of $c$ and the counit properties})}

\item[ ]$= \overline{\partial_{2,3}(\sigma)}\wedge \overline{\partial_{2,1}(\sigma)}^{op} \wedge (\varphi_{B}\co (\mu_{H}\ot \varphi_{B})) $ {\scriptsize ({\blue by (iv) of Proposition \ref{prp}})}

\item[ ]$= \overline{\partial_{2,3}(\sigma)}\wedge ((\overline{\sigma}^{op}\wedge \varphi_{B}^{\ot 2})\co (\mu_{H}\ot H\ot B))$  {\scriptsize ({\blue by  the naturality of $c$ and (a1) of Definition \ref{wha}})}

\item[ ]$=\overline{\partial_{2,3}(\sigma)}\wedge ((\overline{\sigma}\wedge (\varphi_{B}\co (\mu_{H}\ot B)))\co (\mu_{H}\ot H\ot B)) $  {\scriptsize ({\blue by (\ref{t-sigmac})})} 

\item[ ]$=\overline{\partial_{2,3}(\sigma)}\wedge \overline{\partial_{2,1}(\sigma)}\wedge (\varphi_{B}\co (m_{H}^{\ot 3}\ot B)) $  {\scriptsize ({\blue by  the naturality of $c$ and (a1) of Definition \ref{wha}})} 

\item[ ]$=\overline{\partial_{2,3}(\sigma)\ast \partial_{2,1}(\sigma)}\wedge (\varphi_{B}\co (m_{H}^{\ot 3}\ot B))$  {\scriptsize ({\blue by (i) of Proposition \ref{prp}}).} 

\end{itemize}

Also, 
\begin{itemize}
\item[ ]$\hspace{0.38cm} \overline{\partial_{2,3}(\sigma)\ast \partial_{2,1}(\sigma)}\wedge \overline{\partial_{2,1}(\sigma^{-1})\ast \partial_{2,3}(\sigma^{-1})}^{op}\wedge (\varphi_{B}\co (m_{H}^{\ot 3}\ot B))$

\item[ ]$= \overline{\partial_{2,1}(\sigma^{-1})\ast \partial_{2,3}(\sigma^{-1})}^{op}\wedge \overline{\partial_{2,3}(\sigma)\ast \partial_{2,1}(\sigma)}\wedge (\varphi_{B}\co (m_{H}^{\ot 3}\ot B))$ 
{\scriptsize ({\blue by (iv) of Proposition \ref{prp}})}

\item[ ]$=\overline{\partial_{2,1}(\sigma^{-1})\ast \partial_{2,3}(\sigma^{-1})}^{op}\wedge \overline{\partial_{2,3}(\sigma)\ast \partial_{2,1}(\sigma)}^{op}\wedge \varphi_{B}^{\ot 3}$ 
{\scriptsize ({\blue by \ref{wegd-2}})}

\item[ ]$=\overline{\partial_{2,3}(\sigma)\ast \partial_{2,1}(\sigma)\ast \partial_{2,1}(\sigma^{-1})\ast \partial_{2,3}(\sigma^{-1})}^{op}\wedge \varphi_{B}^{\ot 3} $ {\scriptsize ({\blue by (iii) of Proposition \ref{prp}})}

\item[ ]$=\overline{u_{3}^{\varphi_{B}}}^{op}\wedge  \varphi_{B}^{\ot 3}$ {\scriptsize ({\blue by the property of group morphism for $\partial_{2,1}$ and $\partial_{2,3}$})}

\item[ ]$=\varphi_{B}^{\ot 3} $ {\scriptsize ({\blue by (v) of Proposition \ref{prp}})}

\end{itemize}
and, as a consequence, the following identity holds:
\begin{equation}
\label{apo}
\overline{\partial_{2,1}(\sigma^{-1})\ast \partial_{2,3}(\sigma^{-1})}^{op}\wedge (\varphi_{B}\co (m_{H}^{\ot 3}\ot B))=\overline{\partial_{2,1}(\sigma^{-1})\ast \partial_{2,3}(\sigma^{-1})}\wedge \varphi_{B}^{\ot 3}.
\end{equation}

Therefore, 
\begin{itemize}
\item[ ]$\hspace{0.38cm} \overline{\omega_{\sigma}}^{op}\wedge \varphi_{B}^{\ot 3}$
\item[ ]$=\overline{\partial_{2,1}(\sigma^{-1})\ast \partial_{2,3}(\sigma^{-1})}^{op}\wedge \overline{\partial_{2,0}(\sigma)\ast \partial_{2,2}(\sigma)}^{op}\wedge \varphi_{B}^{\ot 3}  $ {\scriptsize ({\blue by (iii) of Proposition  \ref{prp}})}
\item[ ]$=\overline{\partial_{2,1}(\sigma^{-1})\ast \partial_{2,3}(\sigma^{-1})}^{op}\wedge \overline{\partial_{2,0}(\sigma)\ast \partial_{2,2}(\sigma)}\wedge  (\varphi_{B}\co (m_{H}^{\ot 3}\ot B))$ {\scriptsize ({\blue by (\ref{wegd-1})})}
\item[ ]$=\overline{\partial_{2,0}(\sigma)\ast \partial_{2,2}(\sigma)}\wedge \overline{\partial_{2,1}(\sigma^{-1})\ast \partial_{2,3}(\sigma^{-1})}^{op}\wedge (\varphi_{B}\co (m_{H}^{\ot 3}\ot B)) $ {\scriptsize ({\blue by (iv) of Proposition  \ref{prp}})}
\item[ ]$=\overline{\partial_{2,0}(\sigma)\ast \partial_{2,2}(\sigma)}\wedge \overline{\partial_{2,1}(\sigma^{-1})\ast \partial_{2,3}(\sigma^{-1})}\wedge \varphi_{B}^{\ot 3}$ {\scriptsize ({\blue by (\ref{apo})})}
\item[ ]$=\overline{\omega_{\sigma}}\wedge \varphi_{B}^{\ot 3} $ {\scriptsize ({\blue by (i) of Proposition \ref{prp}})}
\end{itemize}
\end{proof}

\begin{defin}
{\rm Let  $H$ be a cocommutative weak Hopf algebra, let $(B,\varphi_{B})$ be a left weak $H$-module algebra and let $\sigma \in Reg_{\varphi_{B}}(H^{\ot 2},B)$ satisfying the twisted condition (\ref{t-sigma}). 
The obstruction of $\sigma$ is defined as the unique morphism $\theta_{\sigma}: H^{\ot 3}\rightarrow Z(B)$ such that $z_{B}\co \theta_{\sigma}=\omega_{\sigma}$, where $\omega_{\sigma}$ is the pre-obstruction of $\sigma$. 

Note that, by the previous proposition, we can assure that  $\theta_{\sigma}$ exists. Also, $\theta_{\sigma}\in Reg_{\varphi_{Z(B)}}(H^{\ot 3}, Z(B))$.

}
\end{defin}

\begin{teo}
\label{3cocy}
Let  $H$ be a cocommutative weak Hopf algebra, let $(B,\varphi_{B})$ be a left weak $H$-module algebra and let $\sigma \in Reg_{\varphi_{B}}(H^{\ot 2},B)$ satisfying the twisted condition (\ref{t-sigma}). Then, the pre-obstruction of $\sigma$ is a 3-cocycle, i.e., the following equality holds:
\begin{equation}
\label{3cs}
\partial_{3,0}(\omega_{\sigma})\ast \partial_{3,2}(\omega_{\sigma})\ast \partial_{3,4}(\omega_{\sigma})=
\partial_{3,1}(\omega_{\sigma})\ast \partial_{3,3}(\omega_{\sigma}).
\end{equation}

\end{teo}

\begin{proof} In order to prove the theorem we will see some equalities. First of all observe that by the definition of the pre-obstruction $\omega_{\sigma}$ we have: 
\begin{equation}
\label{3cs1}
\partial_{2,0}(\sigma)\ast \partial_{2,2}(\sigma)=
\omega_{\sigma}\ast \partial_{2,3}(\sigma)\ast \partial_{2,1}(\sigma).
\end{equation}

Now using that $\partial_{3,2}$ is group morphism we have that 
\begin{equation}
\label{3cs2}
\partial_{3,2}(\partial_{2,0}(\sigma))\ast \partial_{3,2}(\partial_{2,2}(\sigma))=
\partial_{3,2}(\omega_{\sigma})\ast \partial_{3,2}(\partial_{2,3}(\sigma))\ast \partial_{3,2}(\partial_{2,1}(\sigma)).
\end{equation}

But observe that, by the associativity of $\mu_{H}$ and (\ref{G3V}), we have 
\begin{equation}
\label{3cs3}
\partial_{3,2}(\partial_{2,3}(\sigma))= \partial_{3,4}(\partial_{2,2}(\sigma)),
\end{equation}
\begin{equation}
\label{3cs31}
\partial_{3,3}(\partial_{2,2}(\sigma))= \partial_{3,2}(\partial_{2,2}(\sigma)),
\end{equation}
and, trivially, 
\begin{equation}
\label{3cs4}
\partial_{3,2}(\partial_{2,0}(\sigma))= \partial_{3,0}(\partial_{2,1}(\sigma)).
\end{equation}

Then, as a consequence of (\ref{3cs2}), (\ref{3cs3}) and (\ref{3cs4}) we obtain 
\begin{equation}
\label{3cs5}
\partial_{3,0}(\partial_{2,1}(\sigma))\ast \partial_{3,2}(\partial_{2,2}(\sigma))=
\partial_{3,2}(\omega_{\sigma})\ast \partial_{3,4}(\partial_{2,2}(\sigma))\ast \partial_{3,2}(\partial_{2,1}(\sigma)).
\end{equation}

On the other hand, by (\ref{G3V}), we have 
\begin{equation}
\label{3cs6}
\partial_{3,0}(\partial_{2,3}(\sigma))= \partial_{3,4}(\partial_{2,0}(\sigma)),
\end{equation}
and, trivially, 
\begin{equation}
\label{3cs8}
\partial_{3,3}(\partial_{2,0}(\sigma))= \partial_{3,0}(\partial_{2,2}(\sigma))
\end{equation}
holds.

Also, 
\begin{equation}
\label{3cs10}
\partial_{3,3}(\partial_{2,1}(\sigma^{-1}))\ast \partial_{3,3}(\partial_{2,3}(\sigma^{-1}))= \partial_{3,1}(\partial_{2,2}(\sigma^{-1}))\ast \partial_{3,4}(\partial_{2,3}(\sigma^{-1}))
\end{equation}
holds, because 
\begin{itemize}
\item[ ]$\hspace{0.38cm}\partial_{3,3}(\partial_{2,1}(\sigma^{-1}))\ast \partial_{3,3}(\partial_{2,3}(\sigma^{-1})) $
\item[ ]$=\partial_{3,3}(\partial_{2,1}(\sigma^{-1})\ast \partial_{2,3}(\sigma^{-1})) $ {\scriptsize ({\blue by the condition of group morphism for $\partial_{3,3}$})}
\item[ ]$=\mu_{B}\co (\sigma^{-1}\ot B)\co (H\ot c_{B,H})\co  (G_{\sigma^{-1}}\ot \mu_{H}) $ {\scriptsize ({\blue by (\ref{G31V}), the counit properties and the naturality of $c$})}
\item[ ]$=\partial_{3,1}(\partial_{2,2}(\sigma^{-1}))\ast \partial_{3,4}(\partial_{2,3}(\sigma^{-1}))  $ {\scriptsize ({\blue by (a1) of Definition \ref{wha}, naturality of $c$ and (\ref{empi})}).}
\end{itemize}

Then, as a consequence of (\ref{3cs10}), we have the identity
\begin{equation}
\label{3cs11}
\partial_{3,3}(\partial_{2,3}(\sigma))\ast \partial_{3,3}(\partial_{2,1}(\sigma))= \partial_{3,4}(\partial_{2,3}(\sigma))\ast \partial_{3,1}(\partial_{2,2}(\sigma))
\end{equation}
and using that $\partial_{3,3}$ is a group morphism, (\ref{3cs8}) and (\ref{3cs31}) we can assure that 
\begin{equation}
\label{3cs11i}
\partial_{3,0}(\partial_{2,2}(\sigma))\ast \partial_{3,2}(\partial_{2,2}(\sigma))=
\partial_{3,3}(\omega_{\sigma})\ast \partial_{3,4}(\partial_{2,3}(\sigma))\ast \partial_{3,1}(\partial_{2,2}(\sigma))
\end{equation}
holds.

Moreover, 
\begin{equation}
\label{3cs12}
\partial_{3,0}(\partial_{2,0}(\sigma))\ast \partial_{3,4}(\partial_{2,3}(\sigma))= \partial_{3,4}(\partial_{2,3}(\sigma))\ast \partial_{3,1}(\partial_{2,0}(\sigma))
\end{equation}
holds because 
\begin{itemize}
\item[ ]$\hspace{0.38cm}\partial_{3,0}(\partial_{2,0}(\sigma))\ast \partial_{3,4}(\partial_{2,3}(\sigma))$
\item[ ]$=\mu_{B}\co ((\varphi_{B}\co (H\ot \varphi_{B}))\ot \sigma)\co (H\ot H\ot c_{H,B}\ot H)\co 
 (H\ot H\ot H\ot c_{H,B})\co (\delta_{H^{\ot 2}}\ot ((B\ot \varepsilon_{H})\co F_{\sigma}))$ 
 \item[ ]$\hspace{0.38cm}${\scriptsize ({\blue by (\ref{empi}) and the naturality of $c$})}
\item[ ]$=\mu_{B}\co (B\ot \sigma)\ot (P_{\varphi_{B}}\ot H)\co (H\ot (P_{\varphi_{B}}\co (H\ot \sigma)))$ {\scriptsize ({\blue by the naturality of $c$ and (\ref{sigmaHB-3})})}
\item[ ]$=\mu_{B}\co (B\ot \varphi_{B})\co (F_{\sigma}\ot \sigma)$ {\scriptsize ({\blue by (\ref{t-sigma})})}
\item[ ]$=\mu_{B}\co (B\ot \varphi_{B})\co (F_{\sigma}\ot ((\varepsilon_{H}\ot B)\co G_{\sigma}))$ {\scriptsize ({\blue by (\ref{sigmaHB-31})})}
\item[ ]$=\partial_{3,4}(\partial_{2,3}(\sigma))\ast \partial_{3,1}(\partial_{2,0}(\sigma))$ {\scriptsize ({\blue by (\ref{G3V}), the naturality of $c$ and (\ref{empi})})}
\end{itemize}
and by (\ref{G3V}) and the associativity of $\mu_{H}$ we obtain the equalities 
\begin{equation}
\label{3cs121}
\partial_{3,1}(\partial_{2,3}(\sigma))= \partial_{3,4}(\partial_{2,1}(\sigma)), 
\end{equation}
\begin{equation}
\label{3cs122}
\partial_{3,1}(\partial_{2,1}(\sigma))= \partial_{3,2}(\partial_{2,1}(\sigma)).
\end{equation}

Finally, observe that, as $\omega_{\sigma}$ factors through the center of $B$, for all $i\in \{1,2,3,4\}$ and $\tau\in Reg_{\varphi_{B}}(H^{\ot 4},B)$, we have 
\begin{equation}
\label{com3cs}
\tau\ast \partial_{3,i}(\omega_{\sigma})=\partial_{3,i}(\omega_{\sigma})\ast \tau.
\end{equation}

Therefore, we conclude the proof by cancellation because in one hand
\begin{itemize}
\item[ ]$\hspace{0.38cm}  \partial_{3,0}(\partial_{2,0}(\sigma)\ast \partial_{2,2}(\sigma))\ast \partial_{3,2}(\partial_{2,2}(\sigma))$
\item[ ]$=\partial_{3,0}(\partial_{2,0}(\sigma))\ast \partial_{3,0}(\partial_{2,2}(\sigma))\ast \partial_{3,2}(\partial_{2,2}(\sigma)) $ {\scriptsize ({\blue by the condition of group morphism for $\partial_{3,0}$})}
\item[ ]$=\partial_{3,0}(\omega_{\sigma})\ast \partial_{3,0}(\partial_{2,3}(\sigma))\ast \partial_{3,0}(\partial_{2,1}(\sigma))\ast  \partial_{3,2}(\partial_{2,2}(\sigma))${\scriptsize ({\blue by the condition of group morphism for $\partial_{3,0}$ and}}
\item[ ]$\hspace{0.38cm}${\scriptsize {\blue (\ref{3cs1})})}
\item[ ]$=\partial_{3,0}(\omega_{\sigma})\ast \partial_{3,0}(\partial_{2,3}(\sigma))\ast \partial_{3,2}(\omega_{\sigma})\ast \partial_{3,4}(\partial_{2,2}(\sigma))\ast \partial_{3,2}(\partial_{2,1}(\sigma)) $ {\scriptsize ({\blue by (\ref{3cs5})})}
\item[ ]$=\partial_{3,0}(\omega_{\sigma})\ast \partial_{3,2}(\omega_{\sigma})\ast \partial_{3,0}(\partial_{2,3}(\sigma))\ast \partial_{3,4}(\partial_{2,2}(\sigma))\ast \partial_{3,2}(\partial_{2,1}(\sigma))  $ {\scriptsize ({\blue by (\ref{com3cs})})}
\item[ ]$=\partial_{3,0}(\omega_{\sigma})\ast \partial_{3,2}(\omega_{\sigma})\ast \partial_{3,4}(\partial_{2,0}(\sigma))\ast \partial_{3,4}(\partial_{2,2}(\sigma))\ast \partial_{3,2}(\partial_{2,1}(\sigma))  $ {\scriptsize ({\blue by (\ref{3cs6})})}
\item[ ]$=\partial_{3,0}(\omega_{\sigma})\ast \partial_{3,2}(\omega_{\sigma})\ast \partial_{3,4}(\omega_{\sigma})\ast \partial_{3,4}(\partial_{2,3}(\sigma))\ast \partial_{3,4}(\partial_{2,1}(\sigma))\ast \partial_{3,2}(\partial_{2,1}(\sigma))$ {\scriptsize ({\blue by the condition of group}}
\item[ ]$\hspace{0.38cm}${\scriptsize {\blue morphism for $\partial_{3,4}$ and (\ref{3cs1})}),}
\end{itemize}
and on the other hand
\begin{itemize}
\item[ ]$\hspace{0.38cm} \partial_{3,0}(\partial_{2,0}(\sigma)\ast \partial_{2,2}(\sigma))\ast \partial_{3,2}(\partial_{2,2}(\sigma)) $
\item[ ]$=\partial_{3,0}(\partial_{2,0}(\sigma))\ast \partial_{3,0}(\partial_{2,2}(\sigma))\ast \partial_{3,2}(\partial_{2,2}(\sigma)) $ {\scriptsize ({\blue by the condition of group morphism for $\partial_{3,0}$})}
\item[ ]$= \partial_{3,0}(\partial_{2,0}(\sigma))\ast\partial_{3,3}(\omega_{\sigma})\ast \partial_{3,4}(\partial_{2,3}(\sigma))\ast \partial_{3,1}(\partial_{2,2}(\sigma))$ {\scriptsize ({\blue by (\ref{3cs11i})})}
\item[ ]$=\partial_{3,3}(\omega_{\sigma})\ast\partial_{3,0}(\partial_{2,0}(\sigma))\ast \partial_{3,4}(\partial_{2,3}(\sigma))\ast \partial_{3,1}(\partial_{2,2}(\sigma)) $ {\scriptsize ({\blue by (\ref{com3cs})})}
\item[ ]$= \partial_{3,3}(\omega_{\sigma})\ast\partial_{3,4}(\partial_{2,3}(\sigma))\ast \partial_{3,1}(\partial_{2,0}(\sigma))\ast \partial_{3,1}(\partial_{2,2}(\sigma))$ {\scriptsize ({\blue by (\ref{3cs12})})}
\item[ ]$=\partial_{3,3}(\omega_{\sigma})\ast\partial_{3,4}(\partial_{2,3}(\sigma))\ast  \partial_{3,1}(\omega_{\sigma})\ast \partial_{3,1}(\partial_{2,3}(\sigma))\ast \partial_{3,1}(\partial_{2,1}(\sigma))$ {\scriptsize ({\blue by the condition of group morphism}}
\item[ ]$\hspace{0.38cm}${\scriptsize {\blue for $\partial_{3,1}$ and (\ref{3cs1})})}
\item[ ]$=\partial_{3,1}(\omega_{\sigma})\ast\partial_{3,3}(\omega_{\sigma})\ast  \partial_{3,4}(\partial_{2,3}(\sigma))\ast  \partial_{3,1}(\partial_{2,3}(\sigma))\ast \partial_{3,1}(\partial_{2,1}(\sigma)) $ {\scriptsize ({\blue by (\ref{com3cs})}).}
\item[ ]$=\partial_{3,1}(\omega_{\sigma})\ast\partial_{3,3}(\omega_{\sigma})\ast  \partial_{3,4}(\partial_{2,3}(\sigma))\ast  \partial_{3,4}(\partial_{2,1}(\sigma))\ast \partial_{3,2}(\partial_{2,1}(\sigma)) $ {\scriptsize ({\blue by (\ref{3cs121}) and (\ref{3cs122})}).}
\end{itemize}
\end{proof}

\begin{teo}
\label{3cocy}
Let  $H$ be a cocommutative weak Hopf algebra, let $(B,\varphi_{B})$ be a left weak $H$-module algebra and let $\sigma \in Reg_{\varphi_{B}}(H^{\ot 2},B)$ satisfying the twisted condition (\ref{t-sigma}). Let $\theta_{\sigma}$ the obstruction of $\sigma$. Then, $\theta_{\sigma}\in Im(D^{2}_{\varphi_{Z(B)}})$ if, and only if, there exists $\alpha\in Reg_{\varphi_{B}}(H^{\ot 2}, B)$ that satisfies the twisted condition (\ref{t-sigma}) and the cocycle condition (\ref{c-sigma}).
\end{teo}

\begin{proof} If $\theta_{\sigma}\in Im(D^{2}_{\varphi_{Z(B)}})$, there exists $\tau\in Reg_{\varphi_{Z(B)}}(H^{\ot 2}, Z(B))$ such that $D^{2}_{\varphi_{Z(B)}}(\tau)=\theta_{\sigma}$. Then, 
\begin{equation}
\label{hh1}
z_{B}\co D^{2}_{\varphi_{Z(B)}}(\tau)=\omega_{\sigma}.
\end{equation}

By Proposition \ref{talsig}, the morphism $\alpha=(z_{B}\co \tau^{-1})\ast \sigma$ satisfies the twisted condition (\ref{t-sigma}) and belongs to $Reg_{\varphi_{B}}(H^{\ot 2}, B)$. On the other hand, 
\begin{itemize}
\item[ ]$\hspace{0.38cm} \omega_{\alpha}$
\item[ ]$=D^{2}_{\varphi_{B}}(z_{B}\co \tau^{-1})\ast \omega_{\sigma}$ {\scriptsize ({\blue by the properties of $\partial_{i,j}$ and $\tau$})}
\item[ ]$=z_{B}\co D^{2}_{\varphi_{Z(B)}}(\tau^{-1})\ast \omega_{\sigma}$ {\scriptsize ({\blue by the properties of $D^{2}_{\varphi_{B}}$})}
\item[ ]$=\omega_{\sigma}^{-1}\ast \omega_{\sigma}$ {\scriptsize ({\blue by (\ref{hh1})})}
\item[ ]$=u_{3}^{\varphi_{B}}$ {\scriptsize ({\blue by (\ref{s1}) for $\omega_{\sigma}$})}
\end{itemize}
and, as a consequence, $\alpha$ satisfies the cocycle condition (\ref{c-sigma}).

Conversely, assume that there exists $\alpha\in Reg_{\varphi_{B}}(H^{\ot 2}, B)$ that satisfies the twisted condition (\ref{t-sigma}) and the cocycle condition (\ref{c-sigma}). Then, by Proposition \ref{talsig}, there there exists $\tau\in Reg_{\varphi_{Z(B)}}(H^{\ot 2}, Z(B))$ such that (\ref{talsig-eq}) holds, i.e., $\alpha=(z_{B}\co \tau)\ast \sigma$. As a consequence, $\sigma=(z_{B}\co \tau)^{-1}\ast \alpha$ and $\theta_{\sigma}\in Im(D^{2}_{\varphi_{Z(B)}})$ since
$$\omega_{\sigma}=D^{2}_{\varphi_{B}}(z_{B}\co \tau^{-1})\ast \partial_{2,0}(\alpha)\ast \partial_{2,2}(\alpha)\ast \partial_{2,1}(\alpha^{-1})\ast \partial_{2,3}(\alpha^{-1})\stackrel{{\scriptsize \blue  (\ref{deric-coc})}}{=}D^{2}_{\varphi_{B}}(z_{B}\co \tau^{-1})=z_{B}\co D^{2}_{\varphi_{Z(B)}}(\tau^{-1}).$$

\end{proof}

\begin{prop}
\label{cor-talsig} 
Let $H$ be a cocommutative weak Hopf algebra and let $(B,\varphi_{B})$ be a left weak $H$-module algebra. Let $\sigma\in Reg_{\varphi_{B}}^{+}(H^{\ot 2},B)$ satisfying the twisted condition (\ref{t-sigma}). Then, $\alpha\in Reg_{\varphi_{B}}^{+}(H^{\ot 2},B)$ satisfies the twisted condition (\ref{t-sigma}) if, and only if, there exists $\tau\in Reg_{\varphi_{Z(B)}}^{+}(H^{\ot 2},Z(B))$ satisfying  (\ref{talsig-eq}).
\end{prop}

\begin{proof}
First note that, if $H$ is cocommutative, $(D, \varphi_{D})$ is a left weak $H$-module algebra and $\beta\in Reg_{\varphi_{D}}(H^{\ot 2}, D)$, using that $\overline{\Pi}_{H}^{L}=\Pi_{H}^{L}$ and (\ref{sigma-eta-1}) we obtain that 
\begin{equation}
\label{cor-a}
\beta\co (\eta_{H}\ot H)=\beta\co (\Pi_{H}^{L}\ot H)\co \delta_{H}
\end{equation}
holds. Also,  (\ref{sigma-eta-2}), holds for $\beta$  and therefore $\beta$ satisfies the normal condition (\ref{normal-sigma}), i.e., $\beta\in Reg_{\varphi_{D}}^{+}(H^{\ot 2}, D)$ if and only if 
\begin{equation}
\label{cor-b}
\beta\co (\Pi_{H}^{L}\ot H)\co \delta_{H}=\beta\co (H\ot \Pi_{H}^{R})\co \delta_{H}=u_{1}^{\varphi_{D}}.
\end{equation}

Let $\alpha\in Reg_{\varphi_{B}}^{+}(H^{\ot 2},B)$ satisfying the twisted condition (\ref{t-sigma}). By Proposition \ref{talsig} there exists $\tau\in Reg_{\varphi_{Z(B)}}(H^{\ot 2},Z(B))$ satisfying  (\ref{talsig-eq}). Then, $z_{B}\co \tau=\alpha\ast \sigma^{-1}$  and $\tau $ satisfies the normal condition (\ref{normal-sigma}) because, in one hand, 
\begin{itemize}
\item[ ]$\hspace{0.38cm} z_{B}\co \tau\co (\Pi_{H}^{L}\ot H)\co \delta_{H}$
\item[ ]$=\mu_{B}\co ((\alpha\co  (\Pi_{H}^{L}\ot H))\ot (\sigma{-1}\co  (\Pi_{H}^{L}\ot H)))\co \delta_{H^{\ot 2}}\co \delta_{H}$ {\scriptsize ({\blue by the naturality of $c$ and (i) of}}
\item[ ]$\hspace{0.38cm}${\scriptsize {\blue \cite[Proposition 2.6]{nmra5}})}
\item[ ]$=(\alpha\co  (\Pi_{H}^{L}\ot H)\co \delta_{H})\ast (\sigma^{-1}\co  (\Pi_{H}^{L}\ot H)\co \delta_{H})$ {\scriptsize ({\blue by the cocommutativity of $\delta_{H}$})}
\item[ ]$=u_{2}^{\varphi_{B}}\ast u_{2}^{\varphi_{B}}$ {\scriptsize ({\blue by (\ref{cor-b})})}
\item[ ]$=z_{B}\co u_{2}^{\varphi_{Z(B)}}$ {\scriptsize ({\blue by (\ref{udd})})}
\end{itemize}
and, on the other hand, using the same arguments we have $z_{B}\co \tau\co (H\ot \Pi_{H}^{R})\co \delta_{H}=z_{B}\co u_{2}^{\varphi_{Z(B)}}$.

Conversely,  if there exists $\tau\in Reg_{\varphi_{Z(B)}}^{+}(H^{\ot 2},Z(B))$ satisfying  (\ref{talsig-eq}), by the previous arguments,  we obtain that 
$$\alpha\co (\Pi_{H}^{L}\ot H)\co \delta_{H}=(z_{B}\co \tau\co  (\Pi_{H}^{L}\ot H)\co \delta_{H})\ast (\sigma^{-1}\co  (\Pi_{H}^{L}\ot H)\co \delta_{H})=(z_{B}\co u_{2}^{\varphi_{Z(B)}})\ast u_{2}^{\varphi_{B}}=u_{2}^{\varphi_{B}}$$
and similarly $\alpha\co (H\ot \Pi_{H}^{R})\co \delta_{H}=u_{2}^{\varphi_{B}}$. Therefore, $\alpha\in Reg_{\varphi_{B}}^{+}(H^{\ot 2},B)$ .
\end{proof}

\begin{rem}
	\label{cohomo}
	{\rm Let  $H$ be a cocommutative weak Hopf algebra, let $(B,\varphi_{B})$ be a left weak $H$-module algebra and let $\sigma, \beta\in Reg_{\varphi_{B}}^{+}(H^{\ot 2},B)$ satisfying the twisted condition (\ref{t-sigma}). Let $\theta_{\sigma}$, $\theta_{\beta}$ the corresponding obstructions of $\sigma$ and $\beta$. Then, by the previous proposition, it is easy to show that $[\theta_{\sigma}]=[\theta_{\beta}]$ in ${\mathcal H}_{\varphi_{Z(B)}}^{3+} (H,Z(B))$, i.e., $\theta_{\sigma}$ and $\theta_{\beta}$ are cohomologous.
	}
\end{rem}

\begin{cor}
\label{cor-3cocy}
Let  $H$ be a cocommutative weak Hopf algebra, let $(B,\varphi_{B})$ be a left weak $H$-module algebra and let $\sigma \in Reg_{\varphi_{B}}^{+}(H^{\ot 2},B)$ satisfying the twisted condition (\ref{t-sigma}). Let $\theta_{\sigma}$ the obstruction of $\sigma$. Then, $\theta_{\sigma}\in Im(D^{2+}_{\varphi_{Z(B)}})$ if, and only if, there exists $\alpha\in Reg_{\varphi_{B}}^{+}(H^{\ot 2}, B)$ that satisfies the twisted condition (\ref{t-sigma}) and the cocycle condition (\ref{c-sigma}).
\end{cor}

\begin{proof} The result is a direct consequence of Theorem \ref{3cocy} and Proposition \ref{cor-talsig}.
\end{proof}

\begin{cor}
\label{1cor-3cocy}
Let  $H$ be a cocommutative weak Hopf algebra, let $(B,\varphi_{B})$ be a left weak $H$-module algebra and let $\sigma \in Reg_{\varphi_{B}}^{+}(H^{\ot 2},B)$ satisfying the twisted condition (\ref{t-sigma}). Let $\theta_{\sigma}$ the obstruction of $\sigma$. Then, $[\theta_{\sigma}]=0$ in ${\mathcal H}_{\varphi_{Z(B)}}^{3+} (H,Z(B))$ if and only if there exists a morphism $\alpha\in Reg_{\varphi_{B}}^{+}(H^{\ot 2},B)$ 
that satisfies the twisted condition (\ref{t-sigma}),  the cocycle condition (\ref{c-sigma}) and the normal condition (\ref{normal-sigma}).
\end{cor}

\begin{proof}
The proof follows by the previous corollary and Corollary \ref{crossed-product1} 
\end{proof}

As a consequence of this  corollary  we can assure that the obstruction vanishes if and only if there exists a weak crossed product with preunit $\nabla_{B\ot H}^{\varphi_{B}}\co (\eta_{B}\ot \eta_{H})$ and normalized with respect to $\nabla_{B\ot H}^{\varphi_{B}}$. Equivalently, by \cite[Theorem 6.17, Corollary 6.18]{Gu2-Val}, this is equivalent to say that $B$ admits a $H$-cleft extension (see also 
\cite[Proposition 3.5]{nmra6}.

\section*{Funding}
The  authors were supported by  Ministerio de Ciencia e Innovaci\'on of Spain. Agencia Estatal de Investigación. Uni\'on Europea - Fondo Europeo de Desarrollo Regional. Grant: Homolog\'{\i}a, homotop\'{\i}a e invariantes categ\'oricos en grupos y \'algebras no asociativas.

\end{document}